\documentclass[a4paper,12pt]{article}

\usepackage{latexsym}
\usepackage{amsmath}
\usepackage{amsthm}
\usepackage{amsbsy}
\usepackage{amssymb}
\usepackage[pdftex]{graphicx}
\usepackage{epstopdf}
\usepackage{mathrsfs}
\usepackage{times}
\usepackage{afterpage}
\usepackage[prefix]{nomencl}
\usepackage{psfrag}
\usepackage{multido}
\usepackage{algorithm} 
\usepackage{algpseudocode}
\usepackage{mathabx}

\usepackage{amsfonts}
\usepackage{arydshln}
\usepackage{booktabs}	
\usepackage{url}
\usepackage{hyperref} 

\input epsf

\renewcommand{\vec}[1]{\boldsymbol{#1}}
\newcommand{\dif}{\mathrm{d}}

\usepackage{color}

\usepackage{amsthm}
\newtheorem{theorem}{Theorem}[section]
\newtheorem{remark}[theorem]{Remark}

\newtheorem{lemma}[theorem]{Lemma}
\newtheorem{corollary}[theorem]{Corollary}
\newtheorem{proposition}[theorem]{Proposition}
\allowdisplaybreaks[3]

\newcommand{\im}{\mathrm{i}}

\begin{document}
\begin{center}
{\Large \bf 
Characterisation of Multiple Conducting Permeable Objects in Metal Detection by Polarizability Tensors}
\\
P.D. Ledger$^*$, W.R.B. Lionheart$^\dagger$ and A.A.S. Amad$^*$ \\
$^*$Zienkiewicz Centre for Computational Engineering, College of Engineering, \\Swansea University Bay Campus, Swansea. SA1 8EN\\
$^\dagger$School of Mathematics, Alan Turing Building, \\The University of Manchester, Oxford Road, Manchester, M13 9PL\\
20th September 2018
\end{center}

\section*{Abstract}

Realistic applications in metal detection involve multiple inhomogeneous  conducting permeable objects and the aim of this paper is to characterise such objects by polarizability tensors. We show
 that, for the eddy current model, the leading order terms for the perturbation in the magnetic field, due to the presence of $N$ small conducting permeable homogeneous inclusions, comprises of a sum of $N$ terms with each containing a complex symmetric rank 2 polarizability tensor. Each tensor contains information about the shape and material properties of one of the objects and is independent of its position. The asymptotic expansion we obtain extends a previously known result for a single isolated object and applies in situations where the object sizes are small and the objects are sufficiently well separated. We also obtain a second expansion that describes the perturbed magnetic field for inhomogeneous and closely spaced objects, which again characterises the objects by a complex symmetric rank 2 tensor. The tensor's coefficients can be computed by solving a vector valued transmission problem and we include numerical examples to illustrate the agreement between the asymptotic formula describing the perturbed fields and the numerical prediction.  We  also include algorithms for the localisation and identification of multiple inhomogeneous objects.

\vspace{0.1in}

\noindent {\em MSC: 35R30, 35B30}

\vspace{0.1in}

\noindent {\bf Keywords:} Polarizability Tensors; Asymptotic Expansion; Eddy Currents; Metal Detectors; Land Mine Detection

\section{Introduction}

There is considerable interest in being able to locate and characterise  multiple conducting permeable objects from measurements of mutual inductance between a transmitting and a receiving coil, where the coupling is inductive. An obvious example is in metal detection where the goal is to identify and locate the multiple objects present in a low conducting background. Applications include security screening, archaeological digs, ensuring food safety as well as the search for landmines and unexploded ordnance and landmines.
Other applications include magnetic induction tomography for medical imaging applications and monitoring of corrosion of steel reinforcement in concrete structures.

In all these practical applications, the need to locate and distinguish between multiple conducting permeable inclusions is common. This includes benign situations, such as coins and keys accidentally left in a pocket during a security search or a treasure hunter becoming lucky and discovering a hoard of Roman coins, as well as threat situations, where the risks need to be clearly identified from the background clutter. For example, in the case of searching for unexploded landmines, the ground can be contaminated by ring-pulls, coins and other metallic shrapnel, which makes the process of clearing them  very slow as each metallic object needs to be dug up with care.
Furthermore, conducting objects are also often inhomogeneous and made up of several different metals. For instance, the barrel of a gun is invariably steel while the frame could be a lighter alloy, jacketed bullets have a lead shot and a brass jacket and modern coins often consist of a cheaper metal encased in nickel or brass alloy. Thus, in practical metal detection applications, it is important to be able to describe both multiple objects and inhomogeneous objects.

Magnetic polarizability tensors (MPTs) hold considerable promise for the low-cost characterisation in metal detection. An asymptotic expansion describing the perturbed magnetic field due to the presence of a small conducting permeable object has been obtained by Ammari, Chen, Chen,  Garnier and Volkov~\cite{ammarivolkov2013}, which characterises the object in terms of a rank 4 tensor. Ledger and Lionheart have shown that this asymptotic expansion simplifies for orthonormal coordinates and allows a conducting permeable object to be characterised by a complex symmetric rank 2 MPT with an explicit expression for its 6 coefficients~\cite{ledgerlionheart2014}. Ledger and Lionheart have also investigated the properties of this tensor ~\cite{ledgerlionheart2016} and they have written the article~\cite{ledgerlionheart2018}  to explain these developments to the electrical engineering community as well as to show how it applies in several realistic situations. In~\cite{ledgerlionheart2017} they have obtained a complete asymptotic expansion of the magnetic field, which characterises the object in terms of a new class of generalised magnetic polarizability tensors (GMPTs), the rank 2 MPT being the simplest case. The availability of an explicit formula for the MPT's coefficients, and its improved understanding, allows new algorithms for object location and identification to be designed e.g.~\cite{ammarivolkov2013b}.

Electrical engineers have applied MPTs to a range of practical metal detection applications, including walk through metal detectors, in line scanners and demining e.g.~\cite{barrowes2008,shubitidze2004,shubitdze2007,baumbook, marsh2015,dekdouk,zhao2016}, see  also our article~\cite{ledgerlionheart2018} for a recent review, but without knowledge of the explicit formula described above.  Engineers have made a prediction of the form of the response for multiple objects e.g.~\cite{braunisch2002}, but without an explicit criteria on the size or the distance between the objects in order for the approximation to hold. Grzegorczyk,  Barrowes, Shubitidze,  Fern\'andez and O'Neill have applied a time domain approach to classify multiple unexploded ordinance using descriptions related to MPTs~\cite{Gregorczy}. Davidson, Abel-Rehim, Hu, Marsh, O'Toole and Peyton have made measurements of MPTs for inhomogeneous US coins~\cite{davidson2018} and Yin, Li, Withers and Peyton have also made measurements to characterise inhomogeneous aluminium/carbon-fibre reinforced plastic sheets~\cite{yin2010}. But, in all cases, without an explicit formula.

Our work has the following novelties: Firstly, we characterise rigidly joined collections of different metals (i.e. metals touching or held in that configuration by a non-conducting material)  by MPTs  overcoming a deficiency of our previous work. Secondly, we find that the frequency spectra of the eigenvalues of the real and imaginary parts of the MPT for an inhomogeneous object exhibit multiple non-stationary inflection points and maxima, respectively, and the number of these gives an upper bound on the number of materials making up the object. To achieve this, we revisit the asymptotic formula of Ammari {\em et al.}~\cite{ammarivolkov2013} and our previous work~\cite{ledgerlionheart2014} and extend it to treat multiple objects by describing the perturbed magnetic field as a sum of terms involving the MPTs associated with each of the inclusions. We also provide a criteria based on the distance between the objects, which determines the situations in which the expression will hold. We derive a second asymptotic expansion that describes the perturbed magnetic field in the case of inhomogeneous objects and, as a corollary, this also describes the magnetic field perturbation in the case of closely spaced objects. In each case, we provide new explicit formulae for the MPTs. We also present algorithms for the localisation and characterisation of objects, which extends those for the isolated object case~\cite{ammarivolkov2013}.

The paper is organised as follows: In Section~\ref{sect:back}, the characterisation of a single homogeneous object is briefly reviewed. Section~\ref{sect:main} presents our main results for characterising multiple and inhomogeneous objects by MPTs. Sections~\ref{sect:proof1} and~\ref{sect:proof2} contain the details of the proofs for our main results. In Section~\ref{sect:num}, we present numerical results to demonstrate the accuracy of the asymptotic formulae and presents results of algorithms for the localisation and identification of multiple (inhomogeneous) objects.

\section{Characterisation of a Single Conducting Permeable Object} \label{sect:back}
We begin by recalling known results for the characterisation of a single homogenous conducting permeable object.
Following~\cite{ammarivolkov2013,ledgerlionheart2014} we describe a single inclusion by $B_\alpha := \alpha B + {\vec z}$, which means that it can be thought of a unit-sized object $B$ located at the origin, scaled by $\alpha$ and translated by ${\vec z}$. We assume the background is non-conducting and non-permeable and introduce the position dependent conductivity and permeability as
\begin{align}
\sigma_\alpha= \left \{ \begin{array}{ll} \sigma_* & \text{in $B_\alpha$} \\
                                                             0 & \text{in $B_\alpha^c$}
                                                             \end{array},  \right . \qquad
\mu_\alpha= \left \{ \begin{array}{ll} \mu_* & \text{in $B_\alpha$} \\
                                                             \mu_0 & \text{in $B_\alpha^c$}
                                                             \end{array} \right .  , \nonumber
\end{align}
where $\mu_0$ is the permeability of free space and $B_\alpha^c := {\mathbb R}^3 \setminus\overline{ B_\alpha}$. For metal detection, the relevant mathematical model is the eddy current approximation of Maxwell's equations since $\sigma_*$ is large and the angular frequency $\omega=2 \pi f$ is small (see Ammari, Buffa and N\'ed\'elec~\cite{ammaribuffa2000} for a detailed justification). Here 
the electric and magnetic interaction fields, ${\vec E}_\alpha$ and ${\vec H}_\alpha$, respectively, satisfy the curl equations
\begin{equation}
\nabla \times {\vec H}_\alpha = \sigma_\alpha {\vec E}_\alpha +{\vec J}_0, \qquad 
\nabla \times {\vec E}_\alpha = \im \omega \mu_\alpha {\vec H}_\alpha, \label{eqn:eddyqns}
\end{equation}
in ${\mathbb R}^3$ and decay as $O(|{\vec x}|^{-1})$ for $|{\vec x}| \to \infty$. In the above ${\vec J}_0$ is an external current source with support in $B_\alpha^c$. In absence of an object, the background fields ${\vec E}_0$ and ${\vec H}_0$ satisfy (\ref{eqn:eddyqns}) with $\alpha=0$.

The task is to find an economical description for the perturbed magnetic  field $({\vec H}_\alpha - {\vec H}_0)({\vec x})$ due to the presence of $B_\alpha$, which characterises the object's shape and material parameters by a small number of parameters separately to its location ${\vec z}$.  For ${\vec x}$ away from $B_\alpha$, the leading order term in an asymptotic expansion for  $({\vec H}_\alpha - {\vec H}_0)({\vec x})$ as $\alpha \to 0$ has been derived by Ammari {\it et al.}~\cite{ammarivolkov2013}. We have shown that this reduces to the simpler form~\cite{ledgerlionheart2014,ledgerlionheart2018}~\footnote{In order to simplify notation, we drop the double check on ${\mathcal M}$ and the single check on ${\mathcal C}$, which was used in~\cite{ledgerlionheart2014} to denote two and one reduction(s) in rank, respectively. We recall that ${\mathcal M}= ({\mathcal M})_{jk} {\vec e}_j \otimes {\vec e}_k$ by the Einstein summation convention where we use the notation ${\vec e}_j$ to denote the $j$th unit vector. We will denote the $j$th component of a vector ${\vec u}$ by ${\vec u} \cdot {\vec e}_j = ({\vec u})_j$ and the $j,k$th coefficient of a rank 2 tensor ${\mathcal M}$ by ${\mathcal M}_{jk}$.} 
\begin{align}
({\vec H}_\alpha - {\vec H}_0)({\vec x})_i =&( {\vec D}_{x}^2 G({\vec x},{\vec z}))_{ij}  ( {\mathcal M}[\alpha B])_{jk} ({\vec H}_0({\vec z}) )_k +( {\vec R}({\vec x}))_i   \nonumber \\
= &  \frac{1}{4 \pi r^3} \left ( 3\hat{\vec r} \otimes \hat{\vec r} - {\mathbb I} \right)_{ij}  ( {\mathcal M} [\alpha B])_{jk} ({\vec H}_0({\vec z}) )_k + ({\vec R}({\vec x}))_i    \label{eqn:asymp}.
\end{align}
In the above, $G({\vec x},{\vec z}) := 1/ ( 4 \pi |{\vec x}-{\vec z}|)$  is the free space Laplace Green's function, ${\vec r}: ={\vec x}- {\vec z}$, $r=|{\vec r}|$ and $\hat{\vec r}= {\vec r}/ r$ and ${\mathbb I}$ is the rank 2 identity tensor.
The term ${\vec R}({\vec x})$ quantifies the remainder and it is known that  $|{\vec R}| \le C \alpha^4 \| {\vec H}_0 \|_{W^{2,\infty}(B_\alpha) }$.  The result holds when $\nu  := \sigma_* \mu_0 \omega \alpha^2 = O(1)$ (this includes the case of fixed $\sigma_*$, $\mu_*$, $\omega$ as $\alpha \to 0$) and requires that the background field be analytic in the object. Note that (\ref{eqn:asymp}) involves the evaluation of the background field within the object, (usually at it's centre) i.e. ${\vec H}_0({\vec z}) $. 

The complex symmetric rank 2 tensor ${\mathcal M}[\alpha B]:= -{\mathcal C}[\alpha B] + {\mathcal N}[\alpha B]$ in this asymptotic expansion, which depends on $\omega$, $\sigma_*$, $\mu_*/\mu_0$, $\alpha$ and the shape of $B$, but is independent of ${\vec z}$, is the MPT, and its coefficients can be computed from
\begin{subequations}  \label{eqn:mcheck} 
\begin{align}
({\mathcal C} [\alpha B])_{jk} :=& -\frac{\im \nu \alpha^3 }{4}{\vec e}_j \cdot \int_B {\vec \xi} \times ({\vec \theta}_k + {\vec e}_k \times {\vec \xi} ) \dif {\vec \xi},  \\
({\mathcal N} [\alpha B])_{jk} := &  \alpha^3 \left ( 1- \frac{\mu_0}{\mu_*} \right ) \int_B \left (
{\vec e}_j \cdot {\vec e}_k+ \frac{1}{2}  {\vec e}_j \cdot  \nabla_\xi \times {\vec \theta}_k \right ) \dif {\vec \xi}.
\end{align}
\end{subequations}
These, in turn, rely on the vectoral solutions ${\vec \theta}_k$, $k=1,2,3,$ to the transmission problem
\begin{subequations}
\begin{align}
\nabla_\xi \times \mu_*^{-1} \nabla_\xi \times {\vec \theta}_k - \im \omega \sigma_* \alpha^2 {\vec \theta }_k & = \im \omega \sigma_* \alpha^2 {\vec e}_k \times {\vec \xi}  && \text{in $B$ } ,\\
\nabla_\xi \cdot {\vec \theta}_k  = 0 , \qquad \nabla_\xi \times \mu_0^{-1} \nabla_\xi \times {\vec \theta}_k  & = {\vec 0} && \text{in $B^c:= {\mathbb R}^3 \setminus \overline{B}$ }, \\
[{\vec n} \times {\vec \theta}_k ]_\Gamma = {\vec 0},  \qquad [{\vec n} \times \mu^{-1} \nabla_\xi \times {\vec \theta}_k ]_\Gamma & = -2 [\mu^{-1 } ]_\Gamma {\vec n} \times {\vec e}_k  && \text{on $\Gamma:= \partial B$},\\
{\vec \theta}_k & = O( | {\vec \xi} |^{-1}) && \text{as $|{\vec \xi} | \to \infty$ },
\end{align}\label{eqn:transproblemthetar}
\end{subequations}
where $[\cdot ]_\Gamma $ denotes the  jump of the function  over $\Gamma$ and ${\vec \xi}$ is measured from an origin chosen to be in $B$. In~\cite{ledgerlionheart2016} we have presented numerical results for the frequency behaviour of the coefficients of $ {\mathcal M}$ for a variety of simply and multiply connected objects. These have been obtained by applying a $hp$-finite element method to solve (\ref{eqn:transproblemthetar}) for ${\vec \theta}_k$, $k=1,2,3,$ and then to compute $ {\mathcal M}$ using (\ref{eqn:mcheck}). Our previously presented results have exhibited excellent agreement with for MPTs  previously presented in the electrical engineering literature. Pratical applications of the asymptotic formula have been discussed in~\cite{ledgerlionheart2018}.

\section{Main Results} \label{sect:main}

The asymptotic formula given in (\ref{eqn:asymp}) is for a single isolated homogenous object. But, as described in the introduction, for realistic metal detection scenarios, measurements of the perturbed magnetic field often relate to field changes caused by the presence of multiple objects or inhomogeneous objects. The purpose of this work is to extend the description to the cases of well separated multiple homogeneous objects and inhomogeneous objects. As a result of corollary, our second main result also describes the case of objects that of objects that are closely spaced.
Below we summarise the main results of our paper.

\subsection{Multiple Homogeneous Objects that are Sufficiently Well Spaced} \label{sect:wellspaced}

We consider $N$  homogenous conducting permeable objects of the form $(B_\alpha)^{(n)}=\alpha^{(n)} B^{(n)} +{\vec z}^{(n)}$~\footnote{Note no summation over $n$ is implied.} with Lipschitz boundaries where, for the $n$th object, $B^{(n)}$ denotes a corresponding unit sized object located at the origin, $\alpha^{(n)}$ denotes the object's size and ${\vec z}^{(n)}$ the object's translation from the origin.
The union of all objects is $ {\vec B}_{\vec \alpha} :=\bigcup_{n=1}^N (B_\alpha)^{(n)}$ where we use a bold subscript ${\vec \alpha}$ to denote the possibility that each object in the collection  can have a different size.
We also employ the same notation for the fields ${\vec E}_{\vec \alpha}$ and ${\vec H}_{\vec \alpha}$, which satisfy (\ref{eqn:eddyqns}).  An illustration of a typical configuration is shown in Figure~\ref{fig:multobjectsnotclose}. In this figure, there are $N=3$ objects, which are the spheres $(B_\alpha)^{(n)}=\alpha^{(n)} B^{(n)} +{\vec z}^{(n)}$, $n=1,2,3$, where, for the $n$th object, $\alpha^{(n)}$ is its size (here its radius) and ${\vec z}^{(n)}$ is its translation from the origin. In the presented case $B=B^{(1)}=B^{(2)}=B^{(3)}$ is a unit sphere located at the origin although, in practice, the objects do not need to have the same shape.
\begin{figure}
\begin{center}
\includegraphics[width=0.8\textwidth]{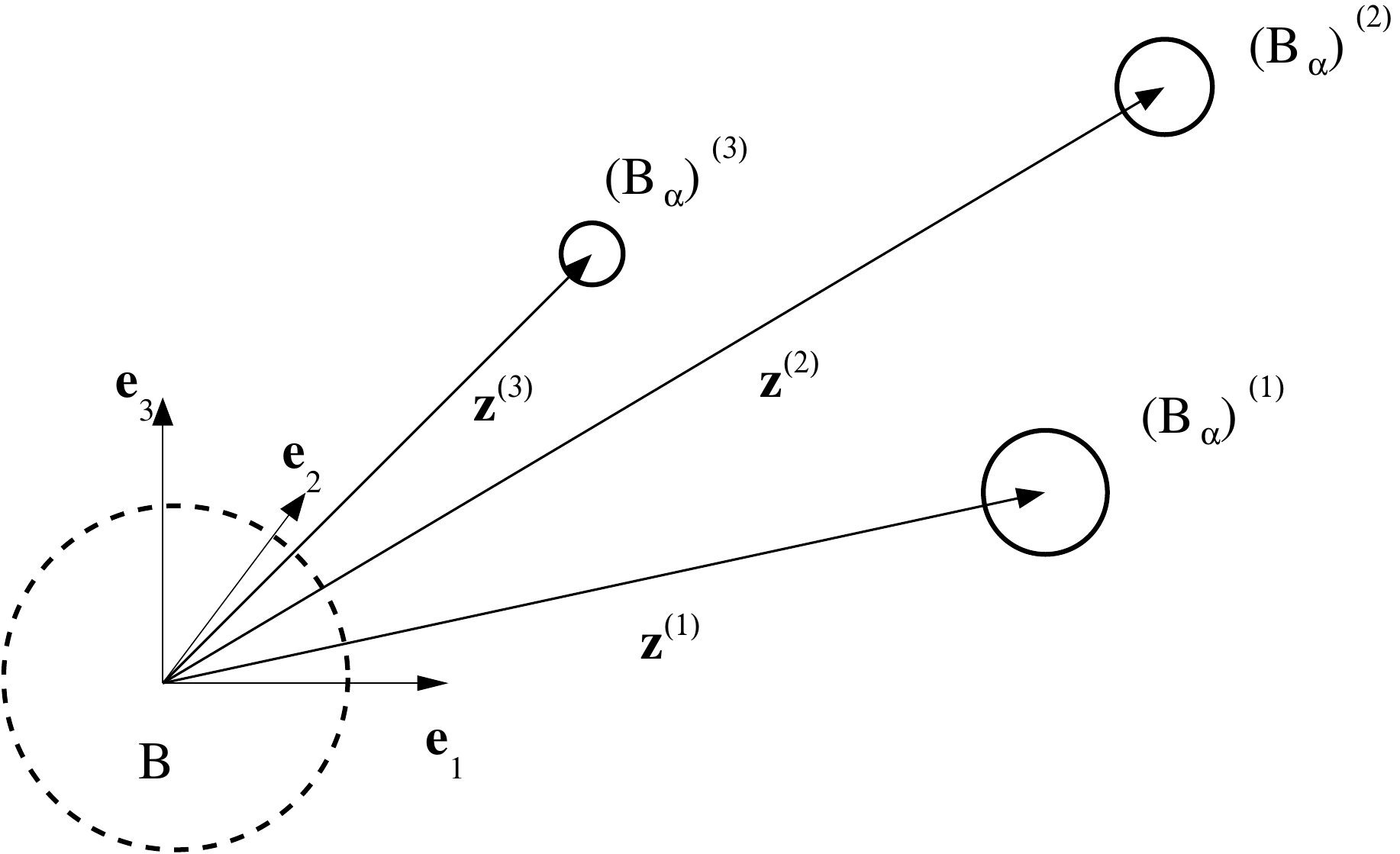}
\end{center}
\caption{Illustration of a typical situation of $N=3$ objects with ${\vec B}_{\vec \alpha}=\bigcup_{n=1}^N (B_\alpha)^{(n)} = \alpha^{(n)} B^{(n)} + {\vec z}^{(n)}$ such  that they are not closely spaced where each object  $  (B_\alpha)^{(n)}  $ is a sphere, $\alpha^{(n)}$ is the radius of the $n$th sphere, ${\vec z}^{(n)}$ describes the translation of the $n$th sphere from the origin and $B=B^{(1)}=B^{(2)}=B^{(3)}$ is a unit sphere positioned at the origin. }\label{fig:multobjectsnotclose}
\end{figure}

We generalise the definitions of $\mu_\alpha$ and $\sigma_\alpha$ previously stated in Section~\ref{sect:back} to
\begin{align}
\sigma_{\vec \alpha}= \left \{ \begin{array}{ll} \sigma_{*}^{(n)} & \text{in $(B_\alpha)^{(n)} $} \\
                                                             0 & \text{in ${\vec B}_{\vec \alpha}^c $}
                                                             \end{array},  \right . \qquad
\mu_{\vec \alpha}= \left \{ \begin{array}{ll} \mu_*^{(n)} & \text{in $(B_\alpha)^{(n)} $} \\
                                                             \mu_0 & \text{in ${\vec B}_{\vec \alpha}^c$}
                                                             \end{array} \right . \nonumber ,
\end{align}
where ${\vec B}_{\vec \alpha}^c:= {\mathbb R}^3\setminus \overline{{\vec B}_{\vec \alpha}}$
and set $\sigma_{\min} \le \sigma_*^{(n)} \le \sigma_{\max}$ and  $\mu_{\min} \le \mu_*^{(n)} \le \mu_{\max}$ for $n=1,\ldots, N$. We introduce $\nu_{\min} \le \nu^{(n)} := \omega \mu_0 \sigma_*^{(n)} (\alpha^{(n)})^2\le \nu_{\max}$ and set
 $\displaystyle \alpha_{\min}=\min_{n=1,\ldots,N} \alpha^{(n)}$, $\displaystyle \alpha_{\max}=\max_{n=1,\ldots,N} \alpha^{(n)}$
  and require that the parameters of the inclusions be such that
\begin{equation}
 \nu_{\max}  =O(1),  \nonumber
\end{equation}
which implies that $\nu^{(n)}=O(1)$.

The task is then to provide a low-cost description of  $({\vec H}_{\vec \alpha} - {\vec H}_0)({\vec x})$ for ${\vec x}$ away from ${\vec B}_{\vec \alpha}$. This is accomplished through the following result.

\begin{theorem} \label{thm:objectsnocloselyspaced}
For the arrangement  $ {\vec B}_{\vec \alpha}  $   of $N$ homogeneous conducting permeable objects $(B_\alpha)^{(n)} =  \alpha^{(n)} B^{(n)} +{\vec z}^{(n)}$ with  $\displaystyle \min_{n,m=1,\ldots,N, n\ne m}|\partial (B_\alpha)^{(n)} -\partial (B_{\alpha})^{(m)}  | \ge \alpha_{max}$ and parameters such that  $\nu^{(n)} =O(1)$ , the perturbed magnetic field at positions ${\vec x}$ away from ${\vec B}_{\vec \alpha}$ satisfies
\begin{align}
({\vec H}_{\vec \alpha} - {\vec H}_0)({\vec x})_i =&\sum_{n=1}^N ({\vec D}_{x}^2 G({\vec x},{\vec z}^{(n)}))_{ij}  ( {\mathcal M}[\alpha^{(n)} B^{(n)}])_{jk} ({\vec H}_0({\vec z}^{(n)}) )_k +( {\vec R}({\vec x}))_i ,  \label{eqn:asymp2}
\end{align}
where 
\begin{equation}
|{\vec R}({\vec x}| \le C \alpha_{\max}^4 \| {\vec H}_0 \|_{W^{2,\infty}({\vec B}_{\vec \alpha})}  , \nonumber
\end{equation}
uniformly in ${\vec x}$ in any compact set away from $ {\vec B}_{\vec \alpha}$.  The coefficients of the complex symmetric MPTs ${\mathcal M}[\alpha^{(n)} B^{(n)}] =- {\mathcal C}[\alpha^{(n)} B^{(n)}] + {\mathcal N}[\alpha^{(n)} B^{(n)}]$, $n=1,\ldots,N$,  can be computed independently for each of the objects $\alpha^{(n)}B^{(n)}$ using the expressions
\begin{subequations}  \label{eqn:mcheckmult} 
\begin{align}
( {\mathcal C} [\alpha^{(n)} B^{(n)}])_{jk} :=& -\frac{\im \nu^{(n)} (\alpha^{(n)})^3 }{4}{\vec e}_j \cdot \int_{B^{(n)}} {\vec \xi}^{(n)} \times ({\vec \theta}_k^{(n)} + {\vec e}_k \times {\vec \xi}^{(n)} ) \dif {\vec \xi}^{(n)},  \\
({\mathcal N} [\alpha^{(n)} B^{(n)}])_{jk} := &  (\alpha^{(n)})^3 \left ( 1- \frac{\mu_0}{\mu_*^{(n)}} \right ) \int_{B^{(n)}} \left (
{\vec e}_j \cdot {\vec e}_k + \frac{1}{2}  {\vec e}_j \cdot  \nabla_\xi \times {\vec \theta}_k^{(n)} \right ) \dif {\vec \xi}^{(n)}.
\end{align}
\end{subequations}
These, in turn, rely on the vectoral solutions ${\vec \theta}_k^{(n)} $, $k=1,2,3,$ to the transmission problem
\begin{subequations}
\begin{align}
\nabla_\xi \times (\mu_*^{(n)} )^{-1} \nabla_\xi \times {\vec \theta}_k ^{(n)}- \im \omega \sigma_*^{(n)} (\alpha^{(n)})^2 {\vec \theta }_k^{(n)} & = \im \omega \sigma_*^{(n)} (\alpha^{(n)})^2 {\vec e}_k \times {\vec \xi}^{(n)}  && \text{in $B^{(n)}$ } ,\\
\nabla_\xi \cdot {\vec \theta}_k ^{(n)} = 0 , \qquad \nabla_\xi \times \mu_0^{-1} \nabla_\xi \times {\vec \theta}_k ^{(n)} & = {\vec 0} && \text{in $(B^{(n)})^c $ }, \\
[{\vec n} \times {\vec \theta}_k^{(n)} ]_{\Gamma^{(n)}}  & = {\vec 0} && \text{on $\Gamma^{(n)}$} , \\
 [{\vec n} \times \mu^{-1} \nabla_\xi \times {\vec \theta}_k^{(n)} ]_{\Gamma^{(n)}} & = -2 [\mu^{-1 } ]_{\Gamma^{(n)}} {\vec n} \times {\vec e}_k  && \text{on $\Gamma^{(n)}$},\\
{\vec \theta}_k^{(n)} & = O( | {\vec \xi}^{(n)} |^{-1}) && \text{as $|{\vec \xi}^{(n)} | \to \infty$ },
\end{align}\label{eqn:transproblemthetar2}
\end{subequations}
where $( B^{(n)})^c : = {\mathbb R}^3 \setminus \overline{B^{(n)}}$, $ \Gamma^{(n)}:= \partial B^{(n)}$ and  ${\vec \xi}^{(n)}$ is measured from an origin chosen to be in $B^{(n)}$.
\end{theorem}

\begin{proof}
The result follows from  by using a tensor representation of the asymptotic formula in Theorem~\ref{thm:rev3_2}, which is an extension of Theorem 3.2 obtained in~\cite{ammarivolkov2013} for $N$ sufficiently well spaced objects. A tensor representation of this result leads to each of the $N$ objects being characterised by a rank 4 tensor. Then, by considering each object in turn and repeating the same arguments as in Theorem 3.1 in~\cite{ledgerlionheart2014}, which exploits the skew symmetries of the tensor coefficients, the result stated in (\ref{eqn:asymp2}) is obtained. The symmetry of  ${\mathcal M}[\alpha^{(n)} B^{(n)}]$ follows from repeating the arguments in Lemma 4.4 in~\cite{ledgerlionheart2014}.
\end{proof}

\begin{corollary} \label{corll:inhomn1}
For the case of $N=1$ then ${\vec B}_{\vec \alpha}$ becomes $B_\alpha$, ${\vec H}_{\vec \alpha}$ becomes ${\vec H}_\alpha$ and Theorem~\ref{thm:objectsnocloselyspaced} reduces to the case of a single homogenous object as obtained in~\cite{ammarivolkov2013,ledgerlionheart2014} and described in Section~\ref{sect:back}.
\end{corollary}

\subsection{Single Inhomogeneous Object}\label{sect:closelyspaced}
In this case,  ${\vec B}_\alpha := \bigcup_{n=1}^N B_\alpha^{(n)} =\alpha  \bigcup_{n=1}^N  B^{(n)} + {\vec z} = \alpha {\vec B} + {\vec z}$ describes a single object comprised of $N$ constitute parts, $B_{\alpha}^{(n)}$, such that there is a single common size parameter $\alpha$, the configuration ${\vec B}$ contains the origin, and ${\vec z}$ is a single translation, as illustrated in Figure~\ref{fig:inhom}. Notice that for the inhomogeneous case we use $B_\alpha^{(n)}$ rather than $(B_\alpha)^{(n)}$ as $\alpha$ is the same for all $n$ and we revert to the use of non-bold $\alpha$ subscripts for the fields ${\vec E}_{\vec \alpha}$ and ${\vec H}_{\alpha}$, which satisfy (\ref{eqn:eddyqns}).

 \begin{figure}
\begin{center}
\includegraphics[width=0.8\textwidth]{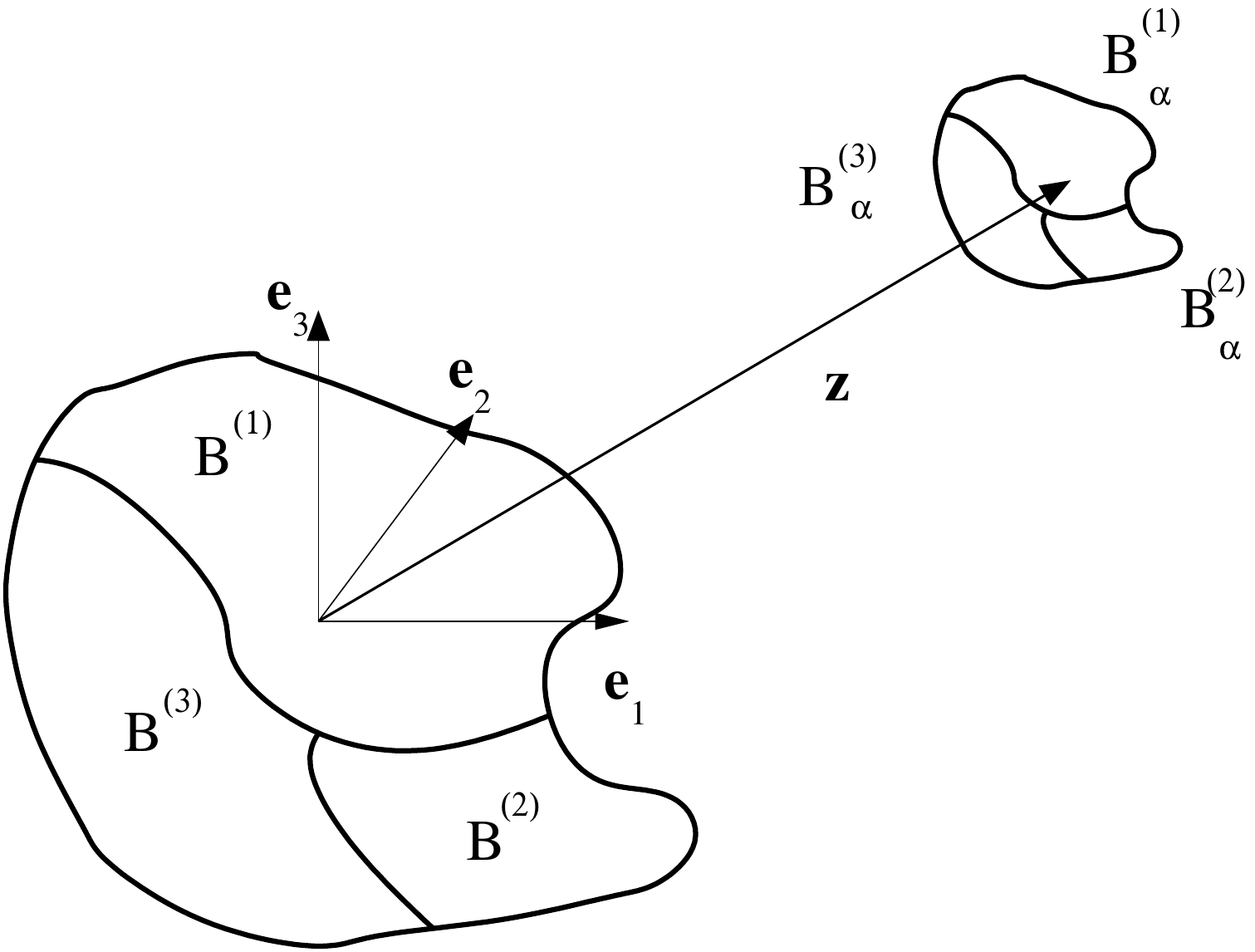}
\end{center}
\caption{Illustration of a typical situation of an inhomogeneous object consisting of $N=3$ subdomains such that the complete object is ${\vec B}_\alpha = \bigcup_{n=1}^N B_{\alpha}^{(n)} =  \alpha\bigcup_{n=1}^N B^{(n)}+{\vec z} = \alpha {\vec B} + {\vec z}$. }\label{fig:inhom}
\end{figure}

 The material parameters of the constitute parts of the object ${\vec B}_\alpha$ are
\begin{align}
\sigma_{ \alpha}= \left \{ \begin{array}{ll} \sigma_{*}^{(n)} & \text{in $B_\alpha^{(n)} $} \\
                                                             0 & \text{in ${\vec B}_\alpha^c $}
                                                             \end{array},  \right . \qquad
\mu_{ \alpha}= \left \{ \begin{array}{ll} \mu_*^{(n)} & \text{in $B_\alpha^{(n)} $} \\
                                                             \mu_0 & \text{in ${\vec B}_\alpha^c $}
                                                             \end{array} \right . \nonumber ,
\end{align}
where ${\vec B}_\alpha^c: = {\mathbb R}^3\setminus \overline{{\vec B}_{ \alpha}}$
and we drop the subscript $\alpha$ on $\mu$ and $\sigma$ when considering the object ${\vec B}$. We redefine $\nu_{\min} \le {\nu}^{(n)} := \omega \mu \sigma_*^{(n)} {\alpha}^2\le \nu_{\max}$ with the same requirements on $\nu_{\max}$ as before.

The task is then to provide a low-cost description of  $({\vec H}_{ \alpha} - {\vec H}_0)({\vec x})$ for ${\vec x}$ away from ${\vec B}_{ \alpha}$. This is accomplished through the following result.
 
 \begin{theorem}  \label{thm:objectscloselyspaced}
 For an inhomogeneous object  ${\vec B}_\alpha = \alpha {\vec B} + {\vec z}$ made up of $N$ constitute parts
 with parameters such that 
$\nu_{min} \le {\nu}^{(n)} \le \nu_{max}$  the perturbed magnetic field at positions ${\vec x}$ away from ${\vec B}_\alpha$ satisfies
\begin{align}
({\vec H}_\alpha - {\vec H}_0)({\vec x})_i =& ({\vec D}_{x}^2 G({\vec x}, {\vec z} ))_{ij}  \left ( {\mathcal M}\left [{\alpha} {\vec B}  \right ] \right )_{jk} ({\vec H}_0( {\vec z} ) )_k +( {\vec R}({\vec x}))_i ,  \label{eqn:asymp3}
\end{align}
where 
\begin{equation}
|{\vec R}({\vec x}| \le C \alpha^4 \| {\vec H}_0 \|_{W^{2,\infty}({\vec B}_\alpha)}  , \nonumber
\end{equation}
uniformly in ${\vec x}$ in any compact set away from $ {\vec B}_\alpha$. The coefficients of the complex symmetric MPT ${\mathcal M}\left [ {\alpha} {\vec B}\right ] =- {\mathcal C}\left [{\alpha} {\vec B} \right] + {\mathcal N}\left [{\alpha} {\vec B} \right ]$
 are given by
\begin{subequations}  \label{eqn:mcheckcup} 
\begin{align}
\left ( {\mathcal C} \left [ {\alpha} {\vec B } \right ] \right )_{jk} :=& -  \frac{\im {\alpha}^3 }{4} \sum_{n=1}^N  \nu^{(n)}  {\vec e}_j \cdot  \int_{{B}^{(n)}  }  {\vec \xi} \times ({\vec \theta}_k + {\vec e}_k \times {\vec \xi} ) \dif {\vec \xi},  \\
\left ({\mathcal N} \left [ {\alpha} {\vec B}   \right ] \right )_{jk} := & {\alpha}^3  \sum_{n=1}^N   \left ( 1- \frac{\mu_0}{\mu_*^{(n)}} \right )  \int_{{B}^{(n)}} \left (
{\vec e}_j \cdot {\vec e}_k + \frac{1}{2}  {\vec e}_j \cdot  \nabla_\xi \times {\vec \theta}_k \right ) \dif {\vec \xi}.
\end{align}
\end{subequations}
which, in turn, rely on the vectoral solutions ${\vec \theta}_k$, $k=1,2,3,$ to the transmission problem
\begin{subequations}
\begin{align}
\nabla_\xi \times \mu^{-1} \nabla_\xi \times {\vec \theta}_k - \im \omega \sigma {\alpha}^2 {\vec \theta }_k & = \im \omega \sigma {\alpha}^2 {\vec e}_k \times {\vec \xi}  && \text{in ${\vec B } $ } ,\\
\nabla_\xi \cdot {\vec \theta}_k  = 0 , \qquad \nabla_\xi \times \mu_0^{-1} \nabla_\xi \times {\vec \theta}_k  & = {\vec 0} && \text{in ${\vec B }^c := {\mathbb R}^3 \setminus \overline{\vec B}$ }, \\
[{\vec n} \times {\vec \theta}_k ]_{\Gamma } = {\vec 0},  \qquad [{\vec n} \times \mu^{-1} \nabla_\xi \times {\vec \theta}_k ]_{\Gamma } & = -2 [\mu^{-1 } ]_{\Gamma } {\vec n} \times {\vec e}_k  && \text{on $\Gamma $},\\
{\vec \theta}_k & = O( | {\vec \xi} |^{-1}) && \text{as $| {\vec \xi} | \to \infty$ },
\end{align}\label{eqn:transproblemthetar3}
\end{subequations}
where ${\vec \xi}$ is measured from the centre of ${\vec B}$ and, in this case, $\Gamma := \partial {\vec B} \cup \{ \partial B^{(n) } \cap \partial B^{ (n)}, n,m=1,\ldots, N, n \ne m \}$. 
\end{theorem}

\begin{proof}
The result follows from  by using a tensor representation of the asymptotic formula in Theorem~\ref{thm:rev3_2v2}, which is an extension of Theorem 3.2 obtained in~\cite{ammarivolkov2013} for an homogeneous object to the inhomogeneous case. Using a tensor representation of this result leads to the object being characterised in terms of a rank 4 tensor. Then, by repeating the same arguments as in Theorem 3.1 in ~\cite{ledgerlionheart2014}, which exploits the skew symmetries of the tensor coefficients, the result stated in (\ref{eqn:asymp3}) is obtained. The symmetry of  ${\mathcal M}[\alpha {\vec B} ]$ follows from repeating the arguments in Lemma 4.4 in~\cite{ledgerlionheart2014}.

\end{proof}

\begin{corollary} \label{corll:inhomn2}
For the case of $N=1$ then ${\vec B}_\alpha$ becomes $B_\alpha$ and Theorem~\ref{thm:objectscloselyspaced} reduces to the case of a single homogenous object as obtained in~\cite{ammarivolkov2013,ledgerlionheart2014} and described in Section~\ref{sect:back}.
\end{corollary}

\begin{corollary} \label{corll:inhom}
Theorem~\ref{thm:objectscloselyspaced} also immediately applies to objects that are closely spaced and, in this case,
  ${\vec B}_\alpha =  \alpha {\vec B} + {\vec z}$
implies a single size parameter ${\alpha}$  and  a single translation ${\vec z} $ for the configuration ${\vec B}$.  An illustration of a typical configuration is shown in Figure~\ref{fig:multobjectclose}. In this figure, there are $N=3$ objects consisting of three spheres  configured such that they scale and translate together according to ${\alpha}$ and ${\vec z}$, respectively, and, in this case, ${\vec B}$ is the combined configuration of three (larger) spheres with different radii and with centres located away from the origin.
 \begin{figure}
\begin{center}
\includegraphics[width=0.8\textwidth]{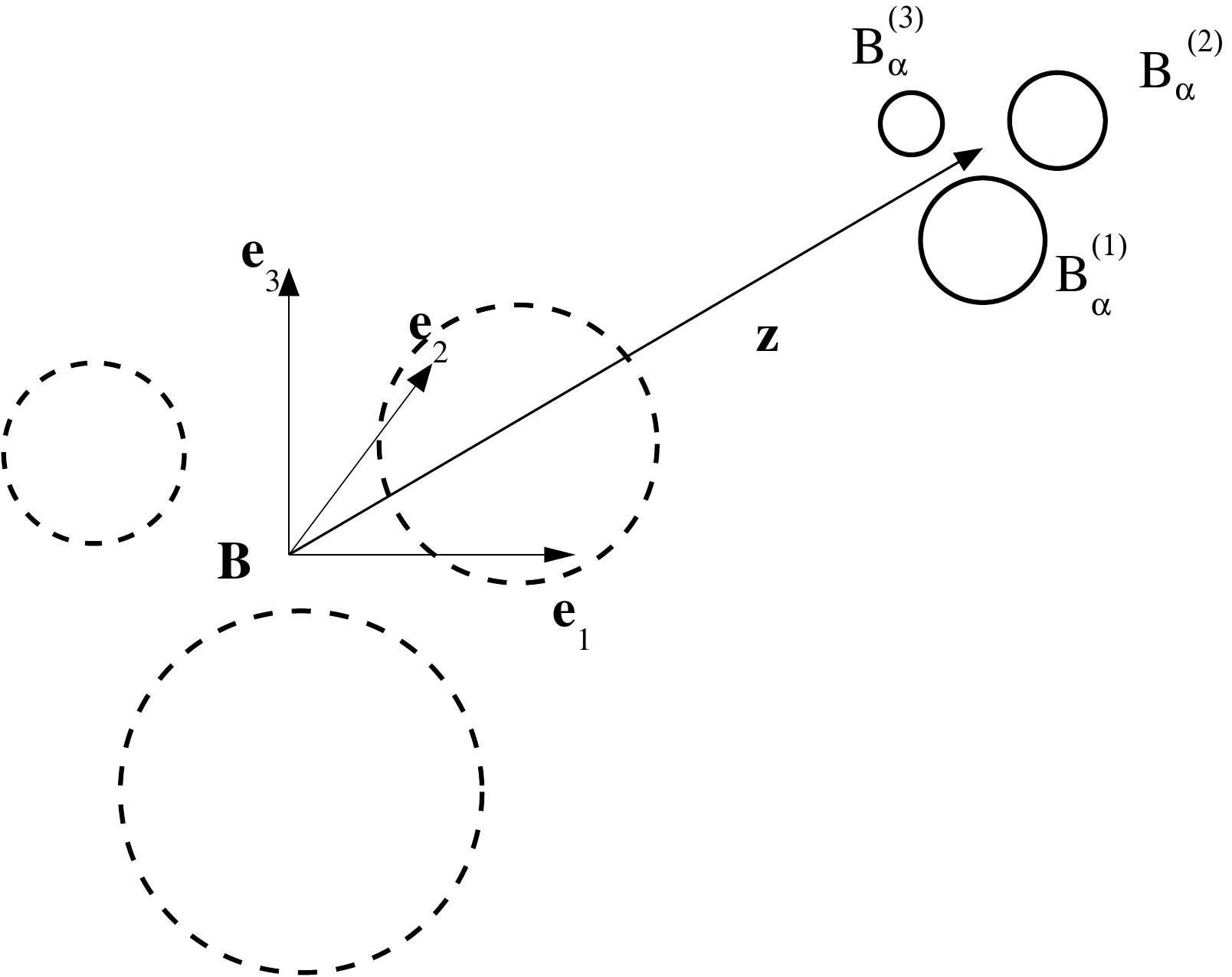}
\end{center}
\caption{Illustration of a typical situation of $N=3$ closely spaced objects of the form ${\vec B}_\alpha = \bigcup_{n=1}^N B_\alpha^{(n)} =\alpha \bigcup_{n=1}^N B^{(n)} +{\vec z}=  \alpha {\vec B} + {\vec z}$ where each object is a sphere, ${\alpha}$ is a single scaling parameter, ${\vec z}$ describes their translation of the configuration from the origin. }\label{fig:multobjectclose}
\end{figure}
\end{corollary}

\begin{remark}
The applicability of Theorem~\ref{thm:objectscloselyspaced} to closely spaced objects  is expected to be limited since, in order to compute the characterisation, prior knowledge of the multiple object configuration (ie location and orientation with respect to each other) is required, which, in practice, will not be the case.  The formula also requires that the objects be closely spaced as there is a single scaling parameter and single translation that describes the configuration, but prior knowledge of the location of the configuration is not required. Instead, this result is expected to be of more practical value in the description of inhomogeneous objects where the configuration of the different regions of an object will be known in advance.
\end{remark}

\begin{remark}
The translation invariance of the MPT coefficients described by Proposition 5.2 in~\cite{ammarivolkov2013b} and the tensor transformation rules described in the proof of Theorem 3.1
in~\cite{ledgerlionheart2014} carry over immediately to the rank 2 MPTs defined in (\ref{eqn:mcheckmult}) and (\ref{eqn:mcheckcup}).
\end{remark}

\section{Results for the Proof of Theorem~\ref{thm:objectsnocloselyspaced}} \label{sect:proof1}

\subsection{Elimination of the Current Source} \label{sect:elimcurrent}

Recall from ~\cite{ammarivolkov2013} that
 \begin{align*}
{\vec X}_{\vec \alpha} ({\mathbb R}^3) : =  &\left \{ {\vec u} : \frac{\vec u}{\sqrt{ 1+ |{\vec x}|^2 }}\in L^2({\mathbb R}^3 )^3 , \nabla \times {\vec u} \in  L^2({\mathbb R}^3 )^3, \nabla \cdot {\vec u} =0 \text{ in ${\vec B}_{\vec \alpha}^c$} \right \} ,\nonumber \\
\tilde{\vec X}_{\vec \alpha} ({\mathbb R}^3) : = &  \left \{ {\vec u} : {\vec u} \in {\vec X}_{\vec \alpha} ({\mathbb R}^3) , \ \int_{\Gamma_\alpha} {\vec u} \cdot {\vec n} |_+ \dif {\vec x} = 0 \right \} \nonumber ,
\end{align*}
and the weak solution for the interaction field is : Find ${\vec E}_{\vec \alpha} \in \tilde{\vec X}_{\vec \alpha}$ such that
\begin{equation*}
a_\alpha ({\vec E}_{\vec \alpha} , {\vec v}) = ({\vec J}_0,{\vec v})_{{\mathbb R}^3} = ({\vec J}_0,{\vec v})_{\hbox{supp(${\vec J}_0$)}}  \qquad \forall {\vec v} \in \tilde{\vec X}_{\vec \alpha},
\end{equation*}
 where $(\cdot,\cdot)_\Omega$ denotes the standard $L^2$ inner product over $\Omega$. In a departure from~\cite{ammarivolkov2013}, we have, for multiple objects, that
\begin{align*}
a({\vec u},{\vec v}):= & (\mu_0^{-1} \nabla \times {\vec u}, \nabla \times {\vec v})_{{\vec B}_{\vec \alpha}^c } + ( \mu_{\vec \alpha}^{-1} \nabla \times {\vec u}, \nabla \times {\vec v})_{{\vec B}_{\vec \alpha}} \nonumber \\
& - \im \omega (\sigma_{\vec \alpha} {\vec u},{\vec v})_{{\vec B}_{\vec \alpha}} .
\end{align*}
Noting that the weak solution for the background field is: Find  ${\vec E}_0 \in \tilde{\vec X}_{\vec \alpha}$ such that
\begin{equation*}
 (\mu_0^{-1} \nabla \times {\vec E}_0 ,\nabla \times  {\vec v})_{{\mathbb R}^3} =  ({\vec J}_0,{\vec v})_{\hbox{supp(${\vec J}_0$)}} \qquad \forall {\vec v} \in \tilde{\vec X}_{\vec \alpha},
\end{equation*}
we can write:  Find ${\vec E}_{\vec \alpha} \in \tilde{\vec X}_{\vec \alpha}$ such that
\begin{equation*}
a ( {\vec E}_{\vec \alpha} , {\vec v}) = (\mu_0^{-1} \nabla \times {\vec E}_0 ,\nabla \times  {\vec v})_{{\mathbb R}^3}\qquad \forall {\vec v} \in \tilde{\vec X}_{\vec \alpha} ,
\end{equation*}
which eliminates the current source. We also obtain that
\begin{align}
(\mu_0^{-1} &  \nabla \times ({\vec E}_{\vec \alpha} - {\vec E}_0), \nabla \times {\vec v})_{{\vec B}_{\vec \alpha}^c }
 +   ( \mu_{\vec \alpha}^{-1}    \nabla \times ({\vec E}_{\vec \alpha} - {\vec E}_0), \nabla \times {\vec v})_{{\vec B}_{\vec \alpha}}  \nonumber \\
&  -\im \omega ( \sigma_{\vec \alpha} ( {\vec E}_{\vec \alpha}- {\vec E}_0),{\vec v})_{{\vec B}_{\vec \alpha}} =
( ((\mu_{\vec \alpha}^{-1}- \mu_0) \nabla \times {\vec E}_0,\nabla \times {\vec v})_{{\vec B}_{\vec \alpha}} \nonumber \\
& + \im \omega ( \sigma_{\vec \alpha} {\vec E}_0, {\vec v})_{{\vec B}_{\vec \alpha}} 
\label{eqn:elimcurrt1}.
 \end{align}

\subsection{Energy Estimates}
In~\cite{ammarivolkov2013} a vector field ${\vec F}({\vec x})$ was introduced such that its curl is equal  to the first two terms of a Taylor's series expansion of $\nabla \times {\vec E}_0$ about ${\vec z}$ for $|{\vec x}-{\vec z}| \to 0$ for the case of a single object $B_\alpha$. This was possible as the current source ${\vec J}_0$ is supported away from the object and so ${\vec H}_0({\vec x})=\frac{1}{\im \omega \mu_0} \nabla \times {\vec E}_0({\vec x})$ is analytic where the expansion is applied. We extend this to the multiple object case by requiring that ${\vec J}_0$ be supported away from ${\vec B}_{\vec \alpha}$ and introduce the following for $n=1,\ldots,N$
\begin{align*}
{\vec F}^{(n)} ({\vec x}) = & \frac{1}{2} (\nabla_{z} \times {\vec E}_0({\vec z})) ({\vec z}^{(n)}) \times ({\vec x}-{\vec z}^{(n)}) \nonumber \\ &+ \frac{1}{3} {\vec D}_{z} ( \nabla_z \times {\vec E}_0({\vec z}))({\vec z}^{(n)}) ({\vec x}-{\vec z}^{(n)})  \times   ({\vec x}-{\vec z}^{(n)}), \\
\nabla \times {\vec F}^{(n)} ({\vec x}) = & (\nabla_{z} \times {\vec E}_0({\vec z})) ({\vec z}^{(n)}) + {\vec D}_{z} ( \nabla_z \times {\vec E}_0({\vec z})) ({\vec z}^{(n)}) ({\vec x}-{\vec z}^{(n)}) ,
\end{align*}
so that
\begin{align*}
 {\vec F}^{(n)} ({\vec x}) = & \frac{\im \omega \mu_0}{2} {\vec H}_0 ({\vec z}^{(n)}) \times ({\vec x}-{\vec z}^{(n)}) + \frac{\im \omega \mu_0}{3} {\vec D}_{z} ( {\vec H}_0 ({\vec z}))({\vec z}^{(n)}) ({\vec x}-{\vec z}^{(n)})  \times   ({\vec x}-{\vec z}^{(n)}) , \\
\nabla \times {\vec F}^{(n)} ({\vec x}) = &\im \omega \mu_0  \left ( {\vec H}_0 ({\vec z}^{(n)})+ {\vec D}_{z} ( {\vec H}_0({\vec z})) ({\vec z}^{(n)}) ({\vec x}-{\vec z}^{(n)}) \right ) .
\end{align*}
In other words, $ \nabla \times {\vec F}^{(n)} ({\vec x})$ is the first two terms in a Taylor series of $\im \omega \mu_0{\vec H}_0({\vec x})$ about ${\vec z}^{(n)}$ as $|{\vec x}-{\vec z}^{(n)}|\to0$ and so
\begin{align}
\| \im \omega \mu_0 {\vec H}_0 ({\vec x}) - \nabla \times {\vec F}^{(n)} \|_{L^\infty\left ((B_\alpha)^{(n)} \right )} \le & C(\alpha^{(n)})^2 \| \nabla \times {\vec E}_0 \|_{W^{2,\infty}\left ((B_\alpha)^{(n)} \right )} , \nonumber  \\
\| \im \omega \mu_0 {\vec H}_0 ({\vec x}) - \nabla \times {\vec F}^{(n)} \|_{L^2 \left ((B_\alpha)^{(n)} \right )} \le & C (\alpha^{(n)})^{\frac{3}{2}} \| \im \omega \mu_0 {\vec H}_0 ({\vec x}) - \nabla \times {\vec F}^{(n)} \|_{L^\infty \left ((B_\alpha)^{(n)} \right )}\nonumber \\ 
\le & C (\alpha^{(n)} )^{\frac{7}{2}} \| \nabla \times {\vec E}_0 \|_{W^{2,\infty}\left ((B_\alpha)^{(n)} \right )} \label{eqn:taylorbackfd} .
\end{align}
where here and in the following $C$ denotes a generic constant unless otherwise indicated.

\begin{remark}
Higher order Taylor series could be considered (as previously in~\cite{ledgerlionheart2017} for the case of a single object) and lead to a more accurate representation of the field in terms of GMPTs. However, in order for such a representation to apply, there will be further implications in the allowable distance between the objects. 
\end{remark}

The introduction of ${\vec F}^{(n)}({\vec x})$ motivates the introduction of the following problem: Find ${\vec w}^{(n)}\in \tilde{\vec X}_{\vec \alpha}$ such that
\begin{align}
(\mu_0^{-1} & \nabla \times {\vec w}^{(n)}, \nabla \times {\vec v})_{(( B_\alpha)^{(n)})^c }+((\mu_*^{(n)})^{-1} \nabla \times {\vec w}^{(n)}, \nabla \times {\vec v})_{(B_\alpha)^{(n)}}  -\im \omega (\sigma_*^{(n)} {\vec w}^{(n)}, {\vec v})_{(B_\alpha)^{(n)}} \nonumber \\
&= ((\mu_0^{-1}- (\mu_*^{(n)})^{-1} ) \nabla \times {\vec F}^{(n)} ,\nabla \times {\vec v})_{(B_\alpha)^{(n)}} + \im \omega ( \sigma_*^{(n)} {\vec F}^{(n)} , {\vec v})_{(B_\alpha)^{(n)}} \qquad \forall {\vec v} \in \tilde{\vec X}_{\vec \alpha} \label{eqn:wmnweakform} ,
\end{align}
where $(( B_\alpha)^{(n)})^c:={\mathbb R}^3 \setminus \overline{(B_\alpha)^{(n)}} $.
By the addition of such problems, we have 
\begin{align}
&\left (\mu_0^{-1} \nabla \times {\vec w} , \nabla \times {\vec v} \right )_{{\vec B}_{\vec \alpha}^c }+
(\mu_{\vec \alpha}^{-1} \nabla \times {\vec w}_{\vec \alpha}, \nabla \times {\vec v})_{{\vec  B}_{\vec \alpha}}  -\im \omega (\sigma_{\vec \alpha} {\vec w}_{\vec \alpha}, {\vec v})_{{\vec  B}_{\vec \alpha}}   \nonumber \\
&+\sum_{n,m=1}^N (\mu_0^{-1} \nabla \times {\vec w}^{(m)} (1-\delta_{mn}), \nabla \times {\vec v})_{(B_\alpha)^{(n)}}
=  ((\mu_0^{-1}- (\mu_{\vec \alpha}^{-1} ) \nabla \times {\vec F}_{\vec \alpha} ,\nabla \times {\vec v})_{{\vec  B}_{\vec \alpha}} \nonumber \\
& + \im \omega ( \sigma_{\vec \alpha} {\vec F}_{\vec \alpha} , {\vec v})_{{\vec  B}_{\vec \alpha}}  \label{eqn:elimcurrt2} ,
\end{align}
where ${\vec w}:= \sum_{n=1}^N  {\vec w}^{(n)}$, ${\vec w}_{\vec \alpha} = {\vec w}^{(n)}$ in $(B_\alpha)^{(n)}$ and ${\vec F}_{\vec \alpha} = {\vec F}^{(n)}$ in $(B_\alpha)^{(n)}$.

We also remark that, associated with (\ref{eqn:wmnweakform}), is the strong form
\begin{subequations} \label{eqn:wnstrongform}
\begin{align}
\nabla \times  (\mu_*^{(n)}) ^{-1} \nabla \times {\vec w}^{(n)} - \im \omega \sigma_*^{(n)}  {\vec w}^{(n)} = &  \im \omega \sigma_*^{(n)} {\vec F}^{(n)} && \text{in $(B_\alpha)^{(n)}$} , \\
\nabla \times \mu_0^{-1} \nabla \times {\vec w}^{(n)} = & {\vec 0} &&\text{in $(( B_\alpha)^{(n)})^c$},  \\
\nabla \cdot {\vec w}^{(n)} = & 0 &&\text{in $(( B_\alpha)^{(n)})^c$} , \\
\left [ {\vec n} \times {\vec w}^{(n)}  \right ]_{(\Gamma_{\alpha})^{(n)}}  = &  {\vec 0} &&\text{on $(\Gamma_{\alpha})^{(n)}:= \partial (B_\alpha)^{(n)}$ } , \\
 \left [ {\vec n} \times \mu^{-1} \nabla \times {\vec w}^{(n)} \right ]_{(\Gamma_{\alpha})^{(n)}} = & {} && \nonumber \\
 - (\mu_0^{-1} - (\mu_*^{(n)})^{-1} ) &{\vec n} \times \nabla \times {\vec F}^{(n)}  &&\text{on $(\Gamma_{\alpha})^{(n)}$ },  \\
 {\vec w}^{(n)} =& O( | {\vec x}|^{-1} ) && \text{as $|{\vec x}| \to \infty $} , 
\end{align}
 \end{subequations}
which follows from using
\begin{align} 
(\mu_0^{-1}- (\mu_*^{(n)})^{-1})  ( \nabla \times {\vec F}^{(n)} , \nabla \times {\vec v})_{(B_\alpha)^{(n)}}  =  &
 ( \mu_0^{-1} - (\mu_*^{(n)})^{-1})  \int_{(\Gamma_{\alpha})^{(n)}} \nabla \times {\vec F}^{(n)} \times {\vec n}  \cdot \overline{\vec v} \dif {\vec x} \nonumber \\
=& \int_{(\Gamma_{\alpha})^{(n)}} \left [ \mu^{-1} \nabla \times {\vec F}^{(n)} \times {\vec n} \right  ]_{(\Gamma_{\alpha})^{(n)}} \cdot \overline{\vec v} \dif {\vec x} \nonumber .
\end{align}

\begin{lemma} \label{lemma:closelyspacedw}
For objects $(B_\alpha)^{(n)}$ and $(B_{\alpha})^{(m)}$ with $n \ne m$ we have that
\begin{align*}
\| \nabla \times {\vec w}^{(n)} \|_{L^2\left ((B_{\alpha})^{(m)} \right )} 
\le C \frac{\alpha_{\max}^{\frac{7}{2}} }{ |{\vec z}^{(m)} - {\vec z}^{(n)}  |^2}  \| \nabla \times {\vec E}_0 \|_{W^{2,\infty} \left ((B_\alpha)^{(n)} \cup (B_{\alpha})^{(m)} \right ) 
} .
\end{align*}
\end{lemma}
\begin{proof}
Introducing ${\vec \xi}^{(n)} = \frac{{\vec x} - {\vec z}^{(n)}}{\alpha^{(n)}}$, which, without loss of generality, we assume the origin to be in $B^{(n)}$. We set ${\vec w}^{(n)} ({\vec x}) = \alpha^{(n)} {\vec w}_{0}^{(n)}\left ( \frac{{\vec x} - {\vec z}^{(n)}}{\alpha^{(n)}} \right )  = \alpha^{(n)} {\vec w}_{0}^{(n)} ({\vec \xi}^{(n)})$ and so $\nabla_ x \times {\vec w}^{(n)} ({\vec x}) = \nabla_\xi \times {\vec w}_0^{(n)} ( {\vec \xi}^{(n)}) =   \nabla_\xi \times {\vec w}_0^{(n)} \left (\frac{{\vec x} - {\vec z}^{(n)}}{\alpha^{(n)}} \right )$. Note that ${\vec w}_0^{(n)}({\vec \xi}^{(n)})$ satisfies
\begin{subequations} \label{eqn:w0nstrongform}
\begin{align}
\nabla_\xi  \times  (\mu_*^{(n)}) ^{-1} \nabla_\xi \times {\vec w}_0^{(n)} - \im \omega \sigma_*^{(n)}  {\vec w}_0^{(n)} = &  \im \omega \sigma_*^{(n)} (\alpha^{(n)})^2  && {} \nonumber  \\
[ (\alpha^{(n)})^{-1} {\vec F}^{(n)} &( {\vec z}^{(n)} + \alpha^{(n)} {\vec \xi}^{(n)} ) ] && \text{in $B^{(n)}$} , \\
\nabla_\xi \times \mu_0^{-1} \nabla_\xi \times {\vec w}_0^{(n)} = & {\vec 0} &&\text{in $ (B^{(n)})^c $} , \\
\nabla_\xi \cdot {\vec w}_0^{(n)} = & 0 &&\text{in $ (B^{(n)})^c $} , \\
\left [ {\vec n} \times {\vec w}_0^{(n)}  \right ]_{\Gamma^{(n)}}  = &  {\vec 0} &&\text{on $\Gamma^{(n)} $ }  , \\
 \left [ {\vec n} \times \mu^{-1} \nabla_\xi \times {\vec w}_0^{(n)} \right ]_{\Gamma^{(n)}} = &  -(\mu_0^{-1} - (\mu_*^{(n)})^{-1} )&& {} \nonumber \\
  {\vec n} \times \nabla_\xi \times & {\vec F}^{(n)} ( {\vec z}^{(n)} + \alpha^{(n)} {\vec \xi}^{(n)}  ) &&\text{on $\Gamma^{(n)}$ } , \\
 {\vec w}_0^{(n)} =& O( | {\vec \xi}^{(n)}|^{-1} ) && \text{as $|{\vec \xi}^{(n)}| \to \infty $} .
\end{align}
 \end{subequations}
From the above we have that $|{\vec w}_0^{(n)}| \le C |{\vec \xi}^{(n)}|^{-1} \| \nabla \times {\vec E}_0 \|_{W^{2,\infty}\left ((B_\alpha)^{(n)} \right )}$ for sufficiently large $|{\vec \xi}^{(n)}|$ and so we estimate that $|\nabla_\xi \times {\vec w}_0^{(n)}| \le C |{\vec \xi}^{(n)}|^{-2} \| \nabla \times {\vec E}_0 \|_{W^{2,\infty}\left ((B_\alpha)^{(n)} \right )}$ for the same case.  Thus, for $m \ne n$,
\begin{align}
\| \nabla \times {\vec w}^{(n)} \|_{L^2(B_\alpha^{(m)} )}  =&  \left (( \alpha^{(m)})^3 \int_{B^{(m)}} \left | \nabla_x \times {\vec w}^{(n)}  ( \alpha^{(m)}  {\vec \xi}^{(m)} + {\vec z}^{(m)} )  \right |^2 \dif {\vec \xi}^{(m)} \right )^{1/2} \nonumber \\
 =& \left (( \alpha^{(m)})^3 \int_{B^{(m)}} \left | \nabla_\xi \times {\vec w}_0^{(n)}  \left (\frac{ \alpha^{(m)}  {\vec \xi}^{(m)} + {\vec z}^{(m)} - {\vec z}^{(n)}}{\alpha^{(n)}}  \right )  \right |^2 \dif {\vec \xi}^{(m)} \right )^{1/2} \nonumber \\
 \le & C (\alpha^{(m)})^{\frac{3}{2}}  \left | \frac{{\vec z}^{(m)} - {\vec z}^{(n)}}{\alpha^{(n)}}   \right |^{-2}  \| \nabla \times {\vec E}_0 \|_{W^{2,\infty} \left ((B_\alpha)^{(n)}  \right )} \nonumber \\
 \le & C \frac{ \alpha_{\max}^{\frac{7}{2}}} {\left  | {\vec z}^{(m)} - {\vec z}^{(n)} \right |^{2}  } \| \nabla \times {\vec E}_0 \|_{W^{2,\infty} \left ((B_\alpha)^{(n)} \cup (B_{\alpha})^{(m)} \right)} \nonumber ,
\end{align}
where we have used  $ \| \nabla \times {\vec E}_0 \|_{W^{2,\infty} \left ((B_\alpha)^{(n)} \right  )} \le \| \nabla \times {\vec E}_0 \|_{W^{2,\infty} \left ((B_\alpha)^{(n)} \cup (B_{\alpha})^{(m)} \right)}$.

\end{proof}
\begin{corollary}  \label{coll:closelyspacedw}
Given the description $(B_\alpha)^{(n)} = \alpha^{(n)} B^{(n)} + {\vec z}^{(n)}$, we are free to configure $B^{(n)}$ in different ways provided that the origin lies at a
point in $B^{(n)}$ (similarly with $(B_\alpha)^{(m)} = \alpha ^{(m)} B^{(m)} + {\vec z}^{(m)}$) . Thus $|{\vec z}^{(m)}- {\vec z}^{(n)}|$ will be smallest when the origin lies in the boundaries of the objects, as illustrated in Figure~\ref{fig:coll:closelyspacedw}. 
 \begin{figure}
 \begin{center}
 \includegraphics[width=0.8\textwidth]{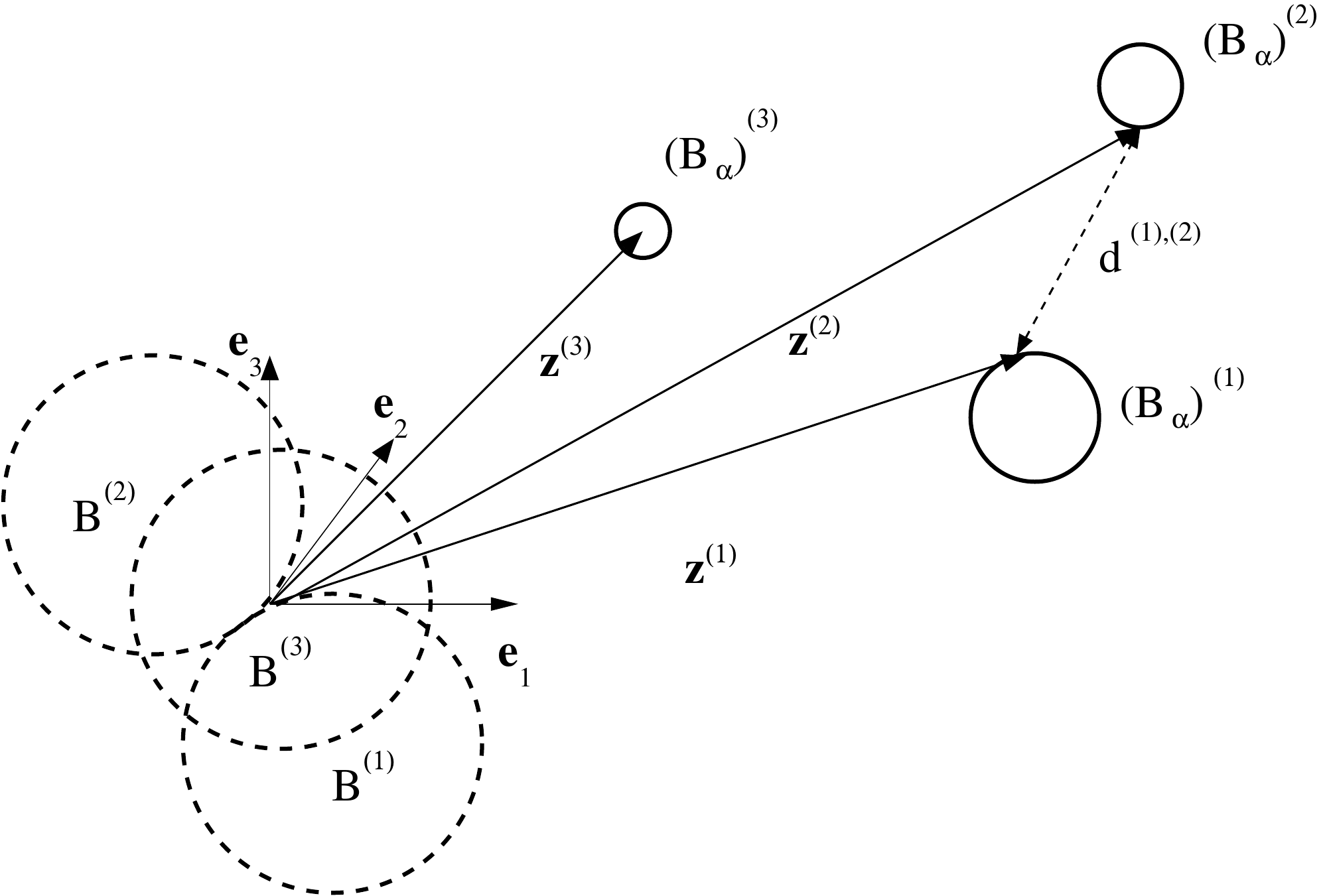}
\end{center}
\caption{Illustration to show how each $B^{(n)}$ can be configured differently provided that the origin lies within the object. Consequently $d^{(1),(2)} = |{\vec z}^{(1)}- {\vec z}^{(2)} |$ will be minimum when the objects $B^{(1)}$ and $B^{(2)}$ are configured such that the origin is a suitable point on the boundaries of these objects.} \label{fig:coll:closelyspacedw}
 \end{figure}
  Requiring that $\displaystyle |{\vec z}^{(m)}- {\vec z}^{(n)} |= \min_{n,m=1,\ldots,N, n\ne m}|\partial (B_\alpha)^{(n)} -\partial (B_\alpha)^{(m)}| > C > \alpha_{\max}$ then Lemma~\ref{lemma:closelyspacedw} implies that
 \begin{align}
\|\nabla \times {\vec w}^{(n)} \|_{L^2(B_\alpha^{(m)})} \le C {\alpha_{\max}^{\frac{7}{2}}} \| \nabla \times {\vec E}_0 \|_{W^{2,\infty} \left ((B_\alpha)^{(n)} \cup (B_{\alpha})^{(m)}\right )
} . \nonumber
\end{align}
 \end{corollary}

The following Lemma extends Ammari's {\it et al}'s  Lemma 3.2~\cite{ammarivolkov2013}  to the case of $N$ multiple objects, when they are sufficiently well spaced.
\begin{lemma} \label{lemma:newlemma3_2}
Provided that $\displaystyle \min_{n,m=1,\ldots,N, n\ne m}|\partial (B_\alpha)^{(n)} -\partial (B_{\alpha})^{(m)} | \ge {\alpha}_{\max} $, there exists a constant $C$ such that
\begin{align*}
\| \nabla \times( {\vec E}_{\vec \alpha} - {\vec E}_0-{\vec w}^{(n)} ) \|_{L^2 \left ((B_\alpha)^{(n)} \right )} \le & C ( \nu_{\max} + |1- \mu_{r,\max}^{-1}| )\alpha_{\max}^{\frac{7}{2} } \| \nabla \times {\vec E}_0 \|_{W^{2,\infty}({\vec B}_{\vec \alpha})} ,\\
\| ( {\vec E}_{\vec \alpha} - {\vec E}_0-({\vec w}^{(n)} +{\vec \Phi}^{(n)} )) \|_{L^2\left ((B_\alpha)^{(n)} \right )} \le & C (\nu_{\max} + |1- \mu_{r,\max}^{-1}| )\alpha_{\max}^{\frac{9}{2} } \| \nabla \times {\vec E}_0 \|_{W^{2,\infty}({\vec B}_{\vec \alpha} )} ,
\end{align*}
for $n=1,\ldots,N$.
\end{lemma}

\begin{proof}
We start by proceeding along the lines presented in~\cite{ammarivolkov2013} and introduce \\
${\vec \Phi}^{(n)} = \left \{ \begin{array}{ll} \nabla \phi_0^{(n)} & \text{in $(B_\alpha)^{(n)}$} \\
\nabla \tilde{\phi}_0^{(n)} & \text{in ${((B_\alpha)^{(n)}})^c$} \end{array} \right .$ where
\begin{align*}
-\Delta \phi_0^{(n)} =&  - \nabla \cdot {\vec F}^{(n)} && \text{in $(B_\alpha)^{(n)} $} ,\\
-\partial_{\vec n} \phi_0^{(n)} = & ( {\vec E}_0({\vec x}) - {\vec F}^{(n)} ({\vec x})) \cdot {\vec n} && \text{on $(\Gamma_{\alpha})^{(n)}$}  , \\
\int_{(B_\alpha)^{(n)}} \phi_{\alpha^{(n)}} \dif {\vec x} = & 0 ,
\end{align*}
with $\tilde{\phi}_0^{(n)}$ being the solution of an exterior problem in an analogous way to $\tilde{\phi}_0$ in ~\cite{ammarivolkov2013}. Using (\ref{eqn:elimcurrt1}) and (\ref{eqn:elimcurrt2}) (and after multiplying by $\mu_0$) we can deduce that
\begin{align}
&A:= \left (\nabla \times  ( {\vec E}_{\vec \alpha} - {\vec E}_0 -   ({\vec w} +{\vec \Phi} )), \nabla \times {\vec v} \right )_{{\vec B}_{\vec \alpha}^c }+\nonumber\\
&  (\mu_0 \mu_{\vec \alpha}^{-1} \nabla \times  ({\vec E}_{\vec \alpha} -{\vec E}_0 - ({\vec w}_{\vec \alpha}+{\vec \Phi}_{\vec \alpha})), \nabla \times {\vec v})_{{\vec B}_{\vec \alpha}}   \nonumber \\
& -\im \omega\mu_0  (\sigma_{\vec \alpha}({\vec E}_{\vec \alpha} - {\vec E}_0 - ( {\vec w}_{\vec \alpha}+ {\vec \Phi}_{\vec \alpha})), {\vec v})_{{\vec B}_{\vec \alpha}}  \nonumber \\
&=  \mu_0((\mu_0^{-1}- \mu_{\vec \alpha}^{-1} ) \nabla \times ({\vec E}_0 -{\vec F}_{\vec \alpha}) ,\nabla \times {\vec v})_{{\vec B}_{\vec \alpha}} \nonumber \\
& + \im \omega \mu_0 ( \sigma_{\vec \alpha} ({\vec E}_0 +{\vec \Phi}_{\vec \alpha}- {\vec F}_{\vec \alpha}) , {\vec v})_{{\vec B}_{\vec \alpha}}  \nonumber\\
&+\sum_{n,m=1}^N ( \nabla \times ({\vec w}^{(m)}+{\vec \Phi}^{(m)}), \nabla \times {\vec v})_{(B_\alpha)^{(n)}}  (1-\delta_{mn})
\qquad  \forall {\vec v} \in \tilde{X}_\alpha \label{eqn:sum terms} ,
\end{align}
where ${\vec \Phi} = \sum_{n=1}^N {\vec \Phi}^{(n)}$ and ${\vec \Phi}_{\vec \alpha}= {\vec \Phi}^{(n)} $ in $(B_\alpha)^{(n)}$.
Choosing ${\vec v}= {\vec E}_{\vec \alpha}- {\vec E}_0 -   ({\vec w}_{\vec \alpha}+{\vec \Phi}_{\vec \alpha}))$ then we have that
\begin{align}
 \| \nabla \times ( {\vec E}_{\vec \alpha} - {\vec E}_0 - ({\vec w}^{(n)} + {\vec \Phi}^{(n)} ) ) \|_{L^2(  B_\alpha^{ (n) } )}^2  & \le  \| \nabla \times ( {\vec E}_{\vec \alpha} - {\vec E}_0 - ({\vec w}_{\vec \alpha} + {\vec \Phi}_{\vec \alpha} ) ) \|_{L^2( {\vec B}_{\vec \alpha} )}^2 \nonumber \\
 &  \le|A| \nonumber .
\end{align}
Also, by application of the Cauchy-Schwartz inequality, we can check that 
\begin{align}
 |A| \le A_1 +A_2 +A_3 , \label{eqn:sum terms2}
\end{align}
where
\begin{align}
A_1 := &\left | \mu_0 \left ( (\mu_0^{-1} - \mu_{\vec \alpha}^{-1}) (\nabla\times( {\vec E}_0 -{\vec F}_{\vec \alpha})) , \nabla \times {\vec v} \right )_{{\vec B}_{\vec \alpha}}
 \right | \nonumber \\
 \le &C  \max_{n=1,\cdots,N} | 1- (\mu_r^{(n)})^{-1} | \|    \nabla \times ({\vec E}_0 -{\vec F}_{\vec \alpha}))\|_{L^2( {\vec B}_{\vec \alpha} )}  \| \nabla \times {\vec v}\|_{L^2( {\vec B}_{\vec \alpha})  } \nonumber \\
 \le &C   |1- \mu_{r,\max}^{-1} | \left ( \sum_{n=1}^N \|    \nabla \times ({\vec E}_0 -{\vec F}^{(n)})\|_{L^2( B_\alpha^{(n)})}^2 \right )^{1/2}  \| \nabla \times {\vec v}\|_{L^2( {\vec B}_{\vec \alpha} )} \nonumber \\
\le & C  | 1- \mu_{r,\max}^{-1}| \alpha_{\max} ^{\frac{7}{2}} \| \nabla \times {\vec E}_0 \|_{W^{2,\infty}({\vec B}_{\vec \alpha}  )}     \| \nabla \times {\vec v}\|_{L^2( {\vec B}_{\vec \alpha} )} \label{eqn:a1}, \\
A_2 : = &  \left |  \omega \mu_0   (\sigma_*^{(n)}( {\vec E}_0 +{\vec \Phi}_{\vec \alpha}- {\vec F}_{\vec \alpha}),{\vec v}) _{{\vec B}_{\vec \alpha}} \right | \nonumber \\
\le  & C  \omega\mu_0 \sigma_{\max}  \left (  \sum_{n=1}^N \left (   (\alpha^{(n)}) \| \nabla \times ({\vec E}_0 - {\vec F}^{(n)} ) \|_{L^2((B_\alpha)^{(n)})} \right )^2 \right)^{1/2} \alpha_{\max} \| \nabla \times {\vec v} \|_{L^2( {\vec B}_{\vec \alpha})} \nonumber \\
\le & C  \nu_{\max} \alpha_{\max}^\frac{7}{2}   \left (   \sum_{n=1}^N \left (    \| \nabla \times {\vec E}_0   \|_{W^{2,\infty} ((B_\alpha)^{(n)})} \right )^2 \right )^{1/2} \| \nabla \times {\vec v} \|_{L^2(  {\vec B}_{\vec \alpha} )} \nonumber\\
\le & C     \nu_{\max}  \alpha_{\max}^\frac{7}{2}  \| \nabla \times {\vec E}_0   \|_{W^{2,\infty} ( {\vec B}_{\vec \alpha})} \| \nabla \times {\vec v} \|_{L^2(  {\vec B}_{\vec \alpha} )}  \label{eqn:a2} , \\ 
A_3  := & \left | \sum_{n,m=1}^N ( \nabla \times ({\vec w}^{(m)}+{\vec \Phi}^{(m)} ), \nabla \times {\vec v})_{B_\alpha^{(m)}}  (1-\delta_{mn}) \right | \nonumber \\
\le & C  \left ( \sum_{n,m=1}^N ( 1- \delta_{mn} )  \| \nabla \times {\vec w}^{(m)} \|_{L^2\left ((B_\alpha)^{(n)}\right  )} \right )  \| \nabla \times {\vec v} \|_{L^2({\vec B}_{\vec \alpha}  ) } .
\end{align}
To bound $A_1$ and $A_2$ we have used (\ref{eqn:taylorbackfd}), 
\begin{align}
\| {\vec E}_0 +{\vec \Phi}^{(n)} - {\vec F}^{(n)} \|_{L^2\left ((B_\alpha)^{(n)} \right  )} \le &C \alpha^{(n)} \|  \nabla \times ( {\vec E}_0 - {\vec F}^{(n)} ) \|_{L^2 \left ((B_\alpha)^{(n)} \right )} \nonumber \\
 \le & C ( \alpha^{(n)})^{\frac{9}{2}} \| \nabla \times {\vec E}_0 \|_{W^{2,\infty}\left ((B_\alpha)^{(n)} \right )} \label{e0pphi0nmfn} ,
\end{align}
and applied  similar arguments to~\cite{ammarivolkov2013}. 
The terms  $A_3$ does not appear in the single object case and dictates the minimum spacing for which the bound holds. Requiring that $\displaystyle |{\vec z}^{(m)}- {\vec z}^{(n)} |= \min_{n,m=1,\ldots,N, n\ne m}|\partial B^{(n)} -\partial B^{(m)}| > C > \alpha_{\max}$ and applying Corollary~\ref{coll:closelyspacedw}
then
\begin{align}
A_3  \le  C\alpha_{\max}^{\frac{7}{2}} \| \nabla \times {\vec E}_0   \|_{W^{2,\infty} ({\vec B}_{\vec \alpha}) }  \| \nabla \times {\vec v} \|_{L^2( {\vec B}_{\vec \alpha} )} \label{eqn:a3} .
\end{align}
Using (\ref{eqn:a1}), (\ref{eqn:a2}) and \ref{eqn:a3}) in (\ref{eqn:sum terms2}) we find that
\begin{align}
\| \nabla \times ( {\vec E}_{\vec \alpha} - {\vec E}_0 - ({\vec w}^{(n)} + {\vec \Phi}^{(n)} ) ) \|_{L^2( B_\alpha^{ (n) } )} \le & \| \nabla \times ( {\vec E}_{\vec \alpha} - {\vec E}_0 - ({\vec w}_{\vec \alpha} + {\vec \Phi}_{\vec \alpha} ) ) \|_{L^2(  {\vec B}_{\vec \alpha} )} \nonumber \\
\le & C \left ( \nu_{max} + | 1- \mu_{r,\max}^{-1} | \right )  \alpha_{\max}^{\frac{7}{2}} \| \nabla \times {\vec E}_0   \|_{W^{2,\infty}(  {\vec B}_{\vec \alpha} )},
\nonumber
\end{align}
and, by additionally using $ \| {\vec E}_{\vec \alpha} - {\vec E}_0 - ({\vec w}^{(n)} + {\vec \Phi}^{(n)} )  \|_{L^2( B_\alpha^{ (n) } )} \le \alpha^{(n)} \| \nabla \times ( {\vec E}_{\vec \alpha} - {\vec E}_0 - ({\vec w}^{(n)} + {\vec \Phi}^{(n)} ) ) \|_{L^2( B_\alpha^{ (n) } )} \le \alpha_{\max} \| \nabla \times ( {\vec E}_{\vec \alpha} - {\vec E}_0 - ({\vec w}^{(n)} + {\vec \Phi}^{(n)} ) ) \|_{L^2( B_\alpha^{ (n) } )} $, this completes the proof.
\end{proof}

By recalling the definition of ${\vec w}_0^{(n)} ({\vec \xi})$ stated in Lemma~\ref{lemma:closelyspacedw},  
Ammari's {\it et al}'s  Theorem 3.1~\cite{ammarivolkov2013} in the case of multiple sufficiently well spaced objects becomes
\begin{theorem} \label{thm:rev3_1}
Provided that $\displaystyle \min_{n,m=1,\ldots,N, n\ne m}|\partial (B_\alpha)^{(n)} -\partial (B_\alpha)^{(m)} | \ge {\alpha}_{\max} $
there exists a constant $C$ such that
\begin{align*}
\left \| \nabla \times \left ( {\vec E}_{\vec \alpha} - {\vec E}_0-\alpha^{(n)} {\vec w}_0^{(n)} \left ( \frac{{\vec x} - {\vec z}^{(n)} }{ \alpha^{(n)}} \right ) \right  ) \right \|_{L^2\left ((B_\alpha)^{(n)}\right )} \le & 
 C ( \nu_{\max} + |1- \mu_{r,\max}^{-1}| ) \nonumber \\
& \alpha_{\max}^{\frac{7}{2} } \| \nabla \times {\vec E}_0 \|_{W^{2,\infty}(    {\vec B}_{\vec \alpha}  )} ,\\
\left \|  {\vec E}_{\vec \alpha} - {\vec E}_0-\left (\alpha^{(n)} {\vec w}_0^{(n)} \left ( \frac{{\vec x} - {\vec z}^{(n)} }{ \alpha^{(n)}} \right ) +{\vec \Phi}^{(n)} \right ) \right  \|_{L^2\left ((B_\alpha)^{(n)}\right )} \le & C (\nu_{\max} + |1- \mu_{r,\max}^{-1}| ) \nonumber \\
& \alpha_{\max}^{\frac{9}{2} } \| \nabla \times {\vec E}_0 \|_{W^{2,\infty}( {\vec B}_{\vec \alpha} )} .
\end{align*}
\end{theorem}
\begin{proof}
The result immediately follows from Lemma~\ref{lemma:newlemma3_2} and the definition of ${\vec w}_0^{(n)}$.
\end{proof}

The expressions for $\alpha^{(n)}  {\vec F}^{(n)} ({\vec z}^{(n)} + \alpha^{(n)} {\vec \xi}^{(n)})$ and ${\vec w}_0({\vec \xi}^{(n)})$ are obtained by extending in an obvious way the expressions in
given in (3.13) and (3.14) in ~\cite{ammarivolkov2013} where the latter is now written in terms of $({\vec H}_0({\vec z}^{(n)}))_i {\vec \theta}_i^{(n)}({\vec \xi}^{(n)})$ as well as $({\vec D}_z({\vec H}_0({\vec z}))( {\vec z}^{(n)}))_{ij} {\vec \psi}_{ij}^{(n)}({\vec \xi}^{(n)})$ where ${\vec \theta}_i^{(n)}({\vec \xi}^{(n)})$ and ${\vec \psi}_{ij}^{(n)}({\vec \xi}^{(n)})$ satisfy the transmission problems
\begin{subequations} \label{eqn:thetatp}
\begin{align}
\nabla_\xi  \times  (\mu_*^{(n)}) ^{-1} \nabla_\xi \times {\vec \theta}_i^{(n)} - \im \omega \sigma_*^{(n)}  (\alpha^{(n)})^2   {\vec \theta}_i^{(n)} = &  \im \omega \sigma_*^{(n)} (\alpha^{(n)})^2  {\vec e}_i \times {\vec \xi}^{(n)} && \text{in $B^{(n)}$} , \\
\nabla_\xi \times \mu_0^{-1} \nabla_\xi \times {\vec \theta}_i^{(n)} = & {\vec 0} &&\text{in $(B^{(n)})^c $} , \\
\nabla_\xi \cdot {\vec \theta}_i^{(n)} = & 0 &&\text{in $(B^{(n)})^c $} , \\
\left [ {\vec n} \times {\vec \theta}_i^{(n)}  \right ]_{\Gamma^{(n)}}  = &  {\vec 0} &&\text{on $\Gamma^{(n)}$ } , \\
 \left [ {\vec n} \times \mu^{-1} \nabla_\xi \times {\vec \theta}_i^{(n)} \right ]_{\Gamma^{(n)}} = &  -2 [\mu^{-1}]_{\Gamma^{(n)}} {\vec n} \times {\vec e}_i && \text{on $\Gamma^{(n)}$ } , \\
 {\vec \theta}_i^{(n)} =& O( | {\vec \xi}^{(n)}|^{-1} ) && \text{as $|{\vec \xi}^{(n)}| \to \infty $} ,
\end{align}
 \end{subequations}
and
\begin{subequations} \label{eqn:psitp}
\begin{align}
\nabla_\xi  \times  (\mu_*^{(n)}) ^{-1} \nabla_\xi \times {\vec \psi}_{ij}^{(n)} - \im \omega \sigma_*^{(n)}  (\alpha^{(n)})^2   {\vec \psi}_{ij}^{(n)} = &  \im \omega \sigma_*^{(n)} (\alpha^{(n)})^2  \xi_j^{(n)} {\vec e}_i \times {\vec \xi}^{(n)} && \text{in $B^{(n)}$} , \\
\nabla_\xi \times \mu_0^{-1} \nabla_\xi \times {\vec \psi}_{ij}^{(n)} = & {\vec 0} &&\text{in $(B^{(n)})^c $} , \\
\nabla_\xi \cdot {\vec \psi}_{ij}^{(n)} = & 0 &&\text{in $( B^{(n)})^c $} , \\
\left [ {\vec n} \times {\vec \psi}_{ij}^{(n)}  \right ]_{\Gamma^{(n)}}  = &  {\vec 0} &&\text{on $\Gamma^{(n)}$ } , \\
 \left [ {\vec n} \times \mu^{-1} \nabla_\xi \times {\vec \psi}_{ij}^{(n)} \right ]_{\Gamma^{(n)}} = &  -3 [\mu^{-1}]_{\Gamma^{(n)}} \xi_j^{(n)} {\vec n} \times {\vec e}_i && \text{on $\Gamma^{(n)}$ } , \\
 {\vec \psi}_{ij}^{(n)} =& O( | {\vec \xi}^{(n)}|^{-1} ) && \text{as $|{\vec \xi}^{(n)}| \to \infty $} .
\end{align}
 \end{subequations}
The properties of ${\vec \theta}_i^{(n)}({\vec \xi}^{(n)})$ and ${\vec \psi}_{ij}^{(n)}({\vec \xi}^{(n)})$ are analogues to the single object case presented in~\cite{ammarivolkov2013}.

\subsection{Integral Representation Formulae}  \label{sect:intrepf}
Repeating the proof of Lemma 3.3 in ~\cite{ammarivolkov2013}  for the multiple object case, it extends in an obvious way to
\begin{lemma}
Let $D= D^{(1)}\cup D^{(2)} \cup \ldots \cup D^{(N)}$ be the union of $N$ bounded domains each with Lipschitz boundaries $\Gamma_D^{(n)}$ whose outer normal is ${\vec n}$. For any ${\vec E}\in{\vec H}_{-1} (  \hbox{\emph{ curl}}; {\mathbb R}^3 \setminus \overline{D})$ satisfying  $\nabla \times \nabla \times {\vec E} = {\vec 0}$, $\nabla \cdot {\vec E} = 0$ in ${\mathbb R}^3 \setminus \overline{D}$, we have, for any ${\vec x} \in {\mathbb R}^3 \setminus \overline{D}$
\begin{align*}
{\vec E}({\vec x}) = \sum_{n=1}^N  &\left ( - \nabla_x \times \int_{\Gamma_D^{(n)}} ( {\vec E}({\vec y}) \times {\vec n} ) G({\vec x},{\vec y}) \dif {\vec y} - \int_{\Gamma_D^{(n)}} \nabla_y \times ( {\vec E}({\vec y}) \times {\vec n}) G ({\vec x}, {\vec y}) \dif {\vec y} \right . \nonumber \\
& \left . - \nabla_x  \int_{\Gamma_D^{(n)}} ({\vec E}({\vec y}) \cdot {\vec n} )  G ({\vec x}, {\vec y}) \dif {\vec y} \right ).
\end{align*}
\end{lemma}
In a similar way, repeating the proof of their Lemma 3.4  for multiple objects it extends in an obvious way to
\begin{lemma} \label{lemma:rev3_4}
Let $\tilde{\vec H}_{\vec \alpha} = {\vec H}_{\vec \alpha} - {\vec H}_0$. Then for ${\vec x} \in {\vec B}_{\vec \alpha}^c $
\begin{align*}
( {\vec H}_{\vec \alpha} - {\vec H}_0 ) ({\vec x}) = & \sum_{n=1}^N \left ( \int_{(B_\alpha)^{(n)} } \nabla_x G({\vec x},{\vec y}) \times \nabla_y \times  \tilde{\vec H}_{\vec \alpha} ({\vec y} ) \dif {\vec y}\right . \nonumber \\
&\left .+ \left ( 1 - \frac{\mu_*^{(n)}}{\mu_0 } \right )\int_{(B_\alpha)^{(n)}}   ( {\vec H}_{\vec \alpha} ({\vec y}) \cdot \nabla_y )  \nabla_x G({\vec x},{\vec y}) \dif {\vec y} \right ) .
\end{align*}

\end{lemma}

\subsection{Asymptotic Formulae}
Theorem 3.2 in~\cite{ammarivolkov2013} presents the leading order term in asymptotic expansion for $({\vec H}_\alpha -{\vec H}_0)({\vec x})$ for a single inclusion $B_\alpha$ as $\alpha \to 0$. In the case of multiple objects that are sufficiently well spaced this extends to
\begin{theorem} \label{thm:rev3_2}
For a collection of $N$ objects such that $\nu^{(n)}$ is order one, $\alpha^{(n)}$ is small and $\min_{n,m=1,\ldots,N, n\ne m}|\partial (B_\alpha)^{(n)} -\partial (B_\alpha)^{(m)}| > C > \alpha_{\max}$ then for ${\vec x}$ away from ${\vec B}_{\vec \alpha}  $ we have
\begin{align}
({\vec H}_{\vec \alpha} &- {\vec H}_0)({\vec x}) = \sum_{n=1}^N \left (  - \frac{ \im \nu^{(n)} \alpha^{(n)} }{2} \sum_{i=1}^3 ( {\vec H}_0({\vec z}^{(n)}))_i \int_{B^{(n)}}
{\vec D}_x^2 G( {\vec x}, {\vec z}^{(n)}) {\vec \xi}^{(n)}  \times ( {\vec \theta}_i^{(n)} + {\vec e}_i \times {\vec \xi}^{(n)} ) \dif {\vec \xi}^{(n)} \right . \nonumber  \\
&+ \left .(\alpha^{(n)})^3 \left ( 1 - \frac{\mu_0}{\mu_*^{(n)}} \right ) \sum_{i=1}^3  ( {\vec H}_0({\vec z}^{(n)}))_i  {\vec D}_x^2 G({\vec x}, {\vec z}^{(n)} )  \int_{B^{(n)} } 
\left ( {\vec e}_i + \frac{1}{2} \nabla \times {\vec \theta}_i^{(n)}  \right ) \dif {\vec \xi}^{(n)} \right ) \nonumber \\
 &+ {\vec R}({\vec x}) \label{eqn:mainasym} ,
 \end{align}
 where ${\vec \theta}_i^{(n)}$ is the solution of (\ref{eqn:thetatp}) and 
 \begin{equation}
 |{\vec R}({\vec x}) | \le C \alpha_{\max}^4 \| {\vec H}_0 \|_{W^{2, \infty}( {\vec B}_{\vec \alpha} )} ,
 \end{equation}
 uniformly in ${\vec x}$ in any compact set away from ${\vec B}_{\vec \alpha}$. 
\end{theorem}

\begin{proof}
The proof uses as its starting point Lemma~\ref{lemma:rev3_4} and considers each object $(B_\alpha)^{(n)}$  in turn. It applies very similar arguments to the proof of  Theorem 3.2 in ~\cite{ammarivolkov2013}  except Theorem~\ref{thm:rev3_1} is used in place of their Theorem 3.1, (\ref{e0pphi0nmfn}) is used in place of their (3.6) and note that
\begin{align}
 \sigma_*^{(n)} \int_{(B_\alpha)^{(n)}} \nabla_x G({\vec x},{\vec z}^{(n)}) \times \left ( {\vec F}^{(n)} ( {\vec y}) + \alpha^{(n)}  {\vec w}_0^{(n)}  \left ( \frac{{\vec y}-{\vec z}^{(n)} }{\alpha^{(n)}} \right ) \right ) \dif {\vec y}={\vec 0} ,
\end{align}
by integration by parts. Furthermore, to recover the negative sign in the first term in (\ref{eqn:mainasym}), we have used 
\begin{equation}
\nabla_x G({\vec x},\alpha^{(n)}{\vec \xi}^{(n)} + {\vec z}^{(n)})  =  \nabla_x G({\vec x},{\vec z}^{(n)})  - \alpha^{(n)} {\vec D}_x^2 G({\vec x},{\vec z}^{(n)}) {\vec \xi} ^{(n)}+ O((\alpha^{(n)})^2 ) ,
\end{equation}
as $\alpha^{(n)}\to 0$.
Theorem 3.2 in~\cite{ammarivolkov2013} mistakingly uses $ \nabla_x G({\vec x},\alpha{\vec \xi} + {\vec z})  =  \nabla_x G( {\vec x},{\vec z} )  + \alpha {\vec D}_x^2 G({\vec x},{\vec z}) {\vec \xi} + O(\alpha^2)$ as $\alpha \to 0$, which leads to the wrong sign in their first term, as previously reported for the single homogenous object case in~\cite{ledgerlionheart2014}.
\end{proof}

\section{Results for the Proof of Theorem~\ref{thm:objectscloselyspaced}} \label{sect:proof2}
Recall that in this case the object is inhomogeneous and is arranged as ${\vec B}_\alpha := \bigcup_{n=1}^N B_\alpha^{(n)} =\alpha  \bigcup_{n=1}^N  B^{(n)} + {\vec z}= \alpha {\vec B} +{\vec z}$ 
 where  ${\alpha}$ is a single small scaling parameter and ${\vec z}$
a single translation.

\subsection{Elimination of the Current Source} 
The results presented in Section~\ref{sect:elimcurrent} hold also in the case when the object is inhomogeneous except the subscript ${\vec \alpha}$ is replaced by $\alpha$.
\subsection{Energy Estimates}
For an inhomogeneous object, we proceed along similar lines as~\cite{ammarivolkov2013}  and introduce a single vector field $ {\vec F} ({\vec x})$ whose curl is such that it is equal to the first two terms of  a Taylor series of $\im \omega \mu_0 {\vec H}_0({\vec x})$ expanded about $ {\vec z}$ as $|{\vec x}- {\vec z}| \to 0$
\begin{align*}
{\vec F} ({\vec x}) = & \frac{\im \omega \mu_0}{2} {\vec H}_0 ( {\vec z}) \times ({\vec x}-{\vec z}) + \frac{\im \omega \mu_0}{3} {\vec D}_{z} ( {\vec H}_0 ({\vec z})) ({\vec x}-{\vec z})  \times   ({\vec x}-{\vec z}) , \\
\nabla \times {\vec F} ({\vec x}) = &\im \omega \mu_0  \left ( {\vec H}_0 ( {\vec z})+ {\vec D}_{z} ( {\vec H}_0({\vec z})) ( {\vec z}) ({\vec x}- {\vec z}) \right ),
\end{align*}
so that
\begin{align}
\| \im \omega \mu_0 {\vec H}_0 ({\vec x}) - \nabla \times {\vec F}  \|_{L^\infty( {\vec B}_\alpha )} \le & C  {\alpha}^2 \| \nabla \times {\vec E}_0 \|_{W^{2,\infty}({\vec B}_\alpha  )}  ,  \nonumber \\
\| \im \omega \mu_0 {\vec H}_0 ({\vec x}) - \nabla \times {\vec F}  \|_{L^2({\vec B}_\alpha )} \le & C {\alpha}^{\frac{3}{2}} \| \im \omega \mu_0 {\vec H}_0 ({\vec x}) - \nabla \times {\vec F}  \|_{L^\infty ({\vec B}_\alpha ) }   \nonumber \\ 
\le & C {\alpha}  ^{\frac{7}{2}} \| \nabla \times {\vec E}_0 \|_{W^{2,\infty}({\vec B}_\alpha )} . \label{eqn:taylorbackfd2}
\end{align}
The introduction of $ {\vec F} ({\vec x})$ motivates the introduction of the following problem: Find $ {\vec w} \in \tilde{\vec X}_\alpha$ such that
\begin{align}
(\mu_0^{-1} & \nabla \times {\vec w} , \nabla \times {\vec v})_{{\vec B}_\alpha^c  }+(\mu_{\alpha}^{-1} \nabla \times {\vec w} , \nabla \times {\vec v})_{{\vec B}_\alpha}  -\im \omega (\sigma_{\alpha}  {\vec w} , {\vec v})_{{\vec B}_\alpha} \nonumber \\
&= ((\mu_0^{-1}- (\mu_*^{(n)})^{-1} ) \nabla \times {\vec F} ,\nabla \times {\vec v})_{{\vec B}_\alpha } + \im \omega ( \sigma_*^{(n)} {\vec F} , {\vec v})_{{\vec B}_\alpha} \qquad \forall {\vec v} \in \tilde{\vec X}_\alpha \label{eqn:wmnweakform2} .
\end{align}

The following Lemma extends Ammari's {\it et al}'s  Lemma 3.2~\cite{ammarivolkov2013}  to the case of an inhomogeneous object.
\begin{lemma} \label{lemma:newlemma3_2v2}
For an inhomogeneous object  ${\vec B}_\alpha $ , there exists a constant $C$ such that
\begin{align}
\| \nabla \times( {\vec E}_\alpha - {\vec E}_0- {\vec w} ) \|_{L^2(  B_\alpha^{(m)})} \le & C ( \nu_{\max} + |1- \mu_{r,\max}^{-1}| ) {\alpha}^{\frac{7}{2} } \| \nabla \times {\vec E}_0 \|_{W^{2,\infty}( {\vec B}_\alpha )}\\
\| ( {\vec E}_\alpha - {\vec E}_0-( {\vec w} + {\vec \Phi}  )) \|_{L^2(  B_\alpha^{(m)})} \le & C (\nu_{\max} + |1- \mu_{r,\max}^{-1}| ) {\alpha}^{\frac{9}{2} } \| \nabla \times {\vec E}_0 \|_{W^{2,\infty}({\vec B}_\alpha )}
\end{align}
for $m=1,\ldots,N$.
\end{lemma}
\begin{proof}
Here we introduce ${\vec \Phi} = \left \{ \begin{array}{ll} \nabla {\phi}_0 & \text{in ${\vec B}_\alpha $} \\
\nabla \tilde{\phi}_0^{(n)} & \text{in ${\vec B}_\alpha^c $} \end{array} \right .$ where
\begin{align*}
-\Delta {\phi}_0 =&  - \nabla \cdot {\vec F} && \text{in $ {\vec B}_\alpha  $} , \\
-\partial_{\vec n} {\phi}_0 = & ( {\vec E}_0({\vec x}) - {{\vec F}} ({\vec x})) \cdot {\vec n} && \text{on $\partial {\vec B}_\alpha $} , \\
\int_{B_\alpha } {\phi}_0 \dif {\vec x} = & 0 ,
\end{align*}
with $\tilde{\phi}_0$ being the solution of an exterior problem in an analogous way to ~\cite{ammarivolkov2013}. Then, by writing
\begin{align}
&\left (\nabla \times {\vec E}_\alpha - {\vec E}_0 -  ( {\vec w}+ {\vec \Phi} ), \nabla \times {\vec v} \right )_{{\vec B}_\alpha^c }
-\im \omega\mu_0  (\sigma_\alpha ({\vec E}_\alpha - {\vec E}_0 - ( {\vec w} + {\vec \Phi} ), {\vec v})_{{\vec B}_\alpha }
\nonumber \\
&+ (\mu_0 \mu_\alpha^{-1} \nabla \times  ({\vec E}_\alpha -{\vec E}_0 - ({\vec w}+{\vec \Phi} ), \nabla \times {\vec v})_{{\vec B}_\alpha }   \nonumber \\
&=  \mu_0((\mu_0^{-1}- \mu_\alpha^{-1} ) \nabla \times ({\vec E}_0 - {\vec F} ) ,\nabla \times {\vec v})_{{\vec B}_\alpha }  + \im \omega \mu_0 ( \sigma_\alpha ({\vec E}_0 +{\vec \Phi} - {\vec F}  , {\vec v})_{{\vec B}_\alpha } , \nonumber
\end{align}
and proceeding with similar steps to~\cite{ammarivolkov2013}, where $B_\alpha$ is replaced by ${\vec B}_\alpha $, we have
\begin{align}
\| {\vec E}_0 + {\vec \Phi} - {\vec F} \|_{L^2( B_\alpha^{(m)})} & \le 
\| {\vec E}_0 + {\vec \Phi} - {\vec F} \|_{L^2({\vec B}_\alpha )} \nonumber \\
&\le 
C \alpha  \| \nabla \times ( {\vec E}_0 - {\vec F} ) \|_{L^2({\vec B}_\alpha  )}  \le C \alpha^{9/2}
\| \nabla \times {\vec E}_0 \|_{W^{2,\infty} ( {\vec B}_\alpha )}  \label{eqn:e0phimf} ,
\end{align}
for $m=1,\ldots,N$ and
\begin{align*}
\| \nabla \times( {\vec E}_\alpha - {\vec E}_0- {\vec w} ) \|_{L^2({\vec B}_\alpha )} \le & C ( \nu_{\max} + |1- \mu_{r,\max}^{-1}| ) {\alpha}^{\frac{7}{2} } \| \nabla \times {\vec E}_0 \|_{W^{2,\infty}({\vec B}_\alpha  )} ,\\
\| ( {\vec E}_\alpha - {\vec E}_0-( {\vec w} + {\vec \Phi}  )) \|_{L^2({\vec B}_\alpha  )} \le & C (\nu_{\max} + |1- \mu_{r,\max}^{-1}| ) {\alpha}^{\frac{9}{2} } \| \nabla \times {\vec E}_0 \|_{W^{2,\infty}( {\vec B}_\alpha  )} .
\end{align*}
Finally, we use $\| \nabla \times( {\vec E}_\alpha - {\vec E}_0- {\vec w} ) \|_{L^2( {\vec B}_\alpha  )} \le \| \nabla \times( {\vec E}_\alpha - {\vec E}_0- {\vec w} ) \|_{L^2({\vec B}_\alpha  )}$ and $\| ( {\vec E}_\alpha - {\vec E}_0-( {\vec w} + {\vec \Phi}  )) \|_{L^2( B_\alpha^{(n)})}\le \| ( {\vec E}_\alpha - {\vec E}_0-( {\vec w} + {\vec \Phi}  )) \|_{L^2({\vec B}_\alpha )}$, which holds for $n=1,\ldots,N$.
\end{proof}

Introducing, $ {\vec w} ({\vec x}) = {\alpha} {{\vec w}}_{0} \left ( \frac{{\vec x} -  {{\vec z}} }{ {\alpha} } \right )  ={\alpha} {\vec w}_{0} ({\vec \xi})$ so that $\nabla_ x \times {\vec w} ({\vec x}) = \nabla_\xi \times {\vec w}_0 ({\vec \xi} ) =   \nabla_\xi \times {\vec w}_0  \left (\frac{{\vec x} - {\vec z}}{{\alpha}} \right )$ we find that ${\vec w}_0 ({\vec \xi} )$ satisfies
\begin{subequations} \label{eqn:w0nstrongformv2}
\begin{align}
\nabla_\xi  \times  \mu ^{-1} \nabla_\xi \times {\vec w}_0 - \im \omega \sigma  {\vec w}_0 = &  \im \omega \sigma {\alpha}^2  && {} \nonumber  \\
&( {\alpha}^{-1} {\vec F} ( {\vec z} + {\alpha}  {\vec \xi} )) && \text{in ${\vec B}$} ,  \\
\nabla_\xi \times \mu_0^{-1} \nabla_\xi \times {\vec w}_0 = & {\vec 0} &&\text{in ${\vec B}^c$} , \\
\nabla_\xi \cdot {{\vec w}}_0  = & 0 &&\text{in ${\vec B}^c $} , \\
\left [ {\vec n} \times {\vec w}_0   \right ]_{\Gamma }  = &  {\vec 0} &&\text{on $\Gamma$ } , \\
 \left [ {\vec n} \times \mu^{-1} \nabla_\xi \times {\vec w}_0  \right ]_{\Gamma} = &  -[ \mu^{-1} ]_\Gamma && {} \nonumber \\
 & {\vec n} \times \nabla \times {\vec F}  ( {\vec z} + {\alpha} {\vec \xi}   ) &&\text{on $\Gamma$ } , \\
 {\vec w}_0 =& O( | {\vec \xi} |^{-1} ) && \text{as $| {\vec \xi} | \to \infty $} .
\end{align}
 \end{subequations}
where, for an inhomogeneous object, $\Gamma := \partial {\vec B} \cup \{ \partial B^{(n) } \cap \partial B^{ (n)}, n,m=1,\ldots, N, n \ne m \}$.

In this case, Ammari {\it et al}'s  Theorem 3.1~\cite{ammarivolkov2013} becomes
\begin{theorem} \label{thm:rev3_1v2}
There exists a constant $C$ such that
\begin{align*}
\left \| \nabla \times \left ( {\vec E}_\alpha - {\vec E}_0-{\alpha} {\vec w}_0 \left ( \frac{{\vec x} - {\vec z}  }{ {\alpha}} \right ) \right  ) \right \|_{L^2(B_\alpha^{(m)})} \le & 
 C ( \nu_{\max} + |1- \mu_{r,\max}^{-1}| ) \nonumber \\
& {\alpha}^{\frac{7}{2} } \| \nabla \times {\vec E}_0 \|_{W^{2,\infty}( {\vec B}_\alpha )} , \\
\left \|  {\vec E}_\alpha - {\vec E}_0-\left ( {\alpha}  {\vec w}_0  \left ( \frac{{\vec x} -{\vec z} }{ {\alpha} }\right ) +{\vec \Phi} \right ) \right  \|_{L^2( B_\alpha^{(m)})} \le & C (\nu_{\max} + |1- \mu_{r,\max}^{-1}| ) \nonumber \\
& {\alpha}^{\frac{9}{2} } \| \nabla \times {\vec E}_0 \|_{W^{2,\infty}( {\vec B}_\alpha )} ,
\end{align*}
for $m=1,\ldots,N$,
which holds for an inhomogeneous object ${\vec B}_\alpha$.
\end{theorem}
\begin{proof}
The result immediately follows from Lemma~\ref{lemma:newlemma3_2v2} and the definition of ${\vec w}_0$.
\end{proof}

The expressions for ${\alpha} {\vec F}( {\vec z} + {\alpha} {\vec \xi} )$ and ${\vec w}_0( {\vec \xi} )$ are identical to (3.13) and (3.14) stated in~\cite{ammarivolkov2013} where $ {\vec \theta}_i ( {\vec \xi})$ and ${\vec \psi}_{ij}( {\vec \xi})$ now satisfy the transmission problems
\begin{subequations} \label{eqn:thetatpv2}
\begin{align}
\nabla_\xi  \times  \mu ^{-1} \nabla_\xi \times {\vec \theta}_i - \im \omega \sigma {\alpha}^2   {\vec \theta}_i = &  \im \omega \sigma {\alpha}^2  {\vec e}_i \times {\vec \xi} && \text{in ${\vec B}$} , \\
\nabla_\xi \times \mu_0^{-1} \nabla_\xi \times {\vec \theta}_i = & {\vec 0} &&\text{in $ {\vec B}^c$} , \\
\nabla_\xi \cdot {\vec \theta}_i = & 0 &&\text{in ${\vec B}^c$} , \\
\left [ {\vec n} \times {\vec \theta}_i  \right ]_{\Gamma}  = &  {\vec 0} &&\text{on $\Gamma$ } , \\
 \left [ {\vec n} \times \mu^{-1} \nabla_\xi \times {\vec \theta}_i  \right ]_{\Gamma} = &  -2 [\mu^{-1}]_{\Gamma} {\vec n} \times {\vec e}_i && \text{on $\Gamma$ } , \\
 {\vec \theta}_i  =& O( | {\vec \xi} |^{-1} ) && \text{as $|{\vec \xi} | \to \infty $}  ,
\end{align}
 \end{subequations}
and
\begin{subequations} \label{eqn:psitpv2}
\begin{align}
\nabla_\xi  \times  \mu ^{-1} \nabla_\xi \times {\vec \psi}_{ij} - \im \omega \sigma  {\alpha}^2  {\vec \psi}_{ij} = &  \im \omega \sigma {\alpha}^2  {\xi}_j {\vec e}_i \times {\vec \xi}  && \text{in ${\vec B}$}  , \\
\nabla_\xi \times \mu_0^{-1} \nabla_\xi \times {\vec \psi}_{ij} = & {\vec 0} &&\text{in ${\vec B}^c $} , \\
\nabla_\xi \cdot {\vec \psi}_{ij} = & 0 &&\text{in ${\vec B}^c $} , \\
\left [ {\vec n} \times {\vec \psi}_{ij}   \right ]_{\Gamma}  = &  {\vec 0} &&\text{on $\Gamma$ } ,  \\
 \left [ {\vec n} \times \mu^{-1} \nabla \times {\vec \psi}_{ij}  \right ]_{\Gamma} = &  -3 [\mu^{-1}]_{\Gamma} {\xi}_j {\vec n} \times {\vec e}_i && \text{on $\Gamma$ } , \\
{\vec \psi}_{ij}  =& O( | {\vec \xi} |^{-1} ) && \text{as $|{\vec \xi} | \to \infty $} .
\end{align}
 \end{subequations}
The properties of ${\vec \theta}_i ( {\vec \xi} )$ and ${\vec \psi}_{ij} ({\vec \xi} )$ are analogues to the homogeneous object case presented in~\cite{ammarivolkov2013}.

\subsection{Integral Representation Formulae}  
The integral representation formulae presented in Section~\ref{sect:intrepf} only require $(B_\alpha)^{(n)}$ to be replaced by $B_\alpha^{(n)}$ and ${\vec H}_{\vec \alpha}$ to be replaced by ${\vec H}_\alpha$ for an inhomogeneous object.

\subsection{Asymptotic Formulae}
Theorem 3.2 in~\cite{ammarivolkov2013} presents the leading order term in asymptotic expansion for $({\vec H}_\alpha -{\vec H}_0)({\vec x})$ for a single homogenous inclusion $B_\alpha= \alpha B + {\vec z}$ as $\alpha \to 0$. In the case of an inhomogeneous inclusion this becomes
 \begin{theorem} \label{thm:rev3_2v2}
For an inhomogeneous object  ${\vec B}_\alpha$ 
such that ${\nu}^{(n)}$ is order one and ${\alpha}$ is small then for ${\vec x}$ away from ${\vec B}_\alpha $, we have
\begin{align}
({\vec H}_\alpha - {\vec H}_0)({\vec x}) =  &  - \frac{ \im  {\alpha}  }{2} \sum_{i=1}^3 ( {\vec H}_0({\vec z}))_i \sum_{n=1}^N {\nu}^{(n)} \int_{ {B}^{(n)} }
{\vec D}_x^2 G( {\vec x}, {\vec z} ) {\vec \xi}  \times ( {\vec \theta}_i + {\vec e}_i \times {\vec \xi}  ) \dif {\vec \xi} \nonumber  \\
&+ 
{\alpha}^3 \sum_{i=1}^3  ( {\vec H}_0({\vec z} ))_i  {\vec D}_x^2 G({\vec x}, {\vec z} ) \sum_{n=1}^N \left ( 1 - \frac{\mu_0}{\mu_*^{(n)}} \right )
 \int_{{B}^{(n)}  }  \left ( {\vec e}_i + \frac{1}{2} \nabla \times {\vec \theta}_i \right ) \dif {\vec \xi}  \nonumber \\
 &+ {\vec R}({\vec x}) \label{eqn:mainasymv2} ,
 \end{align}
 where $ {\vec \theta}_i$ is the solution of (\ref{eqn:thetatpv2}) and 
 \begin{equation}
 |{\vec R}({\vec x}) | \le C {\alpha}^4 \| {\vec H}_0 \|_{W^{2, \infty}( {\vec B}_\alpha )} ,
 \end{equation}
 uniformly in ${\vec x}$ in any compact set away from $ {\vec B}_\alpha$. 
\end{theorem}
\begin{proof}
The proof uses as its starting point Lemma~\ref{lemma:rev3_4} and considers each region $B_\alpha^{(n)}$ in turn. It applies similar arguments to the proof of their Theorem 3.2 except that our Theorem~\ref{thm:rev3_1v2} is used in place of their Theorem 3.1 and our (\ref{eqn:e0phimf}) instead of their (36).  Furthermore, note that by summing contributions, we have that
\begin{align}
\sum_{n=1}^N \sigma_*^{(n)} \int_{B_\alpha^{(n)}} \nabla_x G({\vec x},{\vec z}) \times \left ( {\vec F}({\vec y}) + \alpha {\vec w}_0 \left ( \frac{{\vec y}-{\vec z}}{\alpha} \right ) \right ) \dif {\vec y}={\vec 0},
\end{align}
by application of  integration by parts and, in a similar manner to the proof of Theorem~\ref{thm:rev3_2}, we use
\begin{equation}
\nabla_x G({\vec x},\alpha{\vec \xi} + {\vec z})  =  \nabla_x G({\vec x},{\vec z})  - \alpha {\vec D}_x^2 G({\vec x},{\vec z}) {\vec \xi} + O(\alpha^2),
\end{equation}
to give the correct negative sign in the first term of (\ref{eqn:mainasymv2}).
\end{proof}

\section{Numerical Examples and Algorithms for Object Localisation and Identification} \label{sect:num}
In this section we consider an illustrative numerical application of the asymptotic formulae (\ref{eqn:asymp2}) and (\ref{eqn:asymp3}), numerical examples of the frequency spectra of the MPT coefficients and propose algorithms for multiple object localisation and inhomogeneous object identification  as extensions of those in~\cite{ammarivolkov2013b}.
\subsection{Numerical Illustration of Asymptotic Formulae for $({\vec H}_\alpha - {\vec H}_0)({\vec x})$}
To illustrate the results in Theorems~\ref{thm:objectsnocloselyspaced} and~\ref{thm:objectscloselyspaced}, comparisons of $({\vec H}_\alpha - {\vec H}_0)({\vec x})$~\footnote{We use  $({\vec H}_{ \alpha}- {\vec H}_0)({\vec x})$ instead of  $({\vec H}_{\vec \alpha}- {\vec H}_0)({\vec x})$  for Theorem~\ref{thm:objectsnocloselyspaced} throughout this section as the examples  with multiple objects presented have the same object size.} will be  undertaken with a finite element method (FEM) solver~\cite{ledgerzaglmayr2010}  for multiple objects and for inhomogeneous objects. 
 We first show comparisons for two spheres, then comparisons for two tetrahedra followed by comparisons for an inhomogeneous parallelepiped.
\subsubsection{Two Spheres} \label{sect:twosph}
We first consider the situation of two spheres $(B_\alpha)^{(1)}$ and $(B_\alpha)^{(2)}$. These objects are defined as
 \begin{align*}
(B_\alpha)^{(1)} : =&\left \{{\vec x}:  \left (x_1 -\frac{d\alpha^{(1)}}{2}-\alpha^{(1)} \right )^2 + x_2^2 +x_3^2 =( \alpha^{(1)})^2 \right \}, \\
(B_\alpha)^{(2)} : = &\left \{{\vec x}:  \left  (x_1 +\frac{d\alpha^{(2)}}{2}+\alpha^{(2)} \right )^2 + x_2^2 +x_3^2 =( \alpha^{(2)})^2 \right \},
\end{align*}
which means that the radii of the objects are $\alpha^{(1)}$ and $\alpha^{(2)}$, respectively. Setting $B=B^{(1)}=B^{(2)}$ to be a sphere of unit radius placed at the origin then
 \begin{align*}
{\vec z}^{(1)} = &\left (-\frac{d\alpha^{(1)}}{2}-\alpha^{(1)} \right ) {\vec e}_1 +0{\vec e}_2+0{\vec e}_3 ,\qquad
{\vec z}^{(2)} = & \left (\frac{d\alpha^{(2)}}{2}+\alpha^{(2)} \right ){\vec e}_1 +0{\vec e}_2+0{\vec e}_3 ,
\end{align*}
are  the location of the centroids of the physical objects $B_\alpha^{(1)}$ and  $B_\alpha^{(2)}$, respectively.
Thus, the objects $(B_\alpha)^{(n)}$, $n=1,2,$ are centered about the origin with $\min | \partial (B_\alpha)^{(1)} - \partial (B_\alpha)^{(2)} | = \alpha d$.
 The material properties of the spheres are $\sigma_*^{(1)}=\sigma_*^{(2)} = 5.66 \times 10^7 \text{S/m}$, $\mu_*^{(1)}=\mu_*^{(2)}=\mu_0$, we use $\omega =133.5 \text{rad/s}$ and the object sizes are chosen as $\alpha=\alpha^{(1)}=\alpha^{(2)}=0.01 \text{m}$ and hence ${\mathcal M}[\alpha^{(1)} B^{(1)}]={\mathcal M}[\alpha^{(2)} B^{(2)}]$, independent of their separation, which will be used in Theorem~\ref{thm:objectsnocloselyspaced}. For closely spaced objects we expect Theorem~\ref{thm:objectscloselyspaced} to be applicable and in this case we set
\begin{align*}
{\vec B}=\bigcup_{n=1}^2 B^{(n)}  =&  \left  \{ {\vec x}: \left (x_1 -\frac{d}{2}-1 \right )^2 + x_2^2 +x_3^2 =1 \right \}\cup \\
					&   \left  \{ {\vec x}: \left (x_1 +\frac{d}{2}+1 \right )^2 + x_2^2 +x_3^2 =1 \right \},
\end{align*}
and ${\vec z}={\vec 0}$. Note that in this case, ${\mathcal M}\left [{\alpha}{\vec B} \right ]$ must be recomputed for each new $d$.

Comparisons of $ ({\vec H}_\alpha-{\vec H}_0)({\vec x})$ obtained from the asymptotic formulae (\ref{eqn:asymp2}) and (\ref{eqn:asymp3}) in Theorems~\ref{thm:objectsnocloselyspaced} and~\ref{thm:objectscloselyspaced} as well as a full FEM solution are made in Figure~\ref{fig:femtwospheres} for $d=0.2$ and $d=2$ along three different coordinates axes. To ensure the tensor coefficients were calculated accurately, a $p=3$ edge element discretisation and an unstructured mesh of $6\,581$ tetrahedra is used for computing  
${\mathcal M}[\alpha^{(1)} B^{(1)}]={\mathcal M}[\alpha^{(2)} B^{(2)}]$ and meshes of $8\,950$ and $11\,940$ 
unstructured tetrahedral elements are used for computing ${\mathcal M}\left [{\alpha}{\vec B} \right ]$ for $d=0.2$ and $d=2$, respectively. In addition, curved elements with a quadratic geometry resolution are used for representing the curved surfaces of the spheres. For these, and all subsequent examples, the artificial truncation boundary was set to be $100|B|$. To ensure an accurate representation of $({\vec H}_\alpha-{\vec H}_0)({\vec x})$ for the FEM solver, the same discretisation, suitably scaled, as used for ${\mathcal M}\left [{\alpha}{\vec B} \right ]$ is employed.

For the closely spaced objects, with $d=0.2$, we observe good agreement between Theorem~\ref{thm:objectscloselyspaced} and the FEM solution in Figure~\ref{fig:femtwospheres}, with all three results tending to the same result for sufficiently large $|{\vec x}|$. The improvement for larger $|{\vec x}|$ is expected as the asymptotic formulae (\ref{eqn:asymp2}) and (\ref{eqn:asymp3}) are valid for ${\vec x}$ away from ${\vec B}_{\vec \alpha} \equiv {\vec B}_\alpha$. For objects positioned further apart, with $d=2$, we observe that the agreement between Theorem~\ref{thm:objectsnocloselyspaced} and the FEM solution is best. This agrees with what our theory predicts, since, for $d=2$, $\min | \partial (B_\alpha)^{(1)} - \partial (B_\alpha)^{(2)} | = 2 \alpha > \alpha_{\max}$ and so this theorem applies.

\begin{figure}
\begin{center}
$\begin{array}{cc}
\includegraphics[width=0.5\textwidth]{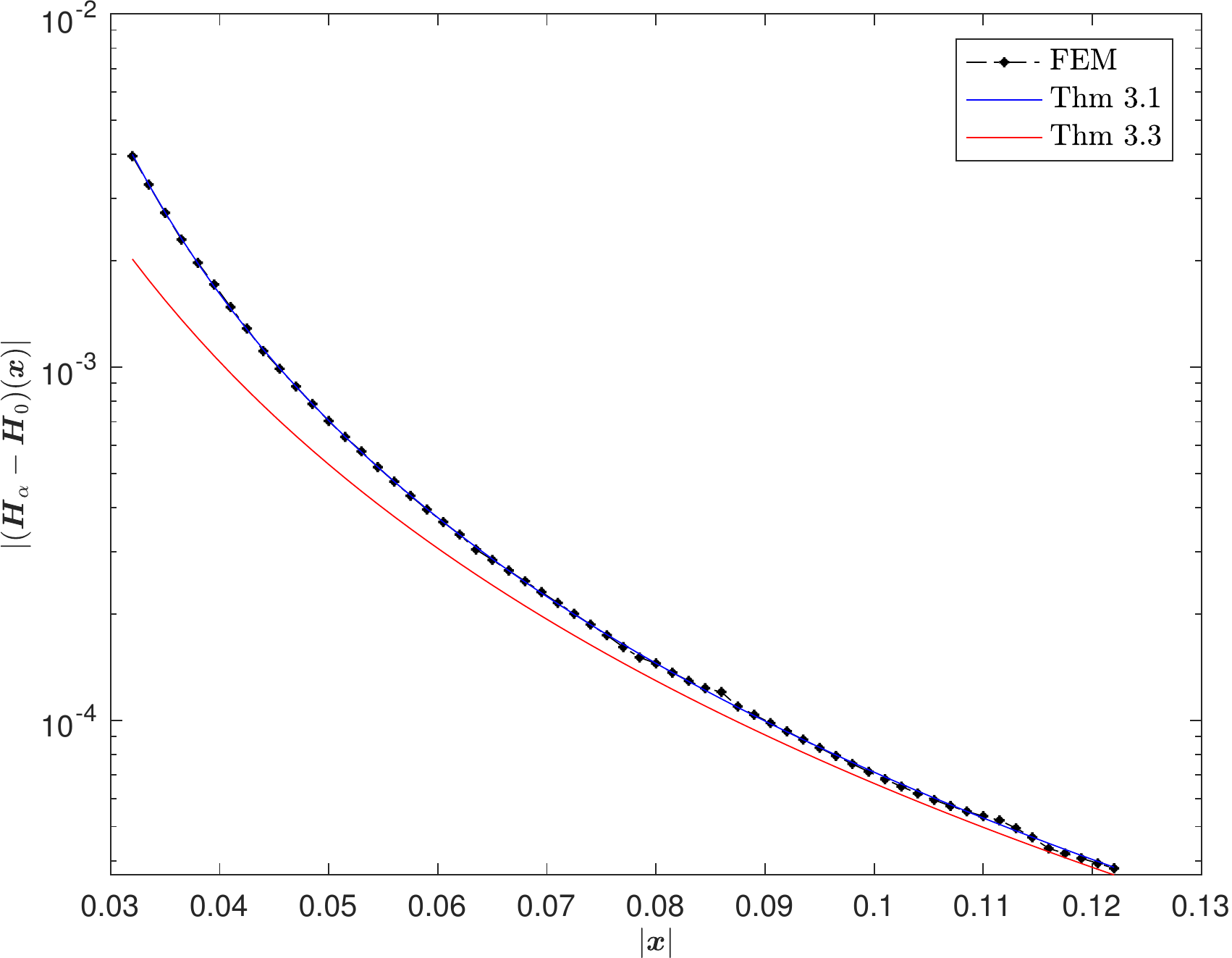} &
\includegraphics[width=0.5\textwidth]{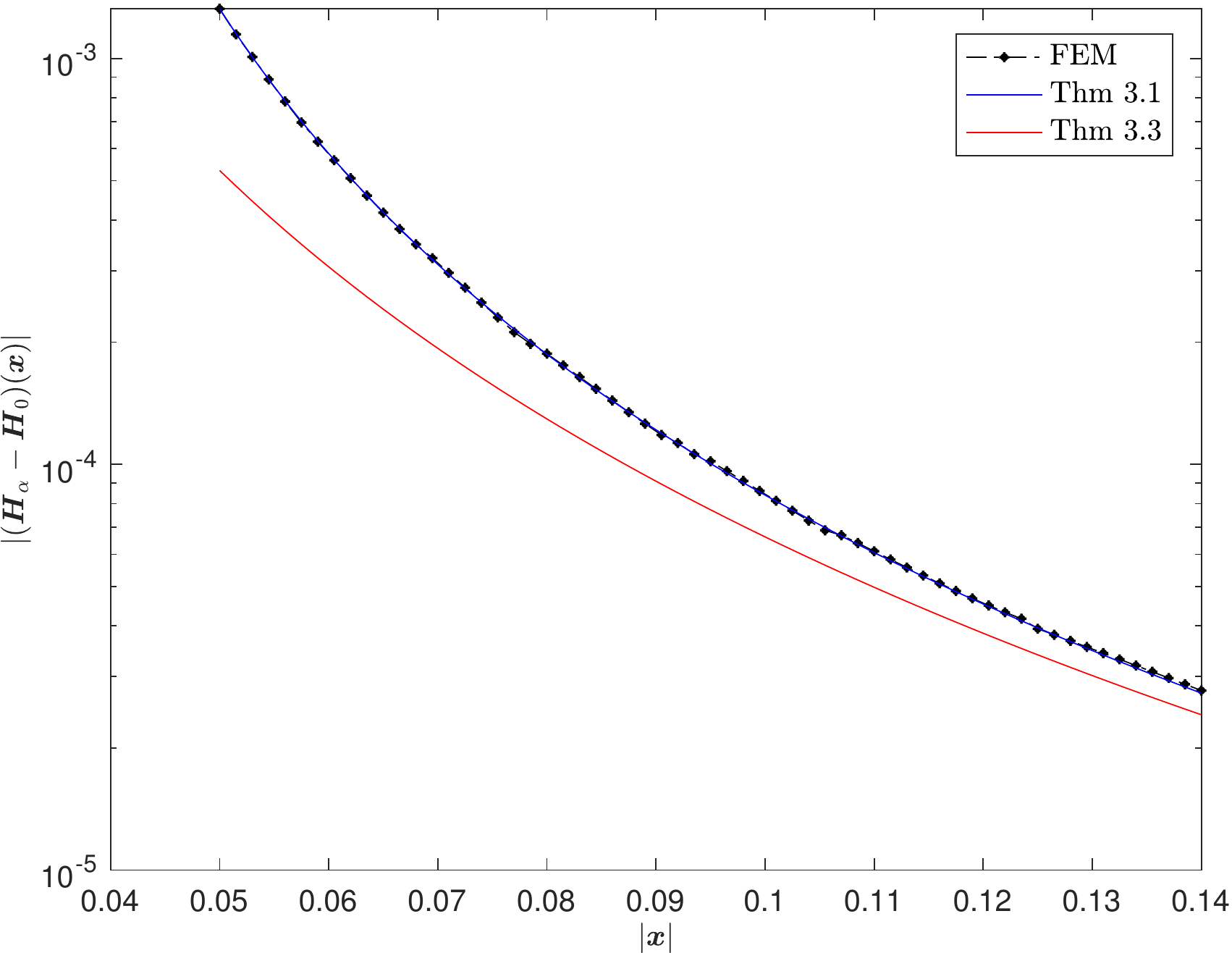} \\
d=0.2, {\vec x}= x_1 {\vec e}_1  & d=2, {\vec x}= x_1 {\vec e}_1 \\
\includegraphics[width=0.5\textwidth]{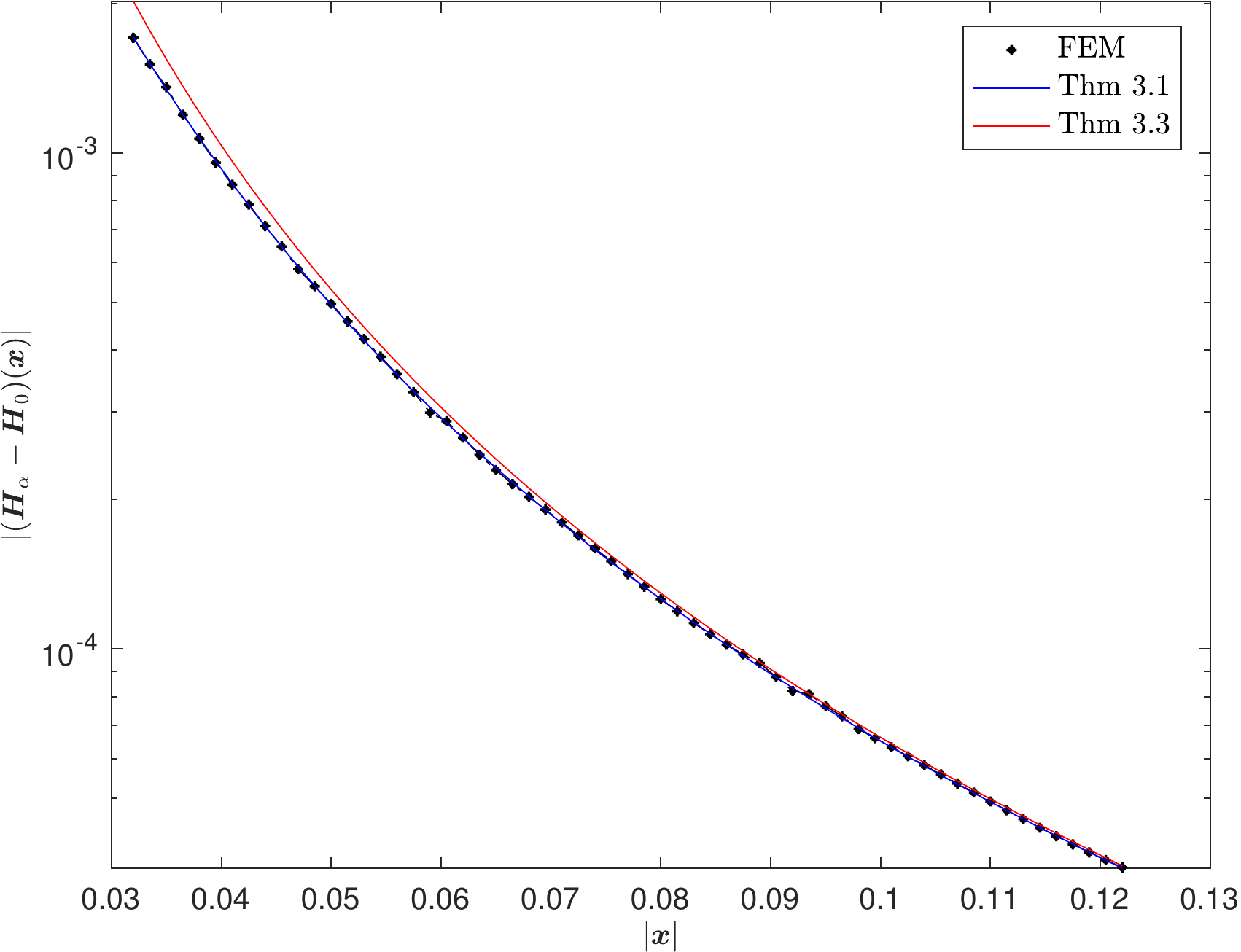} &
\includegraphics[width=0.5\textwidth]{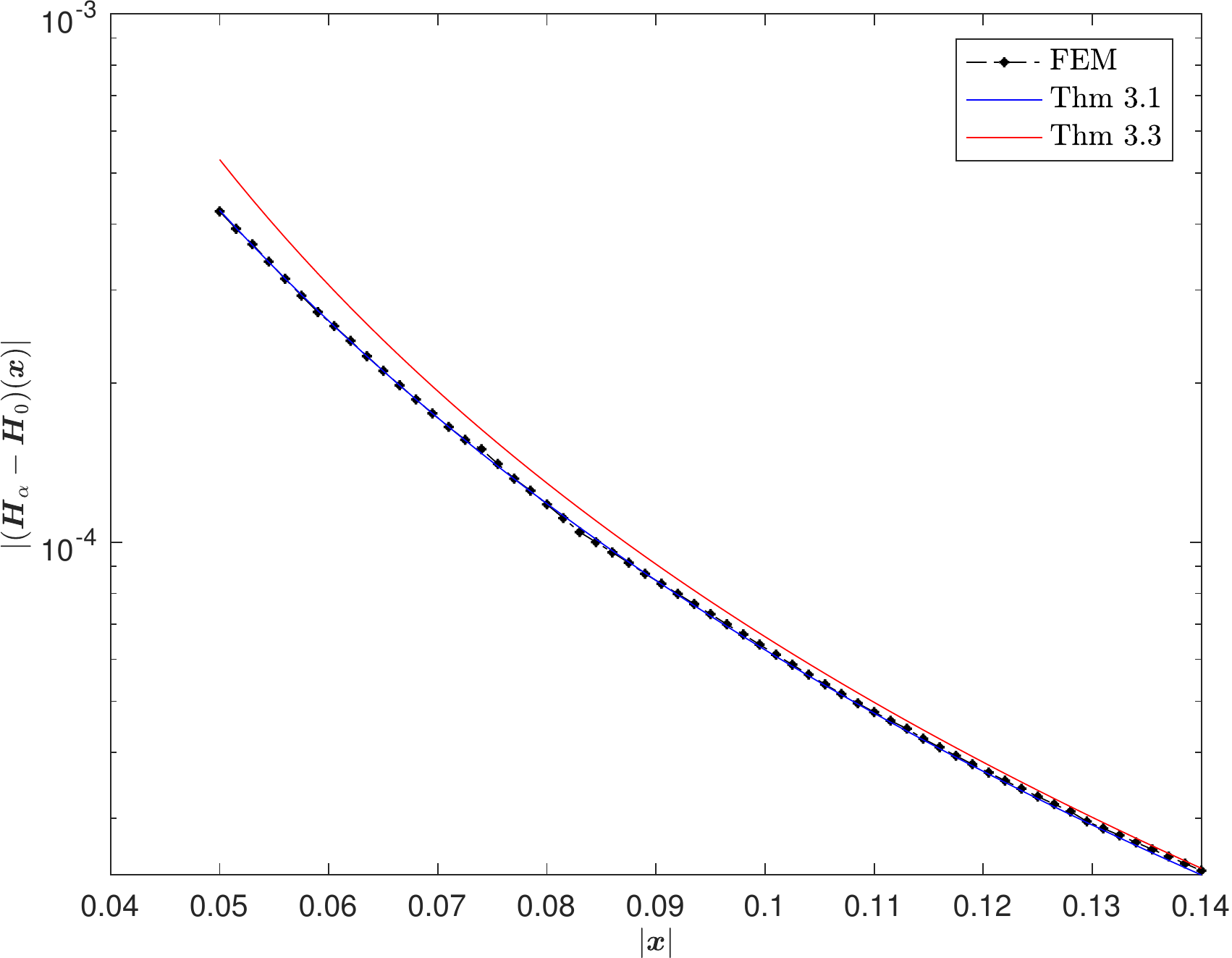} \\
d=0.2, {\vec x}= x_2 {\vec e}_2 & d=2, {\vec x}= x_2 {\vec e}_2 \\
\includegraphics[width=0.5\textwidth]{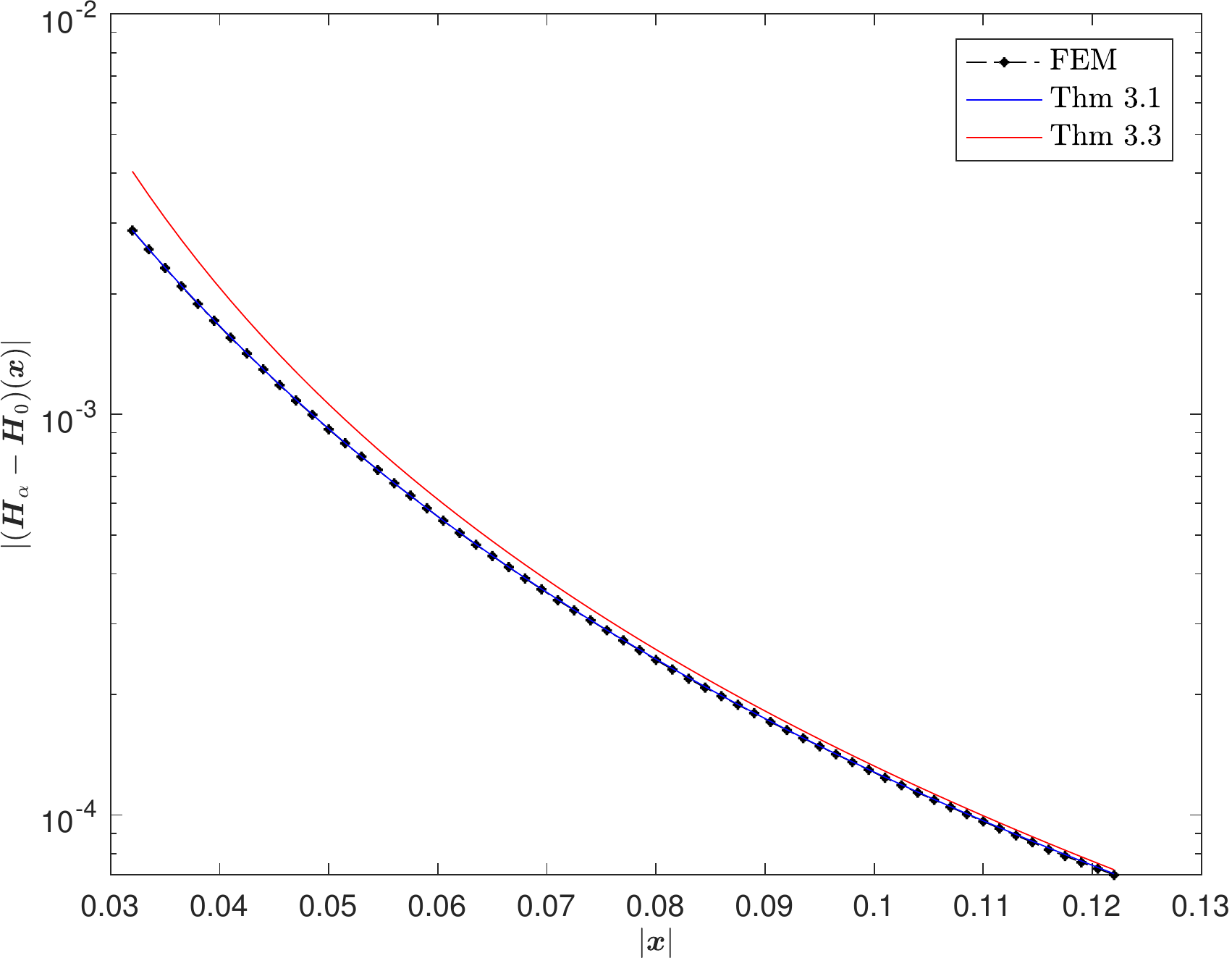} &
\includegraphics[width=0.5\textwidth]{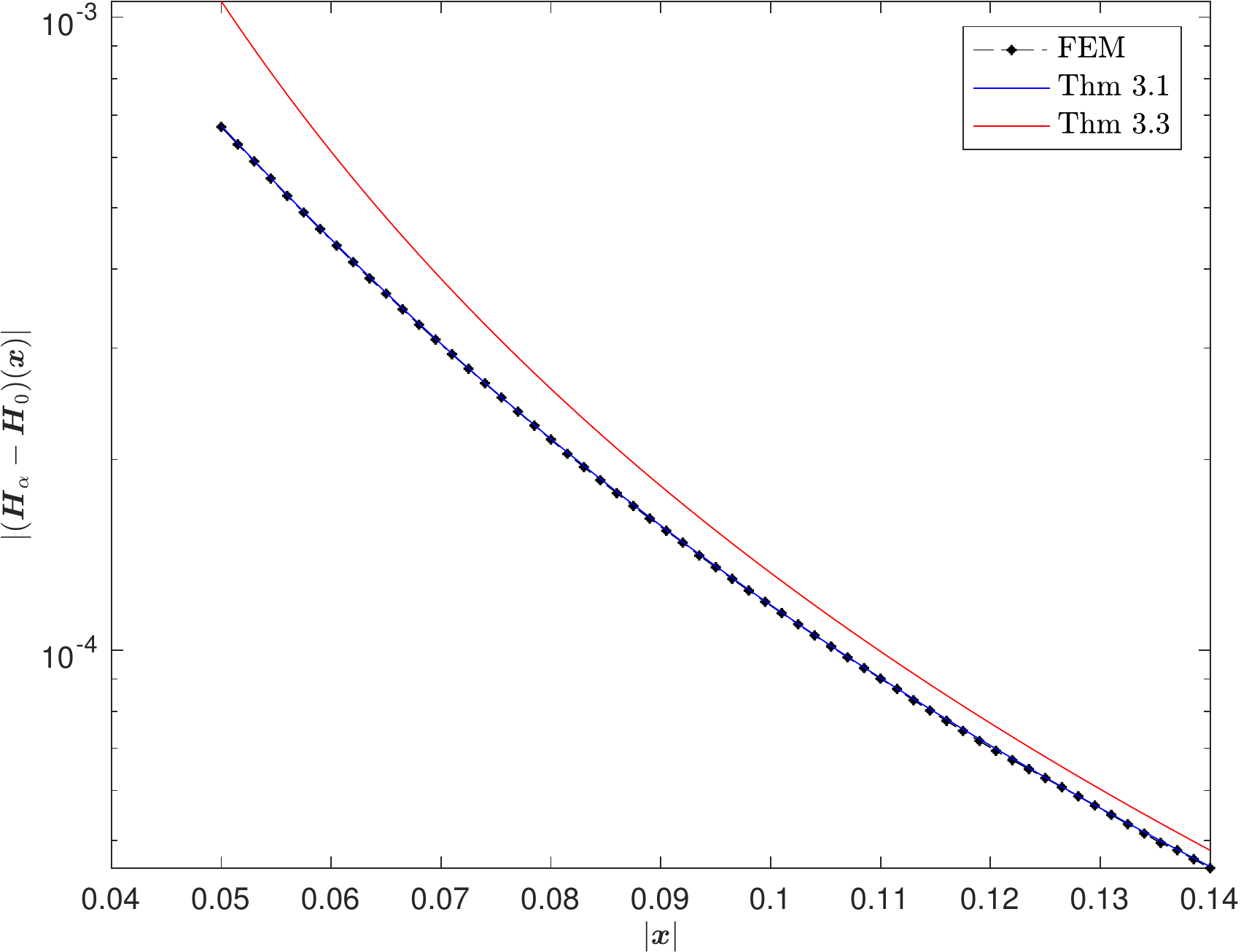} \\
d=0.2, {\vec x}= x_3 {\vec e}_3 & d=2, {\vec x}= x_3 {\vec e}_3
\end{array}$ 
\end{center}
\caption{Comparison of $({\vec H}_\alpha-{\vec H}_0)({\vec x})$ using the asymptotic expansions (\ref{eqn:asymp2})  and (\ref{eqn:asymp3})  in 
Theorems \ref{thm:objectsnocloselyspaced} and \ref{thm:objectsnocloselyspaced} as well as
a FEM solution: along the three coordinate axes for two spheres with different separations $\alpha d$.} \label{fig:femtwospheres}
\end{figure}

\subsubsection{Two Tetrahedra}
Next, we consider the case of two tetrahedra where the physical objects $(B_\alpha)^{(1)}$ and $(B_\alpha)^{(2)}$ are chosen as the tetrahedra with vertices $  (-1-\frac{d}{2}, -\frac{3}{8}, -\frac{1}{4})$, 
$  (-\frac{d}{2}, -\frac{3}{8}, -\frac{1}{4})$, $  (-\frac{d}{2}, \frac{5}{8}, -\frac{1}{4})$, $  (-\frac{d}{2}, \frac{1}{8}, \frac{3}{4})$ and
 $ (\frac{d}{2}, -\frac{3}{8}, -\frac{1}{4})$, 
$  (1+\frac{d}{2}, -\frac{3}{8}, -\frac{1}{4})$, $  (\frac{d}{2}, \frac{5}{8}, -\frac{1}{4})$, $ (\frac{d}{2}, \frac{1}{8}, \frac{3}{4})$, scaled by $\alpha^{(1)}$ and $\alpha^{(2)}$ respectively. Thus, the objects $(B_\alpha)^{(n)}$, $n=1,2,$ are centered about the origin with $\min | \partial (B_\alpha)^{(1)} - \partial (B_\alpha)^{(2)} | = \alpha d$ and we determine $B^{(n)}$ from $(B_\alpha)^{(n)} = \alpha^{(n)} B^{(n)} + {\vec z}^{(n)}$  by setting
\begin{equation*}
{\vec z}^{(1)} = -\alpha^{(1)}\left ( \frac{1}{4} + \frac{d}{2} \right) {\vec e}_1,\qquad {\vec z}^{(2)} = \alpha^{(2)}\left ( \frac{1}{4} + \frac{d}{2} \right) {\vec e}_1,
\end{equation*}
such that the centroid of $B^{(n)}$ lies at the origin.
A typical illustration of the two tetrahedra is shown in Figure \ref{fig:twotetrahedra:case2}. 
\begin{figure}[!htbp]
\centering
    \includegraphics[scale=0.2]{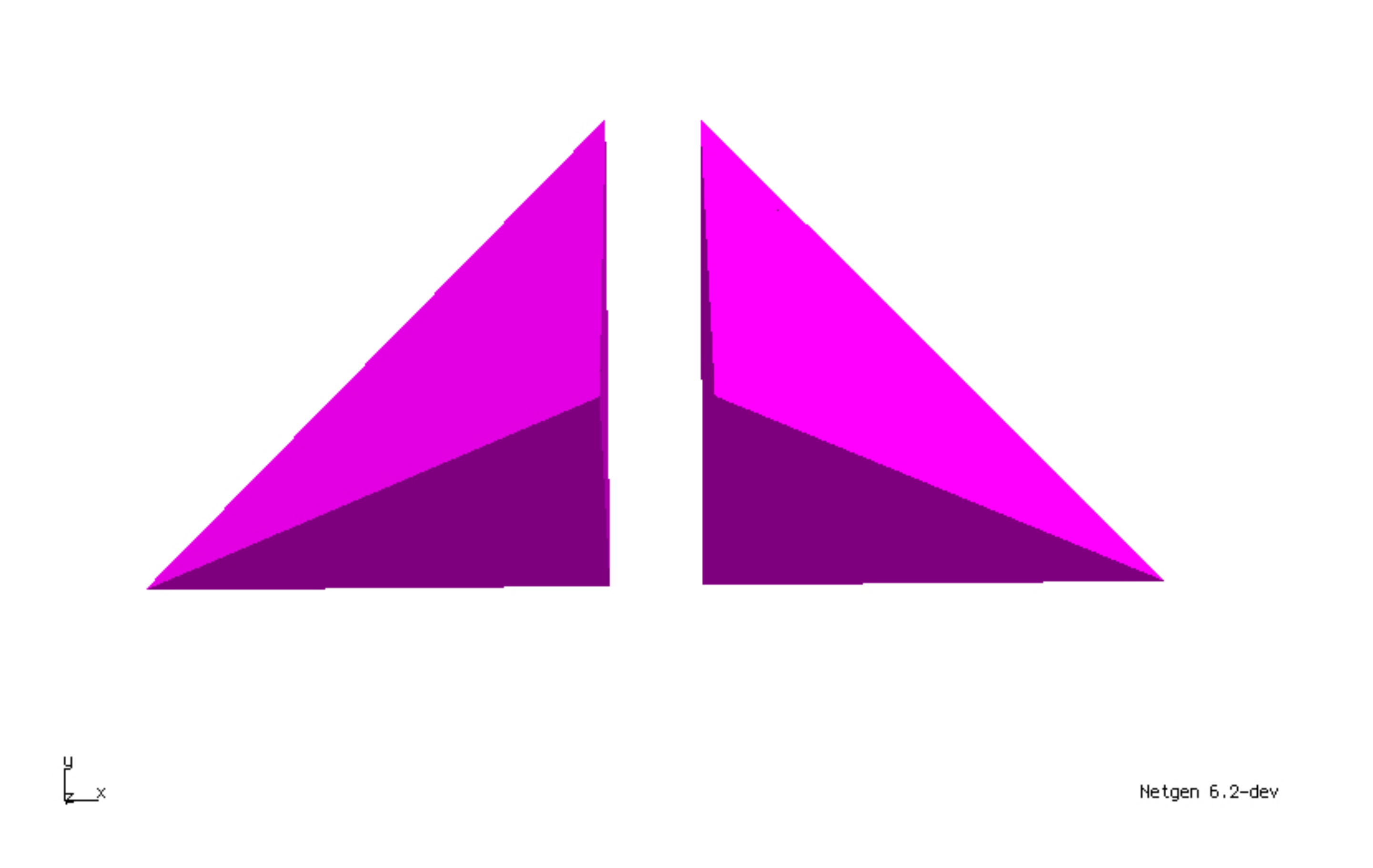}
\caption{Two tetrahedra $(B_\alpha)^{(1)}$ and $(B_\alpha)^{(2)}$ with $\min | \partial (B_\alpha)^{(1)} - \partial (B_\alpha)^{(2)} | =  \alpha d$.}
\label{fig:twotetrahedra:case2}
\end{figure}
The sizes and materials of $(B_\alpha)^{(1)}$ and $(B_\alpha)^{(2)}$ are both the same, as in the previous section, but ${\mathcal M}[\alpha^{(1)} B^{(1)}] \ne{\mathcal M}[\alpha^{(2)} B^{(2)}]$ due to their different shapes, although the MPTs are independent of $d$. However, note that  $(B_\alpha)^{(2)} = \alpha^{(2)} R^x((B_\alpha)^{(1)}) / \alpha^{(1)} $ and $ B^{(2)} = M^x(B^{(1)}) $, where 
\begin{equation*}
M^x =\left ( \begin{array}{ccc} -1 & 0 & 0\\
0 & 1 & 0 \\
0 & 0 & 1 \end{array} \right ), 
\end{equation*}
and since $\alpha=\alpha^{(1)}=\alpha^{(2)}$ the tensor coefficients transform as
\begin{equation}
({\mathcal M}[\alpha^{(2)} B^{(2)} ])_{ij} = (M^x)_{ip} (M^x)_{jq} ( {\mathcal M}[\alpha^{(1)} B^{(1)}])_{pq} . \label{eqn:transtet}
\end{equation}
For ${\vec B}= B^{(1)}\cup B^{(2)} $ we instead choose $B^{(1)}= (B_\alpha)^{(1)}/\alpha^{(1)}$, $B^{(2)}= (B_\alpha)^{(2)} /\alpha^{(2)}$ and set ${\vec z}={\vec 0}$.

Comparisons of  $ ({\vec H}_\alpha-{\vec H}_0)({\vec x})$ for this case are made in Figure~\ref{fig:femtwotet} for $d=0.2$ and $d=2$ along three different coordinates axes. To ensure the tensor coefficients are calculated accurately, a $p=3$ edge element discretisation and unstructured meshes of $15\,617$ and $15\,488$  tetrahedra are used for computing  
${\mathcal M}[\alpha^{(1)} B^{(1)}]$ and ${\mathcal M}[\alpha^{(2)} B^{(2)}]$~\footnote{${\mathcal M}[\alpha^{(2)} B^{(2)}]$ could be alternatively obtained from ${\mathcal M}[\alpha^{(1)} B^{(1)}]$ by applying (\ref{eqn:transtet}).}, respectively, and meshes of $15\,837$ and $22\,045$ unstructured tetrahedral elements are used  for computing ${\mathcal M}\left [{\alpha}{\vec B} \right ]$ for $d=0.2$ and $d=2$, respectively.  To ensure an accurate representation of $({\vec H}_\alpha-{\vec H}_0)({\vec x})$ for the FEM solver, the same discretisation, suitably scaled, as used for ${\mathcal M}\left [{\alpha}{\vec B} \right ]$ is employed.

\begin{figure}
\begin{center}
$\begin{array}{cc}
\includegraphics[width=0.5\textwidth]{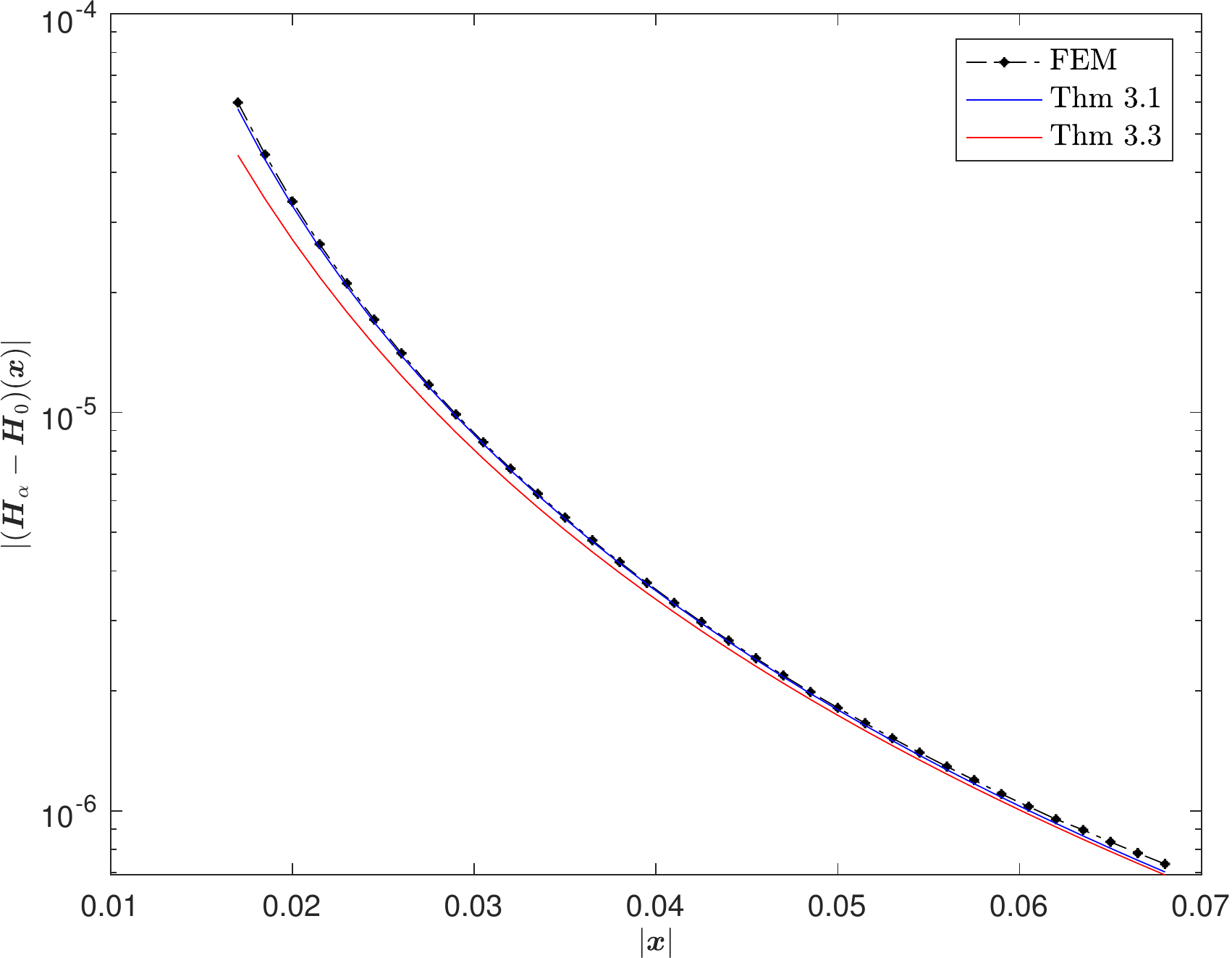} &
\includegraphics[width=0.5\textwidth]{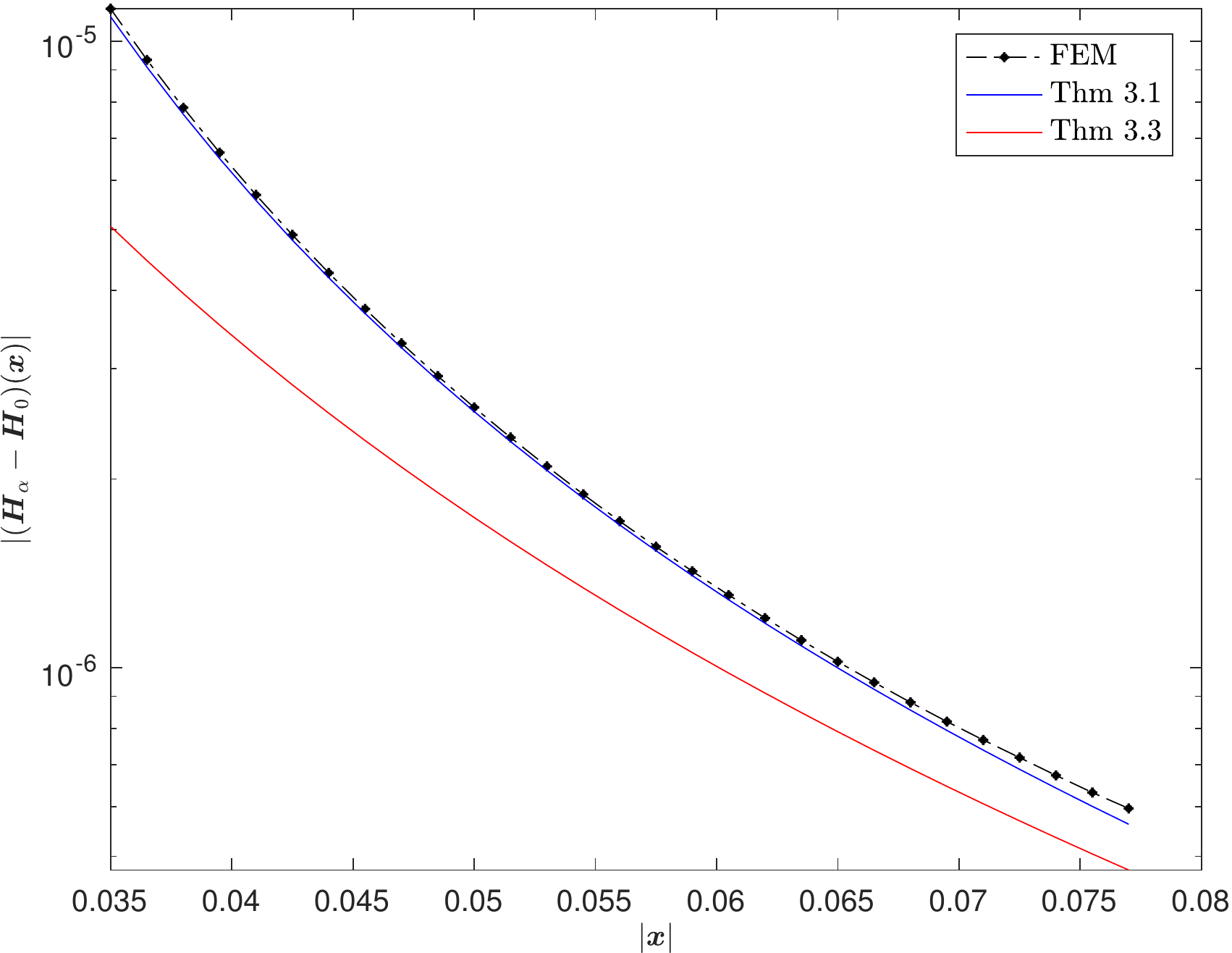} \\
d=0.2, {\vec x}= x_1 {\vec e}_1  & d=2, {\vec x}= x_1 {\vec e}_1 \\
\includegraphics[width=0.5\textwidth]{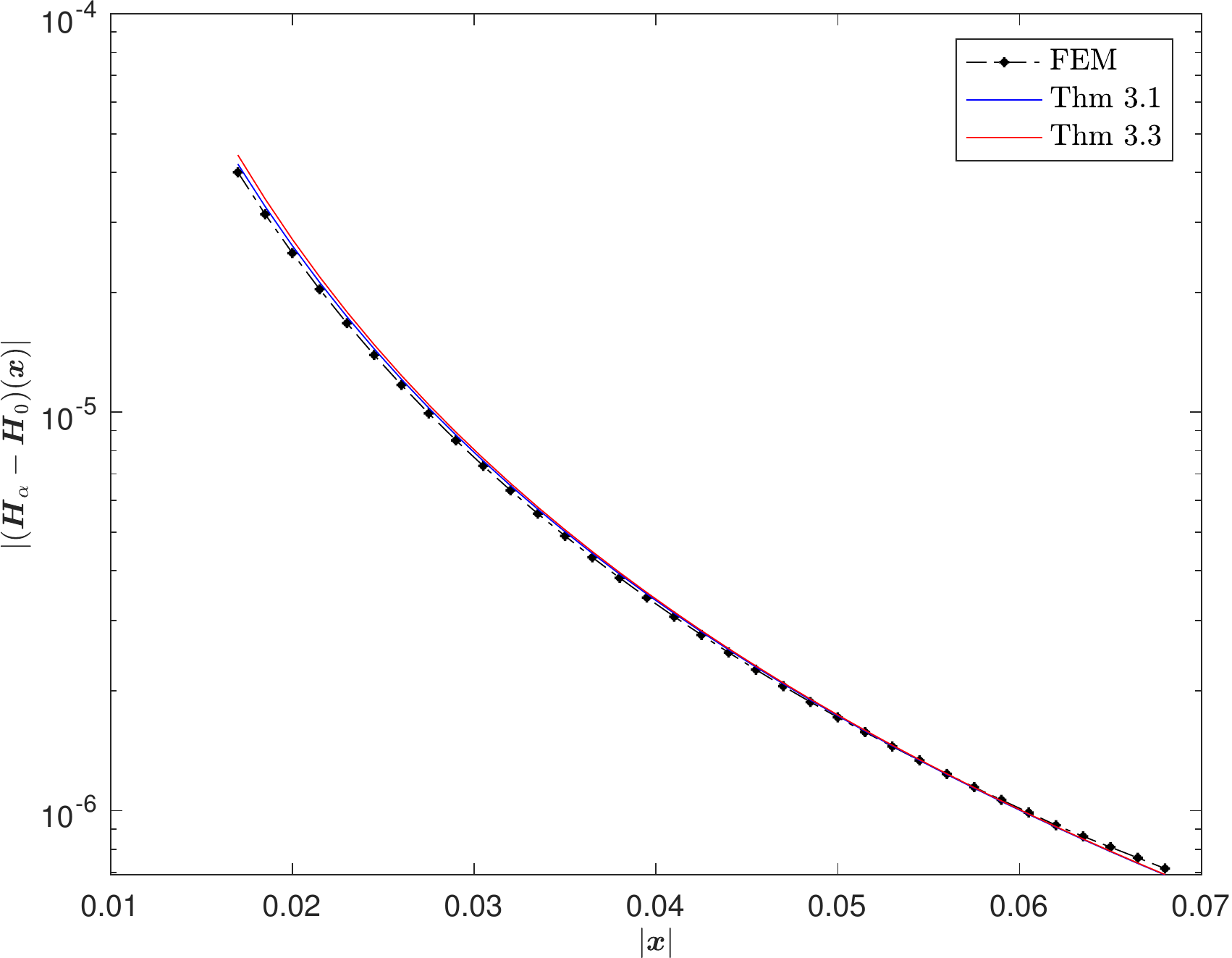} &
\includegraphics[width=0.5\textwidth]{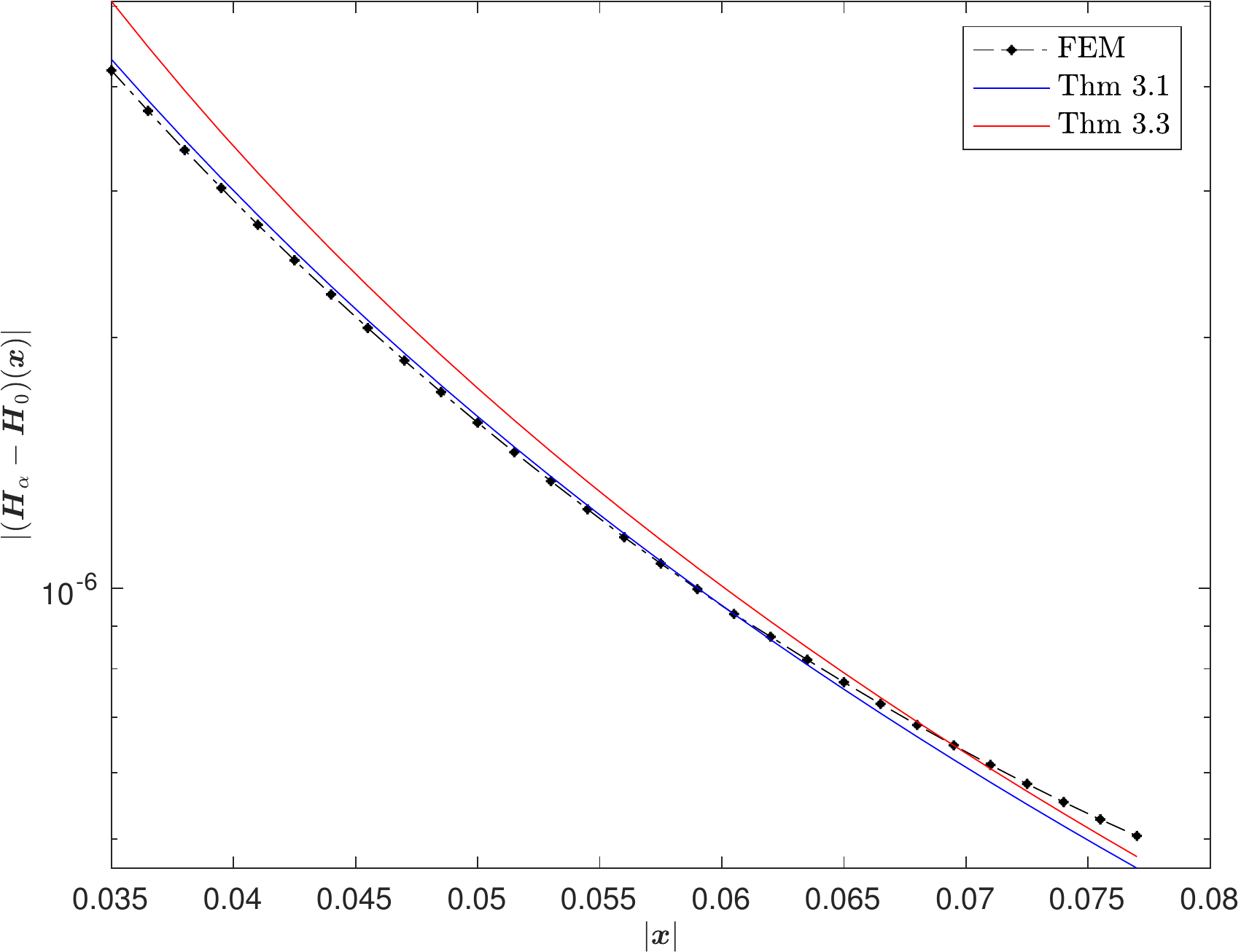} \\
d=0.2, {\vec x}= x_2 {\vec e}_2  & d=2, {\vec x}= x_2 {\vec e}_2 \\
\includegraphics[width=0.5\textwidth]{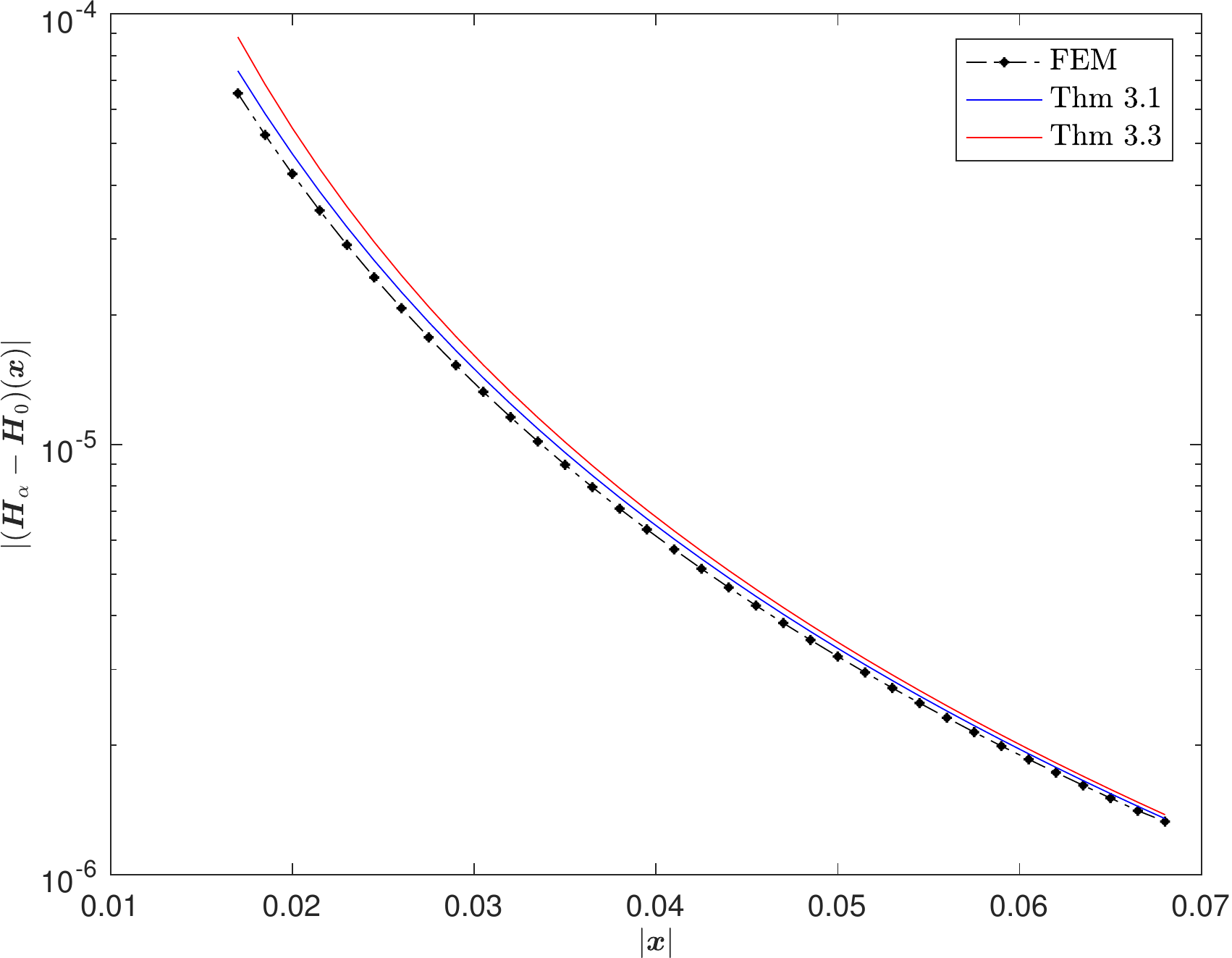} &
\includegraphics[width=0.5\textwidth]{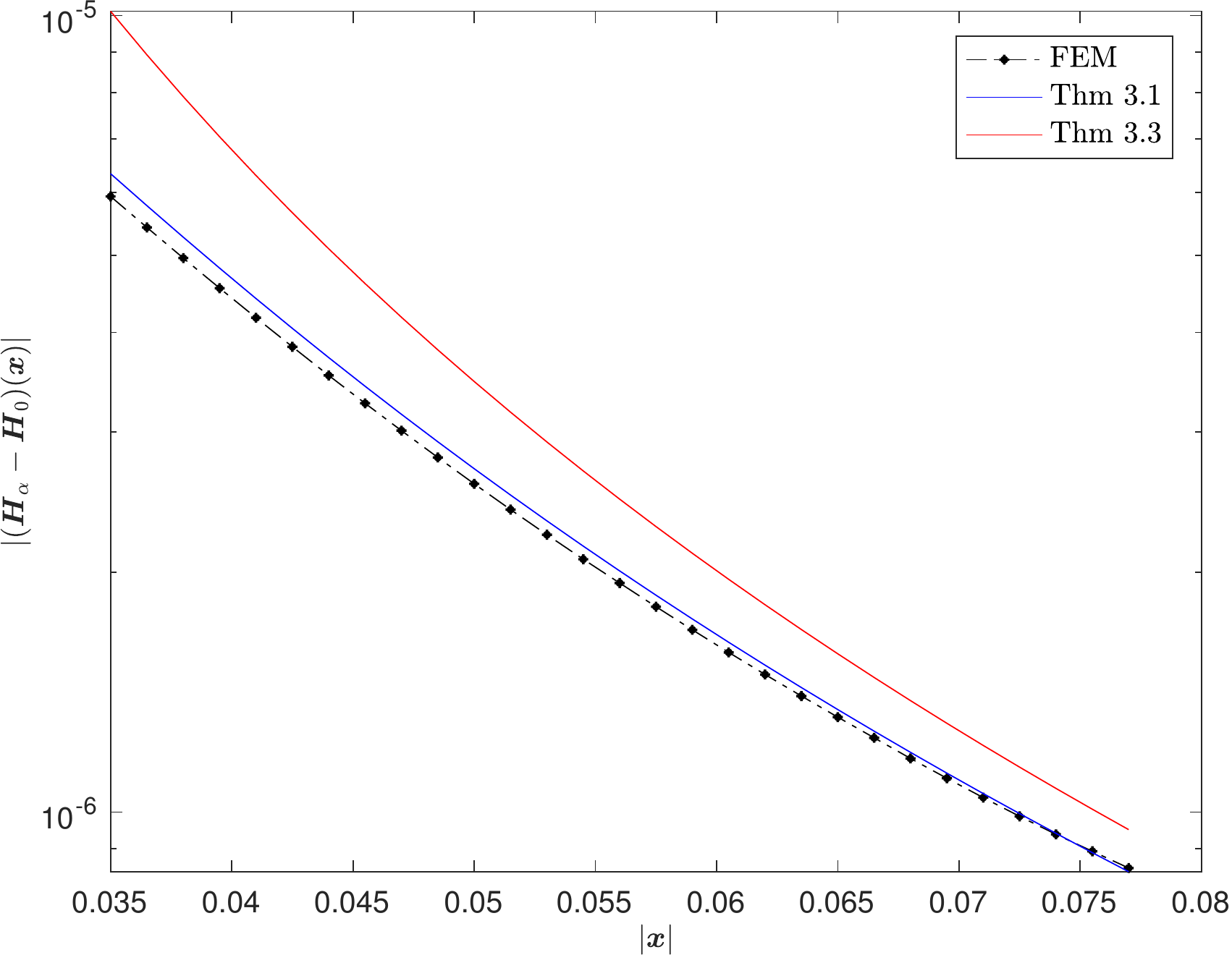} \\
d=0.2, {\vec x}= x_3 {\vec e}_3  & d=2, {\vec x}= x_3 {\vec e}_3 \\
\end{array}$ 
\end{center}
\caption{Comparison of $({\vec H}_\alpha-{\vec H}_0)({\vec x})$ using the asymptotic expansions (\ref{eqn:asymp2})  and (\ref{eqn:asymp3})  in 
Theorems \ref{thm:objectsnocloselyspaced} and \ref{thm:objectscloselyspaced} as well as
 a FEM solution: along the three coordinate axes for two tetrahedra with different separations $\alpha d$.} \label{fig:femtwotet}
\end{figure}

As in Section~\ref{sect:twosph}, we observe good agreement between Theorem~\ref{thm:objectscloselyspaced}  and the FEM solution for the closely spaced objects in Figure~\ref{fig:femtwotet}, with all three results tending to the same result for sufficiently large $|{\vec x}|$. For objects positioned further apart, with $d=2$, we observe that the agreement between Theorem~\ref{thm:objectsnocloselyspaced} and the FEM solution is again best, which again agrees with what our theory predicts, since, for $d=2$, $\min | \partial (B_\alpha)^{(1)} - \partial (B_\alpha)^{(2)} | = 2 \alpha > \alpha_{\max}$ and so this theorem applies.

\subsubsection{Inhomogeneous Parallelepiped}~\label{sect:pip}
In this section, an inhomogeneous parallelepiped ${\vec B}_{ \alpha} = B_\alpha^{(1)} \cup B_\alpha^{(2)} =\alpha ( B^{(1)} \cup B^{(2)} ) = \alpha {\vec B}$ with
\begin{align}
B^{(1)}  = [ -1, 0] \times [0, 1] \times [0, 1], \qquad
B^{(2)}  = [ 0, 1 ] \times [0, 1] \times [0, 1] ,\nonumber
\end{align} 
is considered. The material parameters of $(B_\alpha)^{(1)}$ and $(B_\alpha)^{(2)}$ are $\mu_*^{(1)} = \mu_0$, $\sigma_*^{(1)} = 7.37 \times 10^6 \text{S/m}$, and
$\mu_*^{(2)} = 5.5 \mu_0$, $\sigma_*^{(1)} = 1 \times 10^6 \text{S/m}$, respectively.

Comparisons of $ ({\vec H}_\alpha-{\vec H}_0)({\vec x})$ obtained from  using the asymptotic expansion  (\ref{eqn:asymp3}) in Theorem~\ref{thm:objectscloselyspaced}
and
a full FEM solution are made in Figure~\ref{fig:fempip} along three different coordinates axes. To ensure the tensor coefficients are calculated accurately, a $p=3$ edge element discretisation and an unstructured mesh of $13\,121$ tetrahedra are used for computing  
 ${\mathcal M}\left [{\alpha}{\vec B} \right ]$.  To ensure an accurate representation of $({\vec H}_\alpha-{\vec H}_0)({\vec x})$ for the FEM solver, the same discretisation, suitably scaled, as used for ${\mathcal M}\left [{\alpha}{\vec B} \right ]$ is employed.
We observe a good agreement between Theorem~\ref{thm:objectscloselyspaced}  and the FEM solution for sufficiently large $|{\vec x}|$.
\begin{figure}
\begin{center}
$\begin{array}{cc}
\includegraphics[width=0.5\textwidth]{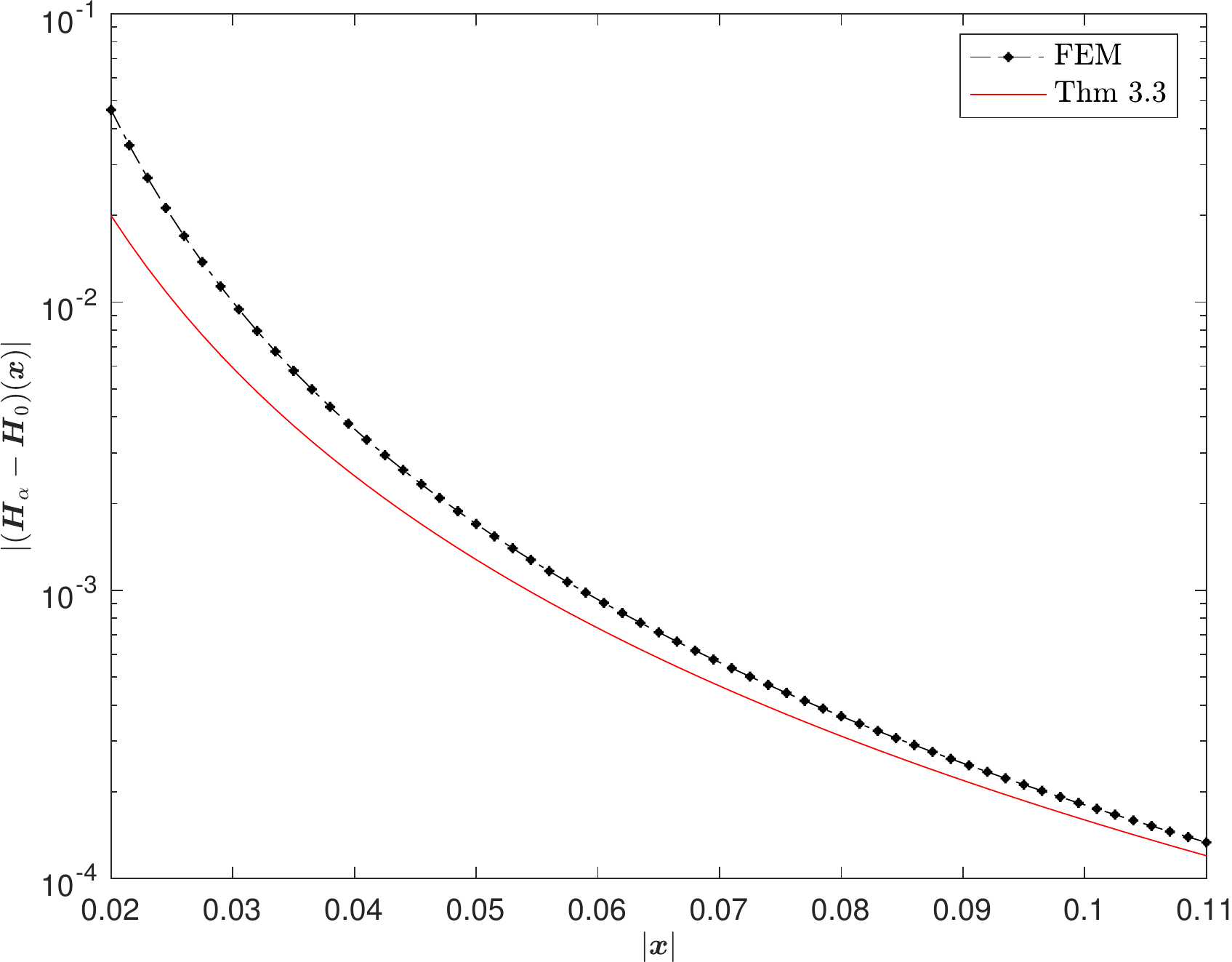} &
\includegraphics[width=0.5\textwidth]{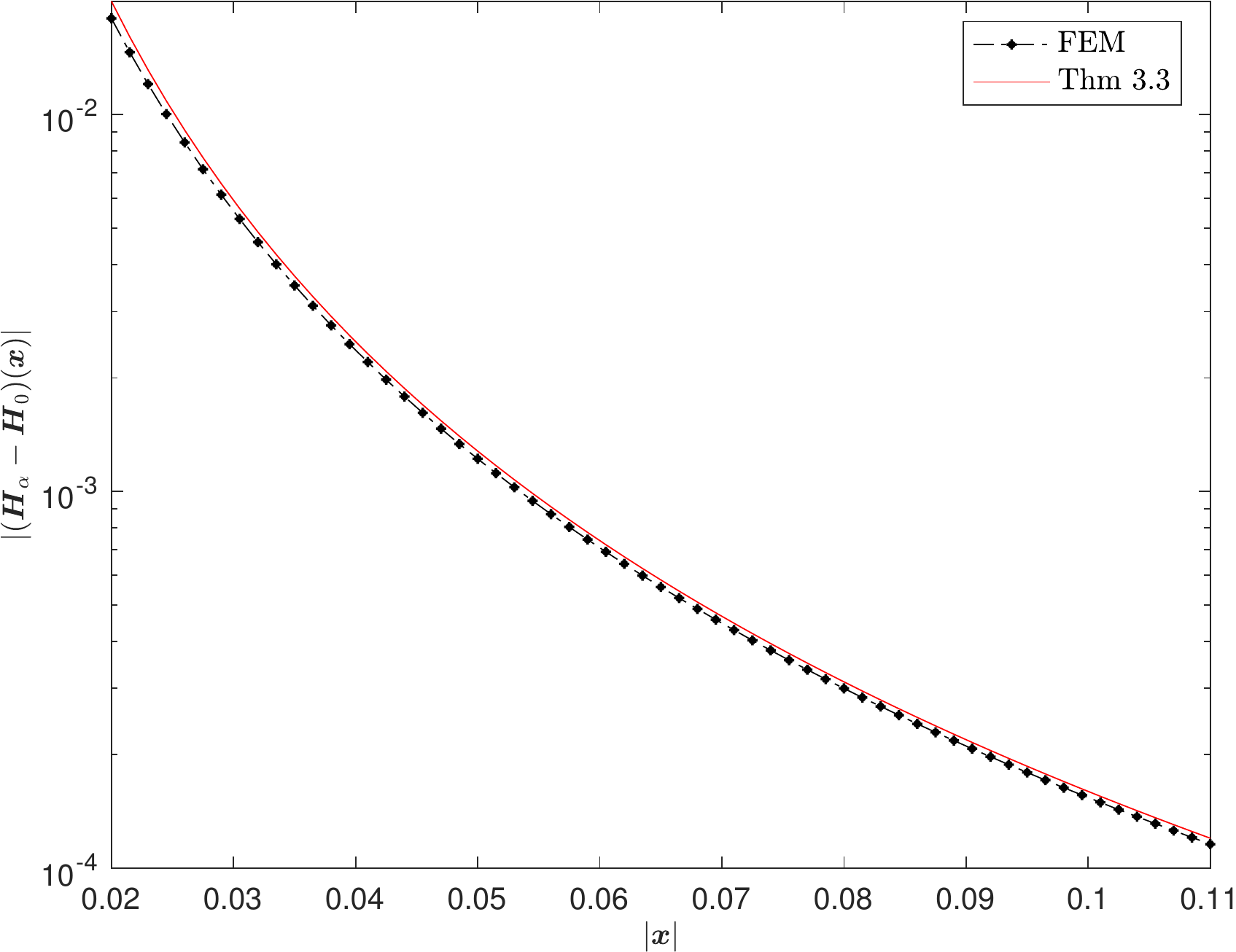} \\
{\vec x}= x_1 {\vec e}_1  &  {\vec x}= x_2 {\vec e}_2 \\
 \end{array}$\\
 $\begin{array}{c}
\includegraphics[width=0.5\textwidth]{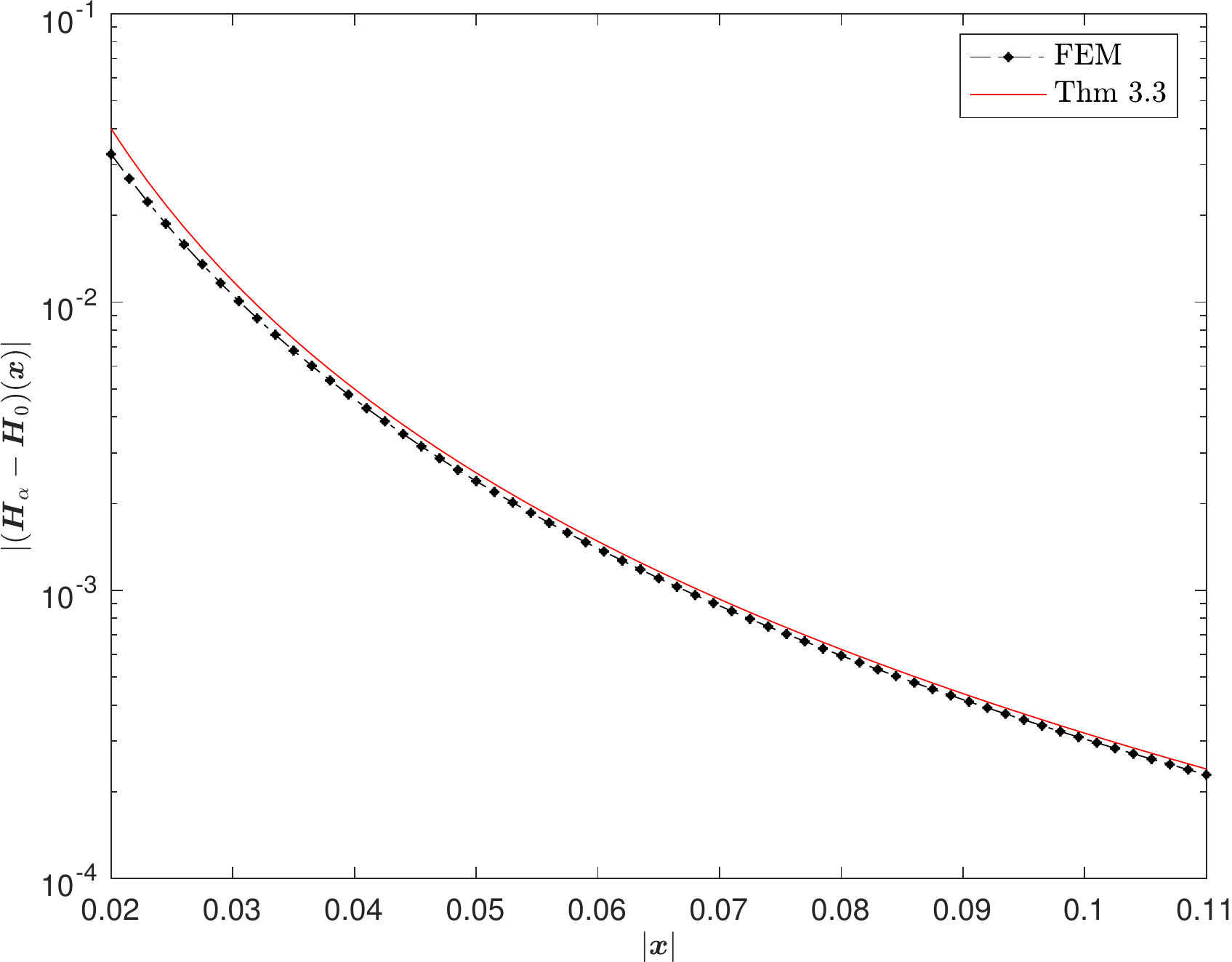} \\
{\vec x}= x_3 {\vec e}_3
\end{array}$ 
\end{center}
\caption{Comparison of $({\vec H}_\alpha-{\vec H}_0)({\vec x})$ using the asymptotic expansion (\ref{eqn:asymp3}) in Theorem~\ref{thm:objectscloselyspaced} and a FEM solution: along the three coordinate axes for an inhomogeneous parallelepiped.} \label{fig:fempip}
\end{figure}

\subsection{Frequency Spectra} \label{sect:freq}
The frequency response of the coefficients of ${\mathcal M}[\alpha { B}]$ for a range of single homogeneous objects has been presented in~\cite{ledgerlionheart2016,ledgerlionheart2018} where the real part was observed to be sigmoid with respect to $\log \omega$ and the imaginary part had a distinctive single maxima. Rather than consider the coefficients, it is in fact better to split ${\mathcal M}[\alpha { B}]$ in to the  real part $\text{Re}({\mathcal M} [\alpha { B}] )$ and an imaginary part  $\text{Im}({\mathcal M} [\alpha { B}] )$, which are both real symmetric rank 2 tensors, and to compute the eigenvalues of these. Indeed, many of the objects previously considered had rotational and/or reflection symmetries such that the eigenvalues coincide with the real and imaginary parts of the diagonal coefficients. 

A theoretical investigation of  ${\mathcal R}[\alpha B] =\text{Re}({\mathcal M}[\alpha B]- {\mathcal N}^0[\alpha B])$ and ${\mathcal I}[\alpha B]= \text{Im}({\mathcal M} [\alpha B] -  {\mathcal N}^0[\alpha B])= \text{Im}({\mathcal M} [\alpha B] )$, where ${\mathcal N}^0[\alpha B]$ corresponds to the real symmetric rank 2 tensor describing the limiting response in the case of $\omega \to 0$, and agrees with the P\'oyla-Sze\"o tensor for a homogenous permeable object,  has been undertaken in~\cite{ledgerlionheartfeq2018}. In this we prove results on the eigenvalues of these tensors.

Now considering ${\mathcal M}[\alpha {\vec B}]$
for an inhomogeneous object ${\vec B}_\alpha$, the coefficients of ${\mathcal N}^0[\alpha {\vec B}]$ are given by
\begin{align}
{\mathcal N}_{ij}^0 [\alpha {\vec B}] & := \frac{\alpha^3}{2} \sum_{n=1}^N  \left ( 1- \frac{\mu_0}{\mu_*^{(n)}} \right ) \int_{B^{(n)}}  \left (
  {\vec e}_i \cdot  \nabla \times {\vec \theta}_j^{(0)} \right ) \dif {\vec \xi} \label{eqn:definen0}.
\end{align}
where
\begin{subequations}
\begin{align} \label{eqn:transproblem0}
\nabla \times \mu^{-1} \nabla \times {\vec \theta}_i^{(0)} & ={\vec 0}  &&  \hbox{in $ {\mathbb R}^3 $}  , \\
\nabla \cdot {\vec \theta}_i^{(0)} &= 0  && \hbox{in ${\mathbb R}^3$} , \\
\left [ {\vec \theta} _i ^{(0)} \times {\vec n} \right ]_\Gamma   &= {\vec 0}   && \hbox{on $\Gamma$} ,\\ 
\,  \left [   \mu^{-1}   \nabla  \times {\vec \theta}_i^{(0)}  \times {\vec n} \right ]_\Gamma  & = 
{\vec 0}  &&
\hbox{on $\Gamma$} ,\\ 
{\vec \theta}_i^{(0)} ( {\vec \xi})- {\vec e}_i \times {\vec \xi} & = O(|{\vec \xi} |^{-1})  && \hbox{as $|{\vec \xi}| \to \infty$} ,
\end{align} 
\end{subequations}
and we have shown that for $0 \le \omega < \infty $, the eigenvalues of $  {\mathcal R} [\alpha {\vec B}]=
\text{Re}({\mathcal M}[\alpha {\vec B}]- {\mathcal N}^0[\alpha {\vec B}])$ and ${\mathcal I}[\alpha {\vec B}]= \text{Im}({\mathcal M} [\alpha {\vec B}] -  {\mathcal N}^0[\alpha {\vec B}])= \text{Im}({\mathcal M} [\alpha {\vec B}] )$ have the properties $\lambda(  {\mathcal R} [\alpha {\vec B}]  ) \le 0$ and $ \lambda( {\mathcal I }[\alpha {\vec B}]  )\ge 0$  (this also applies to a homogenous objects where ${\vec B}_\alpha$ reduces to $B_\alpha$)~\cite{ledgerlionheartfeq2018}.

To illustrate how the behaviour of $\lambda({\mathcal R} [\alpha {\vec B}] )$ and $\lambda( {\mathcal I } [\alpha {\vec B}] )$ changes for an inhomogeneous object, we consider the 
geometry of the parallelepiped described in Section~\ref{sect:pip} placed at the origin
 so that ${\vec B}_\alpha =B_\alpha^{(1)} \cup B_\alpha^{(2)} =  \alpha ( B^{(1)} \cup B^{(2)}) = \alpha {\vec B}  $  with  $\alpha=0.01\text{m}$. Note that, although $B_\alpha^{(1)}$ and $B_\alpha^{(2)}$ have different properties, the object ${\vec B}$ still reflectional symmetries in the ${\vec e}_1$ and ${\vec e}_3$ axes and a $\pi/2$ rotational symmetry about ${\vec e}_1$ so that the independent coefficients of ${\mathcal M}[\alpha {\vec B}]$ are ${\mathcal M}_{11}$ and ${\mathcal M}_{22} = {\mathcal M}_{33}$ (and hence  ${\mathcal R}_{11}$, ${\mathcal R}_{22}= {\mathcal R}_{33}$ are the independent coefficients of ${\mathcal R} [\alpha {\vec B}]$ and    ${\mathcal I}_{11}$, ${\mathcal I}_{22}= {\mathcal I}_{33}$ are the independent coefficients of ${\mathcal I}[\alpha {\vec B}]$).
In Figure~\ref{fig:twocubesconstrastsigma}, we show the computed results for $\lambda({\mathcal R} [\alpha {\vec B}]  )$ and $\lambda( {\mathcal I} [\alpha {\vec B}] )$ for the case where $\sigma_*^{(2)} = 100\sigma_*^{(1)}= 1 \times 10^8 \text{S/m}$ and $\mu_*^{(1)}=\mu_*^{(2)} =\mu_0$
 and, in Figure~\ref{fig:twocubesconstrastmu}, we show the corresponding result for $\sigma_*^{(1)}= \sigma_*^{(2)}= 1 \times 10^6 \text{S/m}$ and $\mu_*^{(2)} = 10 \mu_*^{(1)} = 10 \mu_0$. For this we use similar discretisations to those stated previously. In the former case ${\mathcal N}^0[\alpha {\vec B}]$ vanishes, but not in the latter case.
 
 \begin{figure}
 \begin{center}
 $\begin{array}{cc}
 \includegraphics[width=0.5\textwidth]{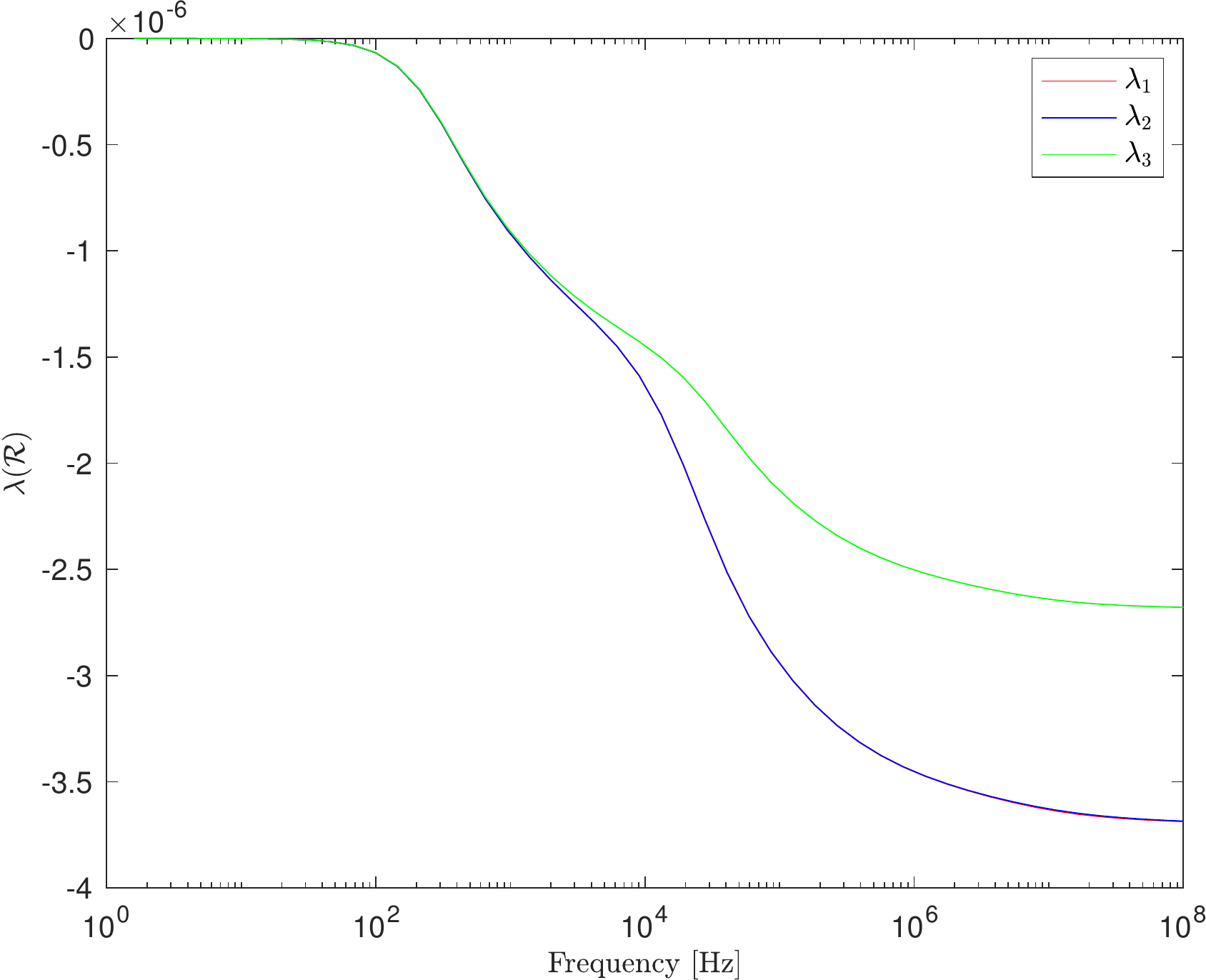} &
 \includegraphics[width=0.5\textwidth]{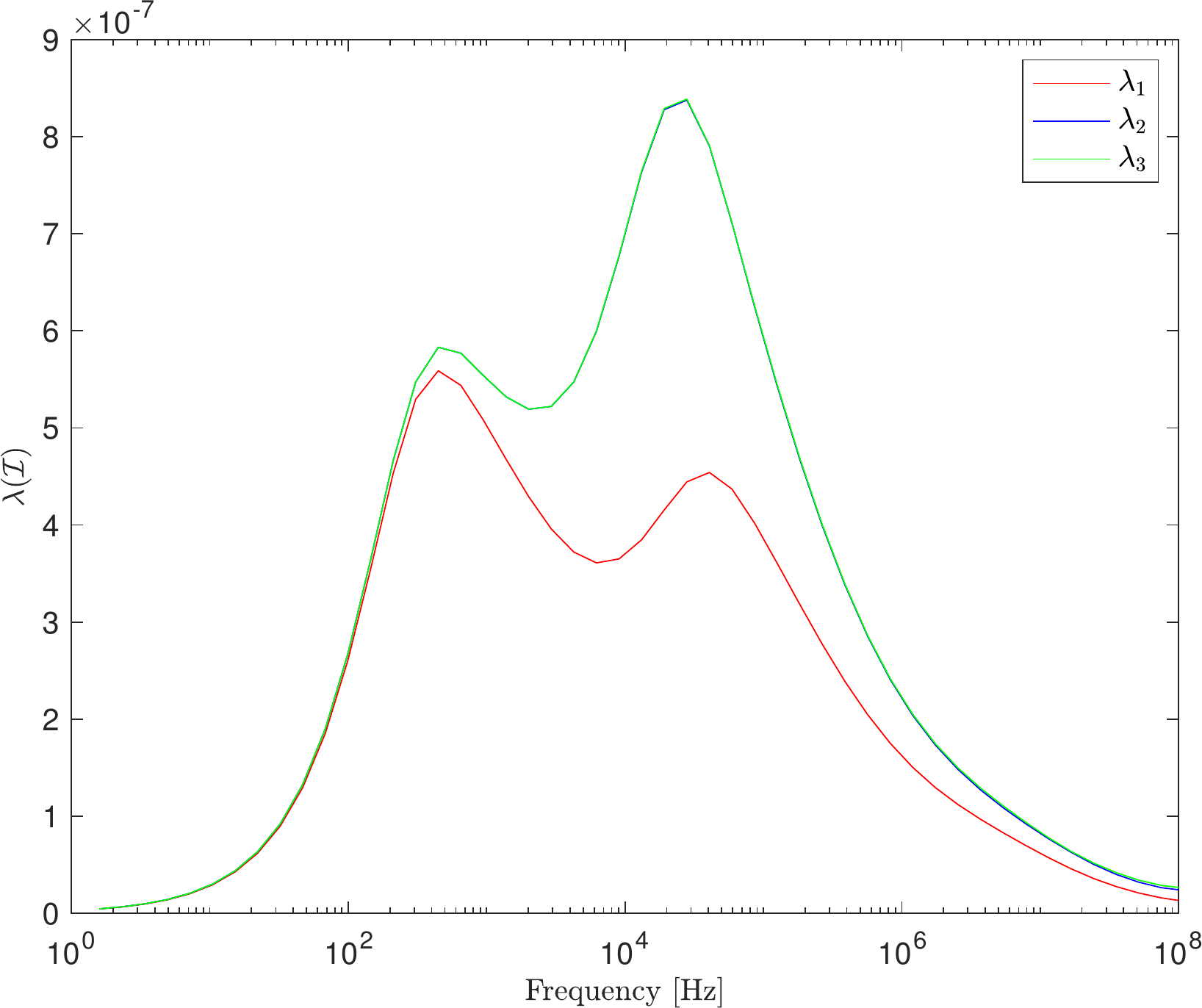} \\
 \lambda( {\mathcal R}[\alpha {\vec B}] ) & \lambda( {\mathcal I}[\alpha {\vec B}] )
 \end{array}$
 \end{center}
 \caption{Frequency dependence of the eigenvalues of $ {\mathcal R}[\alpha {\vec B}]$ and ${\mathcal I}[\alpha {\vec B}] $:  inhomogeneous object parallelepiped  up of two cubes with $\sigma_*^{(2)} = 100\sigma_*^{(1)}= 1 \times 10^8 \text{S/m}$ and $\mu_*^{(1)}=\mu_*^{(2)} =\mu_0$} \label{fig:twocubesconstrastsigma}
 \end{figure}
 
  \begin{figure}
 \begin{center}
 $\begin{array}{cc}
 \includegraphics[width=0.5\textwidth]{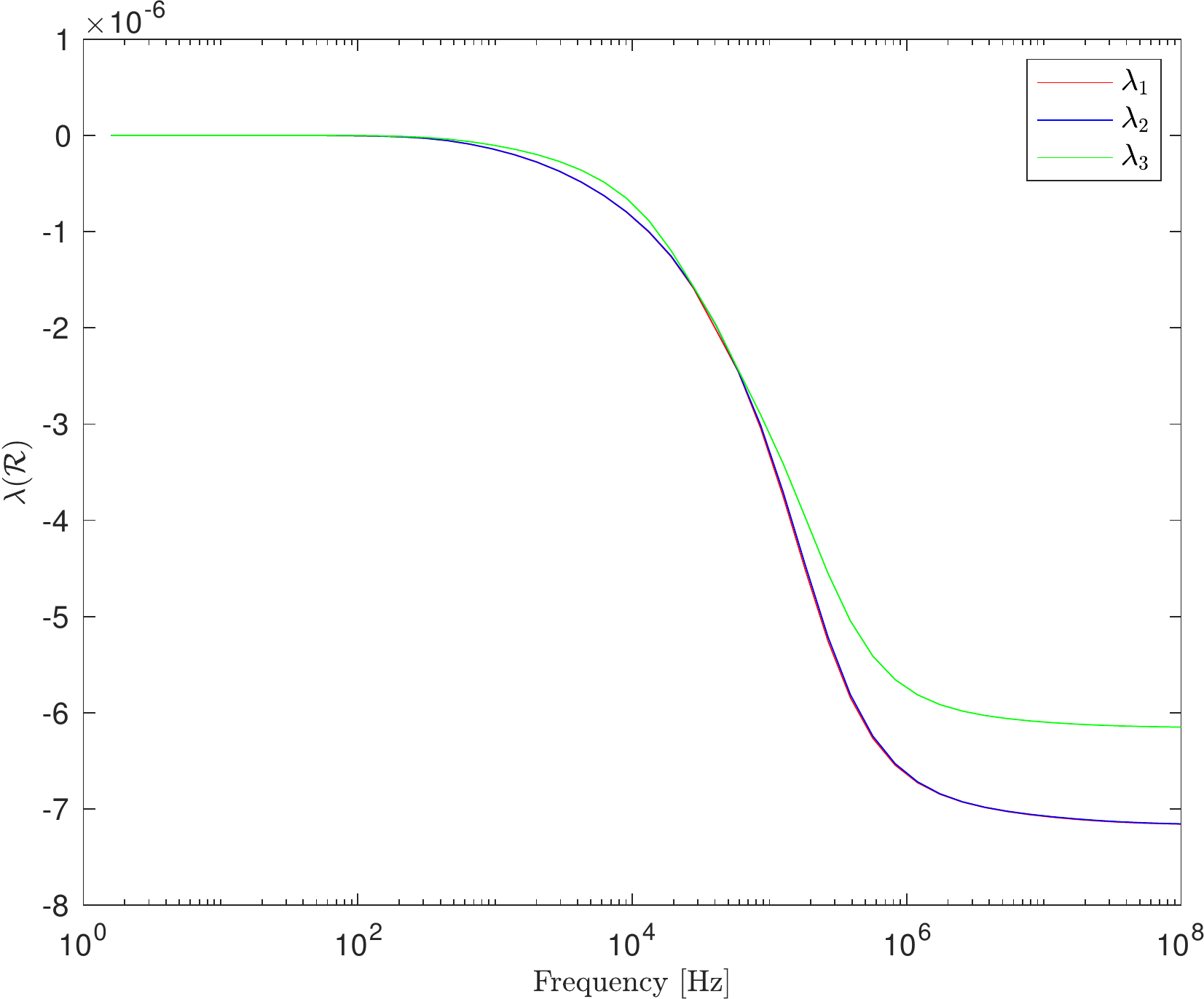} &
 \includegraphics[width=0.5\textwidth]{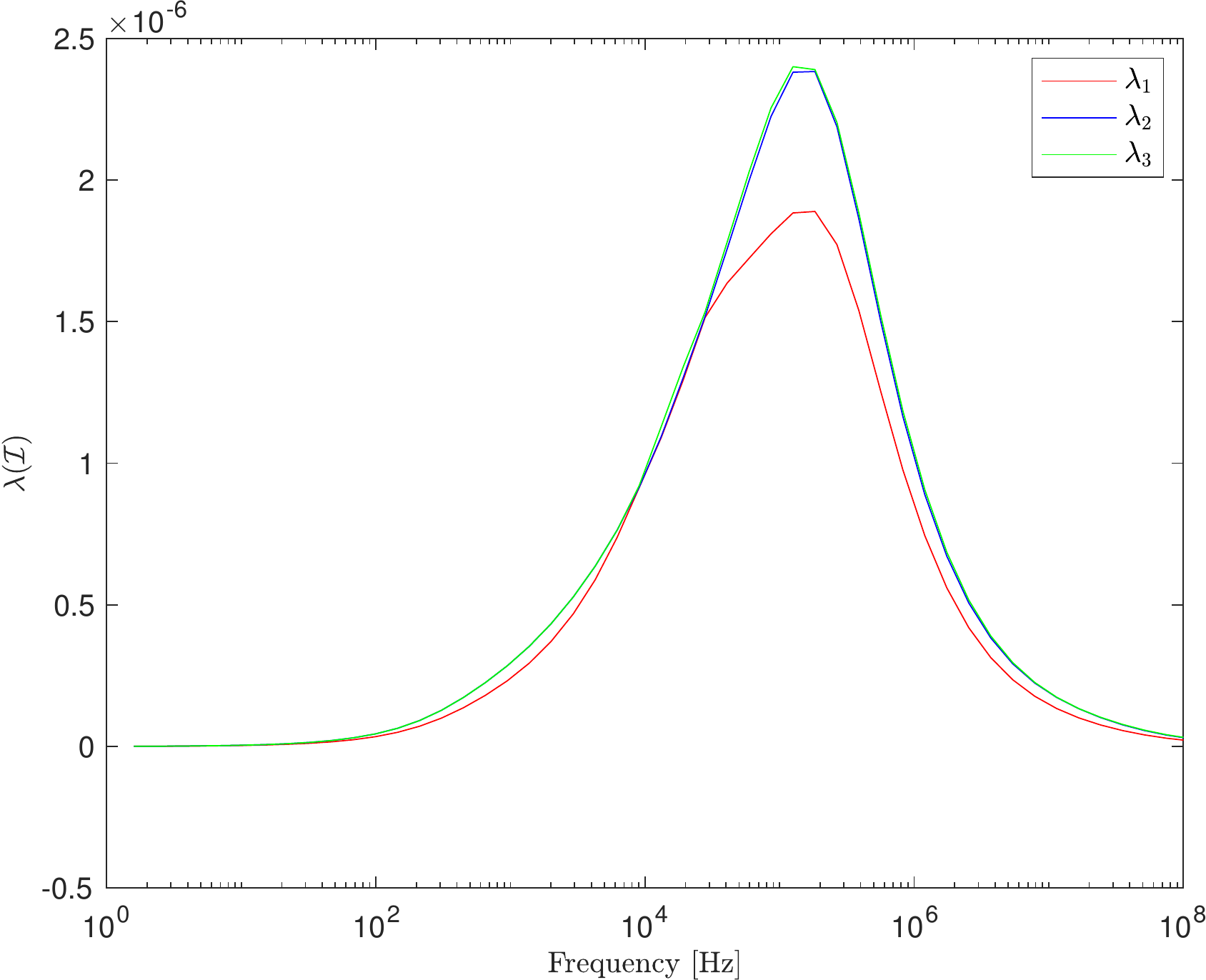} \\
\lambda( {\mathcal R} [\alpha {\vec B}]) & \lambda( {\mathcal I}[\alpha {\vec B}] )
  \end{array}$
 \end{center}
 \caption{Frequency dependence of the eigenvalues of $ {\mathcal R}[\alpha {\vec B}]$ and $ {\mathcal I}[\alpha {\vec B}]$:  inhomogeneous parallelepiped made up of two cubes with $\sigma_*^{(2)} = \sigma_*^{(1)}= 1 \times 10^6 \text{S/m}$ and $\mu_*^{(2)}=100\mu_*^{(1)} =100\mu_0$} \label{fig:twocubesconstrastmu}
 \end{figure}
 
 We observe, in Figure~\ref{fig:twocubesconstrastsigma}, that although $ \lambda_i( {\mathcal R} [\alpha{\vec  B}] )$, $i=1,\ldots,3$ are still monotonically decreasing with $\log f $, it is no longer sigmoid for an inhomogeneous object with varying $\sigma$ and constant $\mu$ and has multiple non--stationary inflection points. Furthermore, 
rather than a single maxima,  $ \lambda_i( {\mathcal I} [\alpha {\vec B}])$, $i=1,\ldots,3$ has two distinct local maxima.
 However, the results shown in Figure~\ref{fig:twocubesconstrastmu} illustrate for an inhomogeneous object with varying $\mu$ and constant $\sigma$, $ \lambda_i( {\mathcal R}[\alpha {\vec B}] )$, $i=1,\ldots,3$, that the behaviour is quite different and, in this case, $ \lambda_i( {\mathcal R} [\alpha {\vec B}])$, $i=1,\ldots,3$ is still sigmoid and the curves for   $ \lambda_i( {\mathcal I}[\alpha {\vec B}] )$, $i=1,\ldots,3$ still have a single maxima.
In the limiting case of $\omega \to 0 $, $\lambda_i(\text{Re} ( {\mathcal M} [\alpha {\vec B}] ))  \to \lambda_i(\text{Re} ( {\mathcal N}^0[\alpha {\vec B}]  ))  $, $i=1,\ldots,3$ and, for the latter case with a contrast in $\mu$, the behaviour is as shown in Figure~\ref{fig:twocubesconstrastmurem}, which is quite different to a homogenous object of the same size. 
  \begin{figure}
 \begin{center}
 $\begin{array}{c}
 \includegraphics[width=0.5\textwidth]{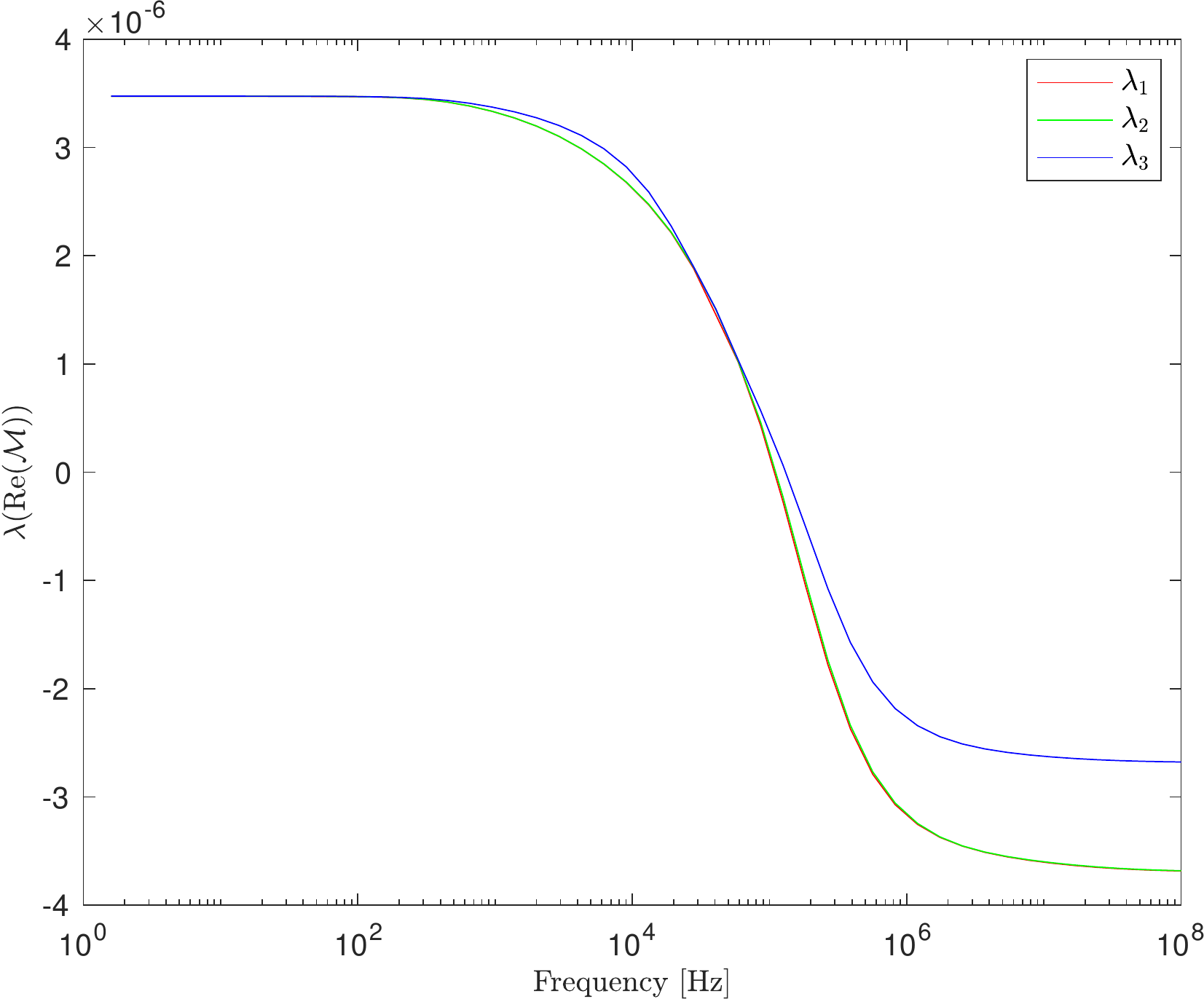} \\
\lambda ( \text{Re} ( {\mathcal M}[\alpha {\vec B}]))  
\end{array}$
\end{center}
 \caption{Frequency dependence of the eigenvalues of $ \text{Re} ( {\mathcal M}[\alpha {\vec B}])$  :  inhomogeneous parallelepipedmade up of two cubes with $\sigma_*^{(2)} = \sigma_*^{(1)}= 1 \times 10^6 \text{S/m}$ and $\mu_*^{(2)}=100\mu_*^{(1)} =100\mu_0$} \label{fig:twocubesconstrastmurem}
\end{figure}

To investigate the behaviour of inhomogeneous objects still further, we next consider the inhomogeneous parallelepiped ${\vec B}_\alpha = B_\alpha^{(1)} \cup B_\alpha^{(2)} \cup B_\alpha^{(3)} = \alpha ( B^{(1)} \cup B^{(2)}\cup B^{(3)} ) =\alpha {\vec B}$ with
\begin{align}
B^{(1)}  = [ -3/2, -1/2] \times [0, 1] \times [0, 1],  \qquad
B^{(2)}  = [ -1/2, 1/2 ] \times [0, 1] \times [0, 1], \nonumber \\
 B^{(3)}=[ 1/2, 3/2] \times [0, 1] \times [0, 1],  \qquad \qquad \qquad \nonumber
\end{align} 
and $\alpha=0.01\text{m}$. To compute ${\mathcal M}[\alpha {\vec B}]$, an unstructured mesh of 15~109 tetrahedral elements is generated and $p=4$ elements  employed. The independent coefficients of ${\mathcal M}[\alpha {\vec B}] $ are again ${\mathcal M}_{11}$ and ${\mathcal M}_{22} = {\mathcal M}_{33}$.

In Figure~\ref{fig:threecubesconstrastsigma}, we show $\lambda({\mathcal R} [\alpha {\vec B}])$ and $\lambda( {\mathcal I}[\alpha {\vec B}]  )$ for the case where $\sigma_*^{(3)}=100 \sigma_*^{(2)}= 10^4 \sigma_*^{(1)} =1 \times 10^8 \text{S/m}$ and $\mu_*^{(1)} =  \mu_*^{(2)} =\mu_*^{(3)} = \mu_0$
 and, in Figure~\ref{fig:threecubesconstrastmu}, we show the corresponding result for $\sigma_*^{(1)}= \sigma_*^{(2)}=  \sigma_*^{(1)} =1 \times 10^6 \text{S/m}$ and $\mu_*^{(3)} =  10 \mu_*^{(2)} =100 \mu_*^{(3)} = 100 \mu_0$. In the former case ${\mathcal N}^0 [\alpha {\vec B}] $ vanishes, but not in the latter case.
 
 \begin{figure}
 \begin{center}
 $\begin{array}{cc}
 \includegraphics[width=0.5\textwidth]{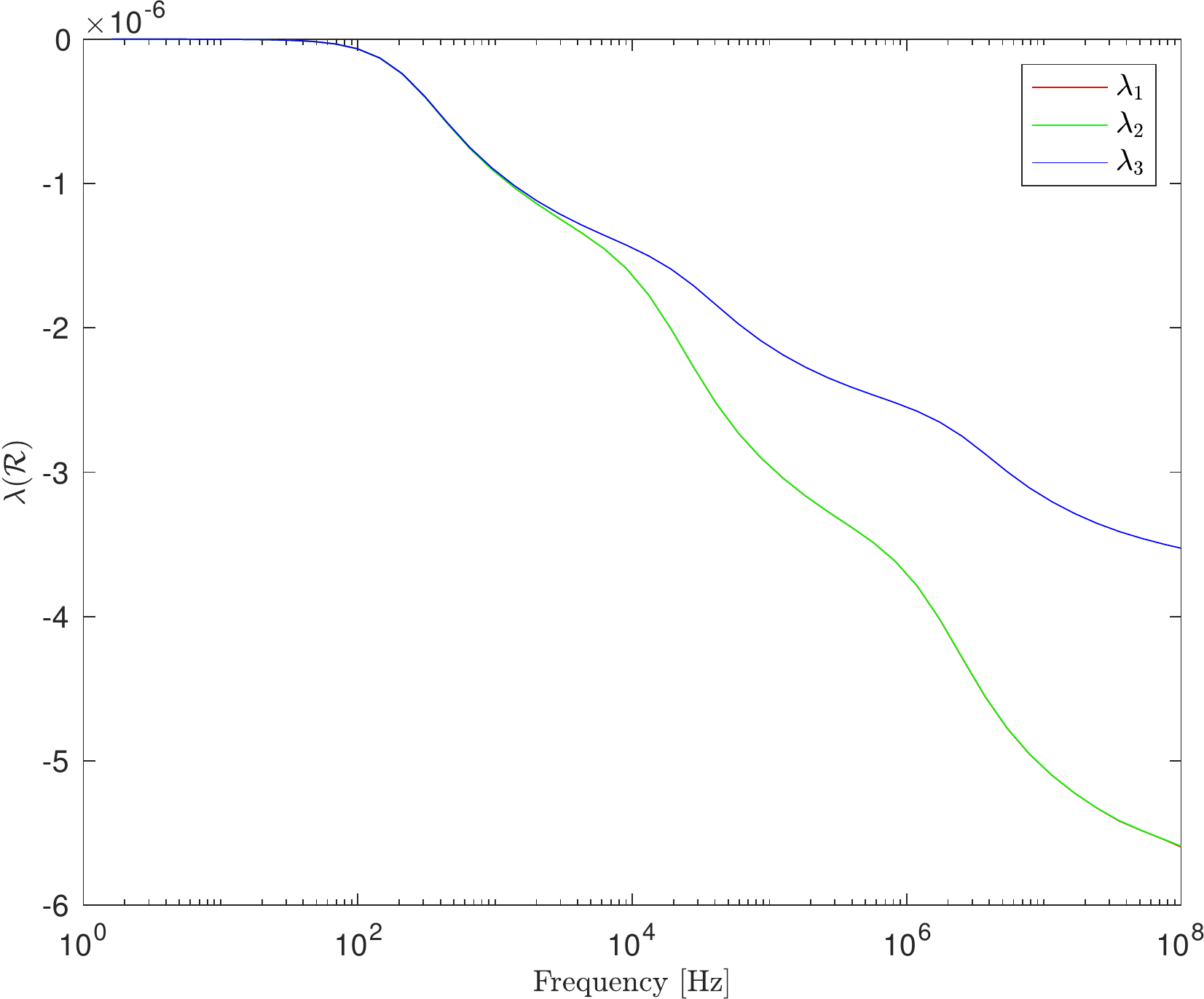} &
 \includegraphics[width=0.5\textwidth]{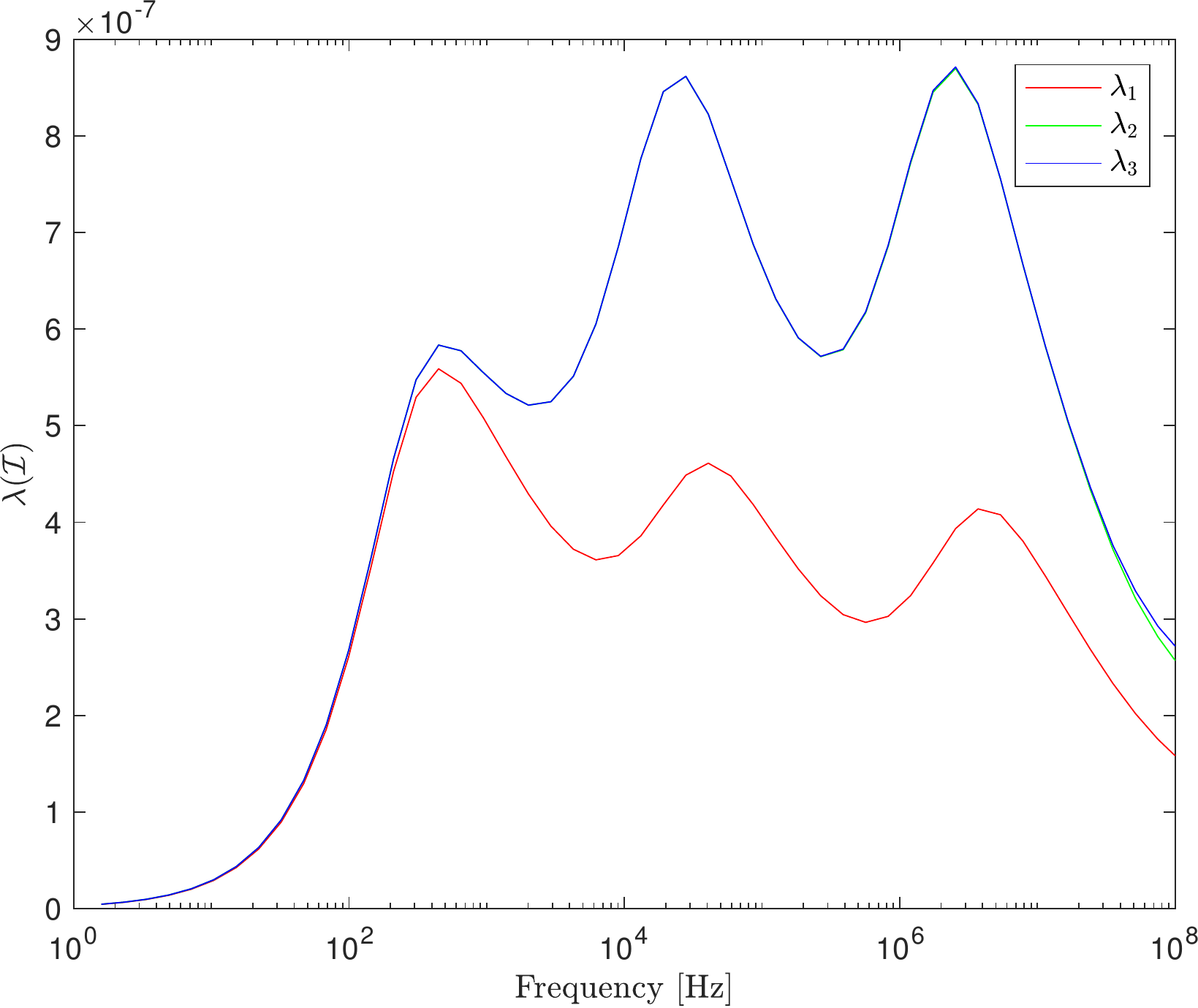} \\
 \lambda( {\mathcal R} [\alpha {\vec B}] ) & \lambda( {\mathcal I}[\alpha {\vec B}] )
 \end{array}$
 \end{center}
 \caption{Frequency dependence of the eigenvalues of $ {\mathcal R}[\alpha {\vec B}]$ and ${\mathcal I}[\alpha {\vec B}] $:  inhomogeneous parallelepiped made up of three cubes with $\sigma_*^{(3)}=100 \sigma_*^{(2)}= 10^4 \sigma_*^{(1)} =1 \times 10^8 \text{S/m}$ and $\mu_*^{(1)} =  \mu_*^{(2)} =\mu_*^{(3)} = \mu_0$} \label{fig:threecubesconstrastsigma}
 \end{figure}
 
 \begin{figure}
 \begin{center}
 $\begin{array}{cc}
 \includegraphics[width=0.5\textwidth]{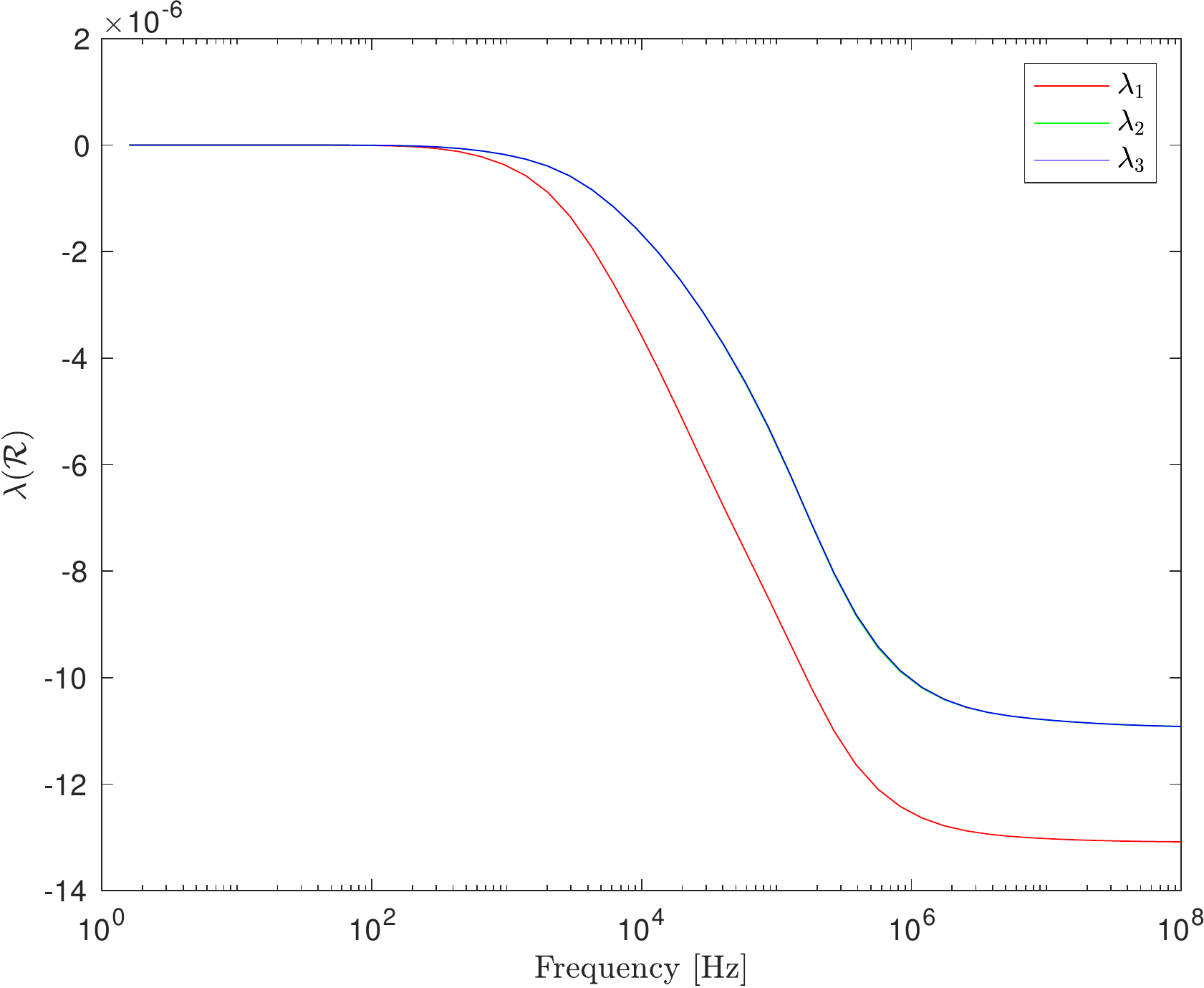} &
 \includegraphics[width=0.5\textwidth]{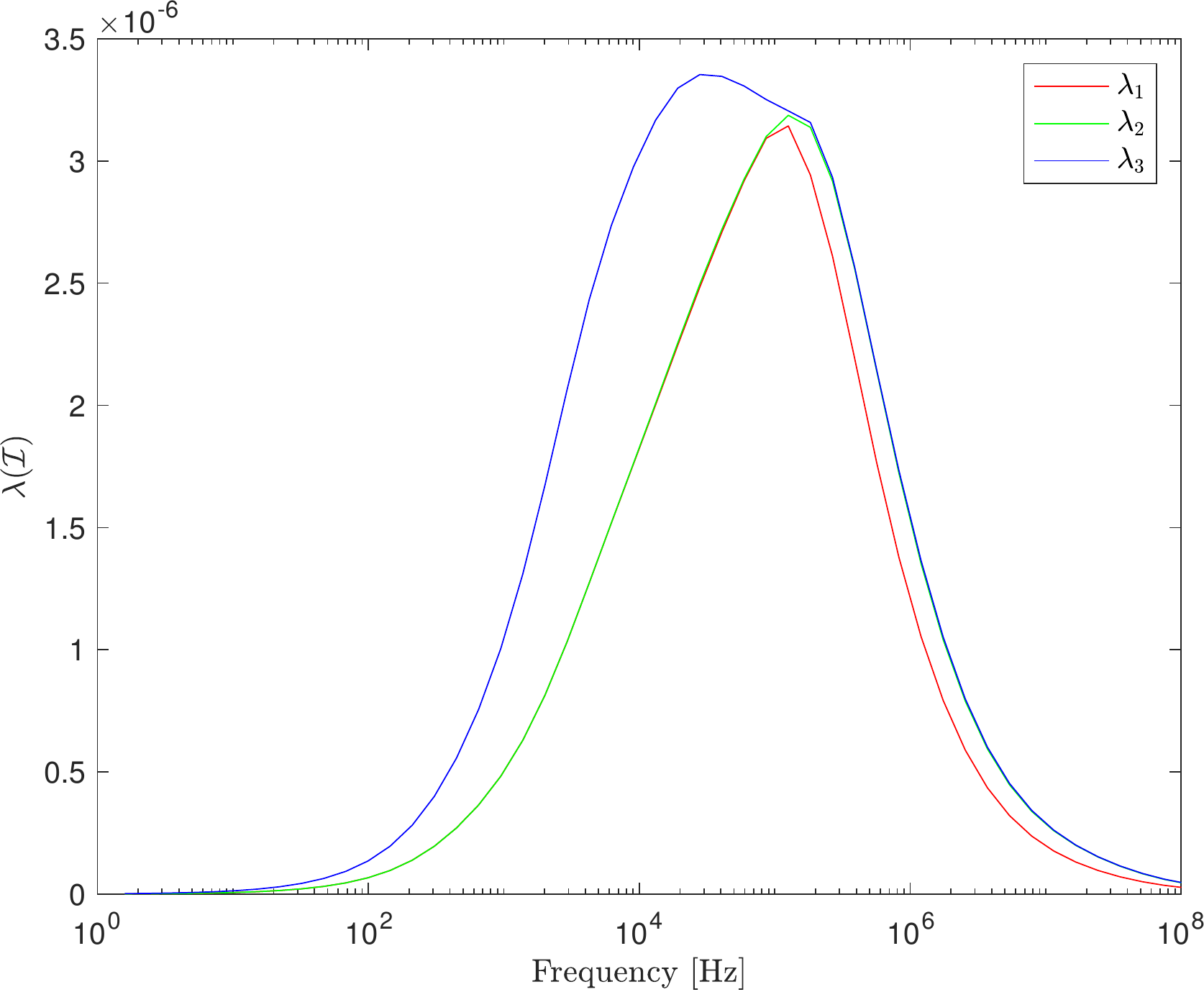} \\
 \lambda( {\mathcal R} [\alpha {\vec B}]) & \lambda( {\mathcal I} [\alpha {\vec B}])
 \end{array}$
 \end{center}
 \caption{Frequency dependence of the eigenvalues of $ {\mathcal R}[\alpha {\vec B}]$ and ${\mathcal I}[\alpha {\vec B}] $:  inhomogeneous parallelepiped made up of three cubes with $\sigma_*^{(1)}= \sigma_*^{(2)}=  \sigma_*^{(1)} =1 \times 10^6 \text{S/m}$ and $\mu_*^{(3)} =  10 \mu_*^{(2)} =100 \mu_*^{(3)} = 100 \mu_0$} \label{fig:threecubesconstrastmu}
 \end{figure}
 
 We observe, in Figure~\ref{fig:threecubesconstrastsigma}, that  $ \lambda_i( {\mathcal R} [\alpha{\vec  B}])$, $i=1,\ldots,3$ is still monotonically decreasing with multiple non-stationary points of inflection and $ \lambda_i( {\mathcal I}[\alpha B] )$, $i=1,\ldots,3$ now has three distinct local maxima.  In Figure~\ref{fig:threecubesconstrastmu}, we see that  $ \lambda_i( {\mathcal R}[\alpha {\vec B}] )$, $i=1,\ldots,3$ is sigmoid and $ \lambda_i( {\mathcal I} [\alpha {\vec B}])$, $i=1,\ldots,3$ has only a single maxima. Unlike, Figure~\ref{fig:twocubesconstrastmurem}, we see in Figure~\ref{fig:threecubesconstrastmurem} that the low frequency response of $\lambda_i(\text{Re}( M [\alpha {\vec B}] ))$, $i=1,\ldots,3$ are different. This is probably due to the fact that the chosen contrasts in $\mu$ imply that one of the three cubes no longer has a dominant effect over the other two.
 
   \begin{figure}
 \begin{center}
 $\begin{array}{c}
 \includegraphics[width=0.5\textwidth]{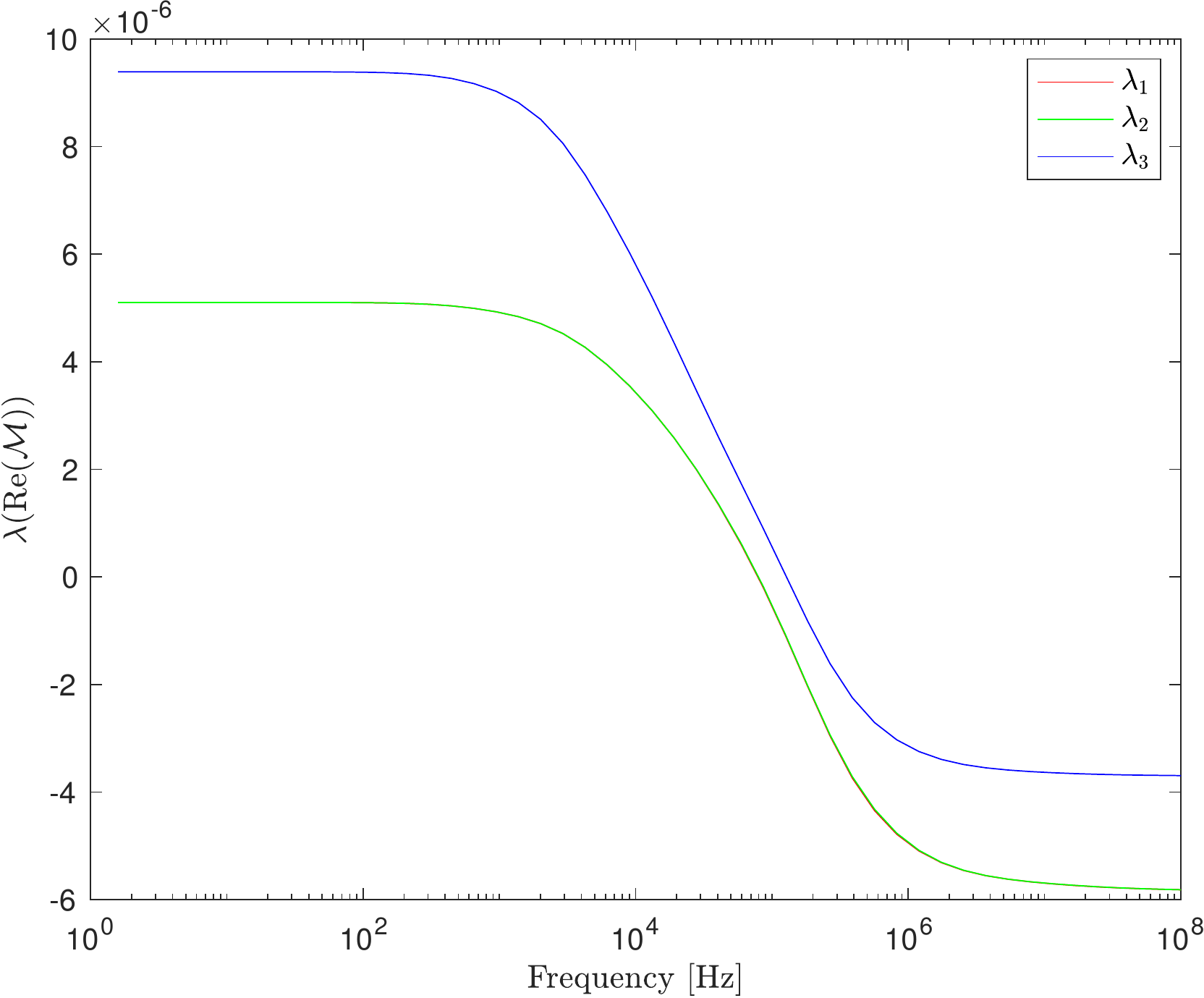} \\
\lambda ( \text{Re} ( {\mathcal M}[\alpha B] ))  
\end{array}$
\end{center}
 \caption{Frequency dependence of the eigenvalues of $ \text{Re} ( {\mathcal M}[\alpha {\vec B}])$  :  inhomogeneous parallelepiped made up of three cubes  with $\sigma_*^{(1)}= \sigma_*^{(2)}=  \sigma_*^{(1)} =1 \times 10^6 \text{S/m}$ and $\mu_*^{(3)} =  10 \mu_*^{(2)} =100 \mu_*^{(3)} = 100 \mu_0$} \label{fig:threecubesconstrastmurem}
\end{figure}

 \begin{remark}
The results shown in Figures~\ref{fig:twocubesconstrastsigma} and ~\ref{fig:threecubesconstrastsigma} indicate that the number of points of non-stationary inflection in  $\lambda_i({\mathcal R}[\alpha {\vec B}]  )$ and the number of local maxima in $\lambda_i( {\mathcal I}[\alpha {\vec B}] )$  can potentially be used to determine an upper bound on the number of regions with varying $\sigma$ that make up an inhomogeneous object ${\vec B}_\alpha$. Note, that contrasts in $\sigma$ between the regions making up the inhomogeneous object have deliberately chosen as large in these examples and we acknowledge that, for small contrasts, detecting such peaks would be more challenging.
\end{remark}
  
 \subsection{Object Localisation}

The approach described by Ammari, Chen, Chen, Volkov and Wang~\cite{ammarivolkov2013b} for a single object localisation using multistatic measurements simplifies given our object characterisation in terms of rank 2 MPT for a single homogenous object and also easily extends to inhomogeneous and multiple objects.  Following~\cite{ammarivolkov2013b}, we assume that there are $K$ receivers at locations ${\vec r}^{(k)}$, $k=1,\ldots,K$, which are associated with small measurement coils with dipole moment ${\vec q}$, and $L$ sources at locations ${\vec s}^{(\ell)}$, $\ell=1,\ldots,L$, which are associated with small exciting coils each with dipole moment ${\vec p}$. Then, by measuring the field perturbation described by Theorem~\ref{thm:objectsnocloselyspaced} for $N_{\text{target}}=N$ objects in the direction ${\vec q}$, this gives rise to the $k$, $\ell$th entry of the multistatic response matrix as
\begin{align}
A_{k \ell } = \sum_{n=1}^{N_{\text{target}}}  ({\vec D}^2G({\vec r}^{(k)} , {\vec z}^{(n)} ) {\vec q}) \cdot ( {\mathcal M}[\alpha^{(n)} B^{(n)} ] ({\vec D}^2G({\vec z}^{(n)} ,{\vec s}^{(\ell)}){\vec p}  )) + R_{k \ell},
\nonumber
\end{align}
where, for the purpose of the following, we arrange the coefficients of the rank 2 tensors ${\mathcal M}[\alpha ^{(n)} B^{(n)}]$ as $3 \times 3$ matrices~\footnote{For (multiple) inhomogeneous objects we replace ${\mathcal M}[\alpha^{(n)} B^{(n)}]$ here and in the following by  ${\mathcal M}[\alpha^{(n)} {\vec B}^{(n)}]$ where $\alpha^{(n)}$ becomes the size of the $n$th inhomogeneous object with configuration ${\vec B}^{(n)}$}.
Assuming that the data is corrupted by measurement noise and is sampled using Hadamard's technique, as in\cite{ammarivolkov2013b}, then the MSR matrix can be written in the form
\begin{align}
A= & \sum_{n=1}^{N_{\text{target}}} U^{(n)} (  {\mathcal M}[\alpha^{(n)} B^{(n)}] V^{(n)}) + R +\frac{ S_{\text{noise}} }{\sqrt{M} } \tilde{W} \nonumber  \\
 = & U {\frak M} V +  R +\frac{ S_{\text{noise}}}{\sqrt{M} } \tilde{W} , \nonumber 
\end{align}
where $\tilde{W} = \frac{1}{\sqrt{2}} ( W + \im W) $ and $W$ is a $K\times L$ matrix with independent and identical Gaussian entries with zero mean and unit variance and $S_{\text{noise}}$ is a positive constant. In addition, $U$ is a matrix of size $ K \times 3N_{\text{target}}$ 
\begin{align}
U = \left ( \begin{array}{ccc} U^{(1)} & \cdots & U^{(N)} \end{array} \right ) \nonumber ,
\end{align}
and $U^{(n)}$ is an $K \times 3$ matrix 
\begin{equation}
U^{(n)} = \left ( \begin{array}{ccc} ({\vec D}^2G({\vec r}_1,{\vec z}^{(n)}){\vec q})_1 & \cdots & ({\vec D}^2G({\vec r}_1,{\vec z}^{(n)}){\vec q})_3 \\
\vdots & \vdots & \vdots \\
({\vec D}^2G({\vec r}_K,{\vec z}^{(n)}){\vec q})_1 & \cdots & ({\vec D}^2G({\vec r}_K,{\vec z}^{(n)}){\vec q})_3 \end{array} \right ) , \nonumber
\end{equation}
The matrix ${\frak M}$ is of size $ 3N_{\text{target}} \times 3N_{\text{target}}$ and is block diagonal in the form
\begin{equation}
{\frak M}= \text{diag} ({\mathcal M} [\alpha^{(1)} B^{(1)}], \cdots,{\mathcal M}[\alpha^{(n)} B^{(n)}] ) ,  \nonumber
\end{equation}
and the matrix $V$ is of dimension $3N_{\text{target}} \times L$  with
\begin{align}
V = \left ( \begin{array}{ccc} V^{(1)} & \cdots V^{(N)} \end{array} \right ) \nonumber ,
\end{align}
where $V^{(n)}$ is the $ 3 \times L$ matrix
\begin{equation}
 V^{(n)}= \left ( \begin{array}{ccc}
 {\vec D}^2G({\vec z}^{(n)}, {\vec s}^{(1)} ){\vec p}& \cdots & {\vec D}^2G({\vec z}^{(n)}, {\vec s}^{(M)}) {\vec p}) \end{array} \right ) .
\end{equation} 

Proceeding in a similar manner to~\cite{ammarivolkov2013b}, and defining the linear operator $L: {\mathbb C}^{3N_{\text{target}} \times 3N_{\text{target}}} \to {\mathbb C}^{K \times L}$ as
\begin{equation}
L( {\frak M} ) = U {\frak M}V ,
\end{equation}
then, by dropping the higher order term, the MSR matrix can be approximated as
\begin{equation}
A\approx A_0+  \frac{ S_{\text{noise}} }{\sqrt{M} } W= L( {\frak M}) +  \frac{ S_{\text{noise}} }{\sqrt{M} } W \nonumber .
\end{equation}
The MUSIC algorithm can then be used to localise the location of the multiple arbitrary shaped targets by using the same imaging functional as proposed in ~\cite{ammarivolkov2013b}
\begin{align}
I_{MU} ( {\vec z}^s) = \left ( \frac{1}{\sum_{i=1}^3 \| { P ( {\vec D}^2 G({\vec z}^s,{\vec s}^{(1)} ){\vec p}\cdot {\vec e}_i ,  \cdots , {\vec D}^2G({\vec z}^s,{\vec s}^{(L)}){\vec p}\cdot {\vec e}_i ) \|^2 }} \right )^{1/2} ,
\end{align}
where $P $ is the orthogonal projection onto the right null space of $L({\frak M} )$.

\begin{proposition}
Suppose that $U{\frak M}$  has full rank. Then $L({\frak M})$ has $3N_{\text{\emph{ target}}}$ non-zero singular values. Furthermore, $I_{MU} $ will have $N_{\text{target}}$ peaks at the object locations ${\vec z} = {\vec z}^s$.
\end{proposition}

The ability to recover the $N_{\text{target}}$ objects will depend on a number of factors :
\begin{enumerate}
\item The number and locations of the measurement and excitor pairs. In practice the number of each will be limited to powers of $4$ for practical reasons~\cite{ammarivolkov2013}.
\item The noise level, which we define as the reciprocal of the signal to noise ratio in terms of the $n+3(n-1)$ th singular value of $A_0$ (ordered as $S_1(A_0) > S_2(A_0) \ldots $
\begin{equation}
\text{noise level} = \text{SNR}^{-1} =\left ( \frac{S_{n+3(n-1)}(A_0)}{S_{\text{noise}}} \right )^{-1} .
\end{equation}
In~\cite{ammarivolkov2013} and~\cite{ammarivolkov2013b}  the SNR was based instead on the largest singular value $S_1(A_0)$.
\item The frequency of excitation.
\end{enumerate}

\begin{remark}
From the examination of the frequency dependence of the coefficients of ${\mathcal M}[\alpha^{(n)} B^{(n)}]$  we have seen that the real and imaginary parts for different objects $(B_\alpha)^{(n)} = \alpha^{(n)} B^{(n)} + {\vec z}^{(n)}$ are different. Moreover, in general, their imaginary components exhibit resonance behaviour at different (possibly multiple) frequencies.  Consequently, different objects, in general, correspond to different singular values of $A_0$. The presence of multiple objects with the same shape and size, but with different locations, will result in multiplicities of the singular values (in the absence of noise). If only a single frequency is considered, and $S_{\text{\emph{noise}}} $ is chosen based on the largest singular value $S_1(A_0)$, then it will generate a $W$ with Gaussian statistics that are associated with only one of the objects possibly present. If the singular values associated with the other objects are much smaller than  $S_1(A_0)$  it may be difficult to detect the other objects present. In particular, to locate those objects with smaller MPT coefficients  (and hence smaller singular values) at that frequency under consideration.
\end{remark}

We explore this through the following experiment. We simulate excitations and measurements taken at regular intervals on the plane $[-1,1] \times [-1,1] \times \{ 0 \}$ such that $M=K=256$. The dipole moments are chosen as ${\vec p}={\vec q}={\vec e}_3$ so that the plane of all measurement and excitation coils are parallel to this horizontal surface. With these measurements, the location identification of a coin $B_\alpha^{(1)}$ of radius $0.01125\text{m}$ and thickness $3.15 \times 10^{-3}\text{m}$ with $\sigma_*^{(1)}=15.9 \times 10^6 \text{S/m} $ and $\mu_*^{(1)}=\mu_0$ and a tetrahedron $B_\alpha^{(2)}$ with vertices $(5.77\times10^{-3},0,0) \text{m}$, $(-2.88,5,0)\times 10^{-3} \text{m}$, $(-2.88,-5,0)\times 10{-3} \text{m}$ and $(0,0,-8.16 \times10^{-3})\text{m}$ and material properties $\sigma_*^{(2)}=4.5 \times 10^6 \text{S/m} $ and $\mu_*^{(2)}  = 1.5 \mu_0$ will be considered. The true locations of these objects are assumed to be ${\vec z}^{(1)}=0.1{\vec e}_1 +0.1{\vec e}_2 -0.5{\vec e}_3$ and ${\vec z}^{(2)}=-0.3{\vec e}_1 +0.3{\vec e}_2 -0.5{\vec e}_3$, respectively.
To perform the imaging, noise is added to the simulated $A_0$ to create $A$ and the image functional $I_{MU}$ is evaluated for different positions ${\vec z}^s$. To do this, we compute $P = {\mathbb I}_M - W_S W_S^*$ where $W_S$ are the first $3N$ singular vectors of $A$, which are chosen based on the magnitudes of the singular values and thereby allows us to also predict the number of objects $N$ present.  

 We first consider location identification at a frequency $f=1 \times 10^5\text{Hz}$, which is close to the resonance peaks for the two objects, and consider the singular values of $A_0$ and $A$ in Figure~\ref{fig:singularvalues}. At this frequency, $S_n(A_0)$, $n=1,2,3$ are associated with the coin and $S_n(A_0)$, $n=4,5,6$ with the tetrahedron. Without noise, $A= A_0$ and the 6 {\em physical} singular values are clearly distinguished, but, by considering a noise level of $1\%$ so that $S_{\text{noise}}=0.01 S_1(A_0)$, it is no longer possible to distinguish $S_n(A)$, $n=4,5,6$ from  the noisy singular values. On the other hand, by setting $S_{\text{noise}}=0.01 S_4(A_0)$, or even $S_{\text{noise}}=0.1 S_4(A_0)$, we can distinguish all 6 singular values from the noise. This means that with $S_{\text{noise}}=0.01 S_1(A_0)$ we expect to only locate the coin, but with  $S_{\text{noise}}=0.01 S_4(A_0), 0.1 S_4(A_0)$  we expect to find both objects.
\begin{figure}
\begin{center}
$\begin{array}{cc}
\includegraphics[width=0.5\textwidth]{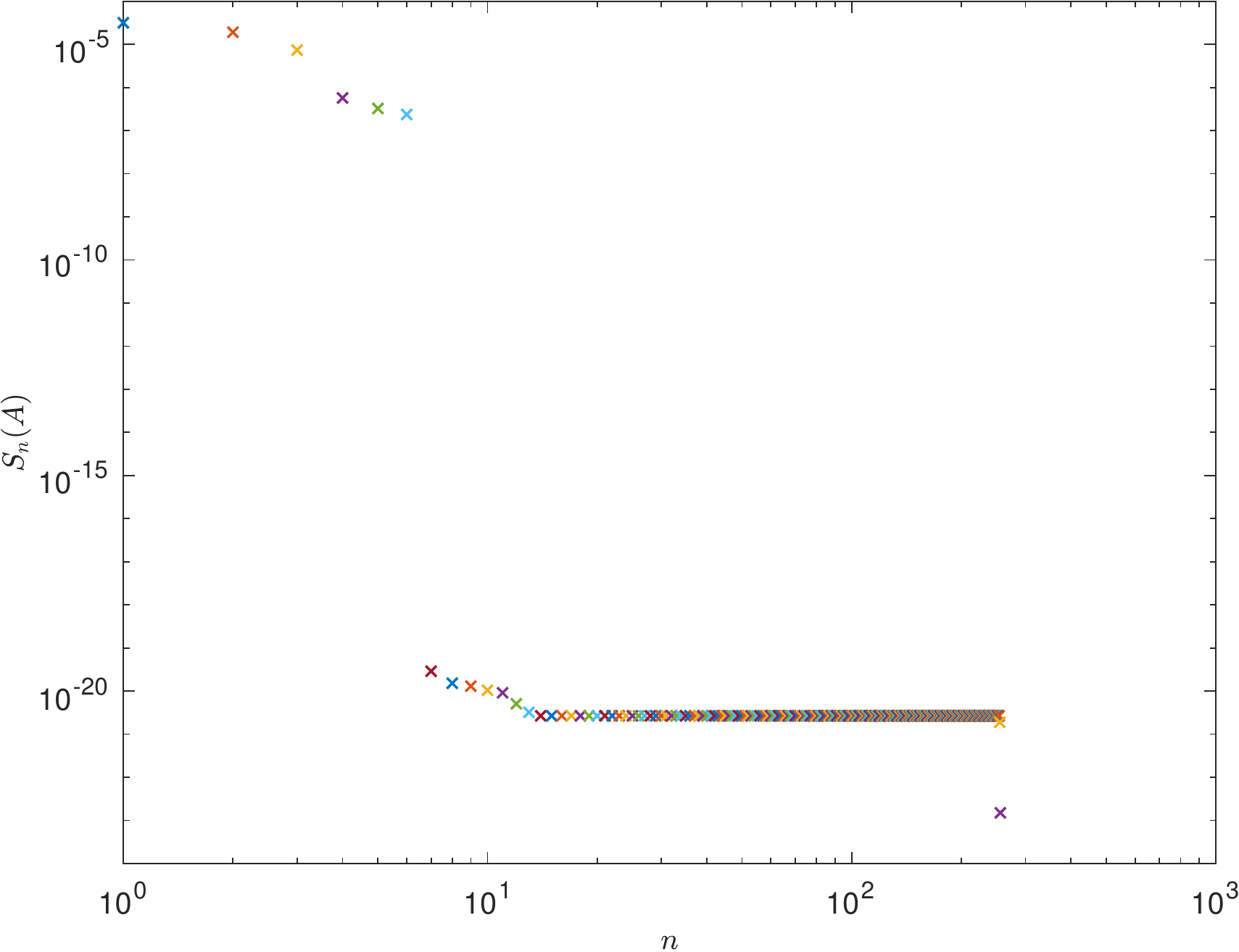} & 
\includegraphics[width=0.5\textwidth]{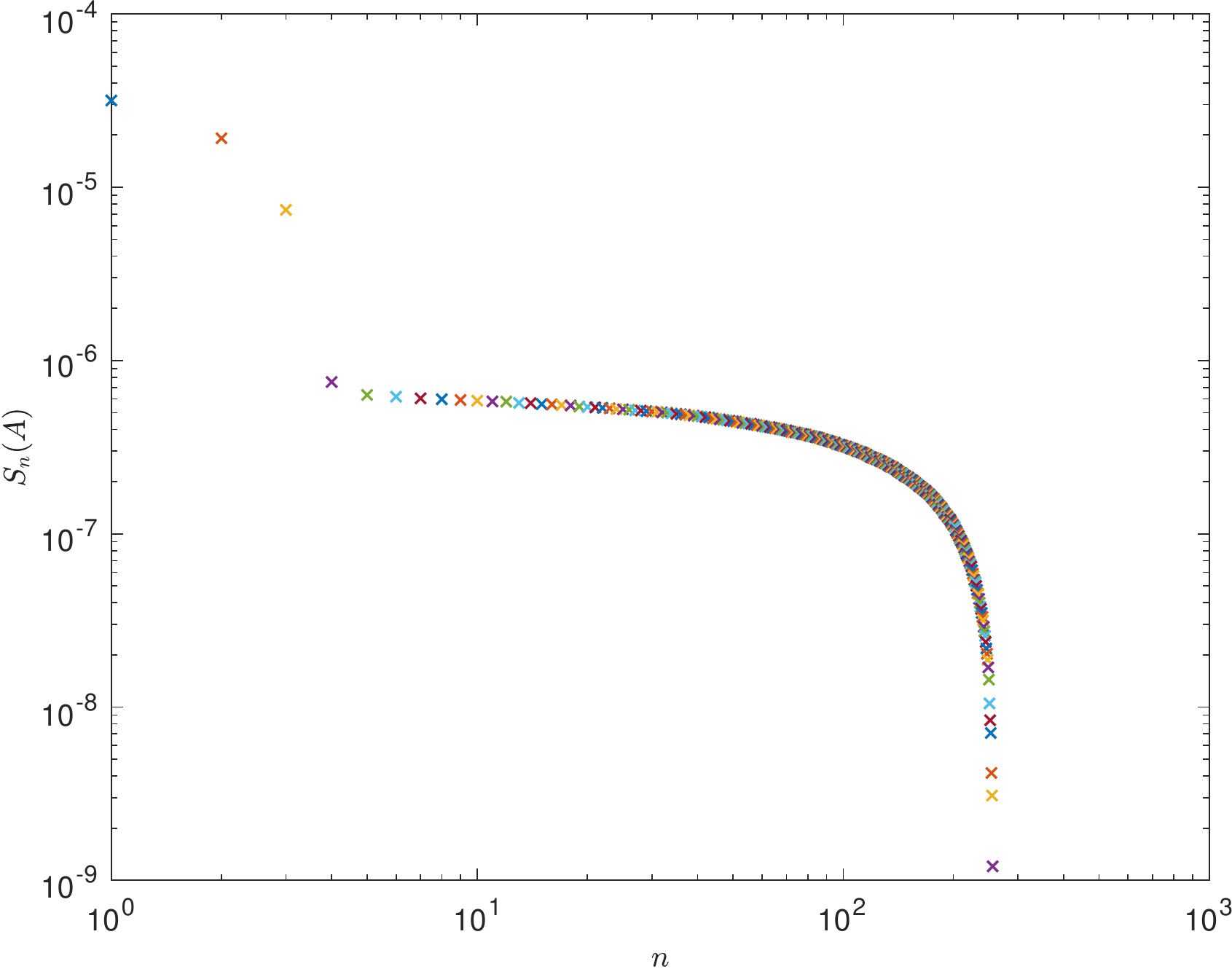} \\
S_{\text{noise}}=0 & S_{\text{noise}}=0.01 S_1(A_0) \\
\includegraphics[width=0.5\textwidth]{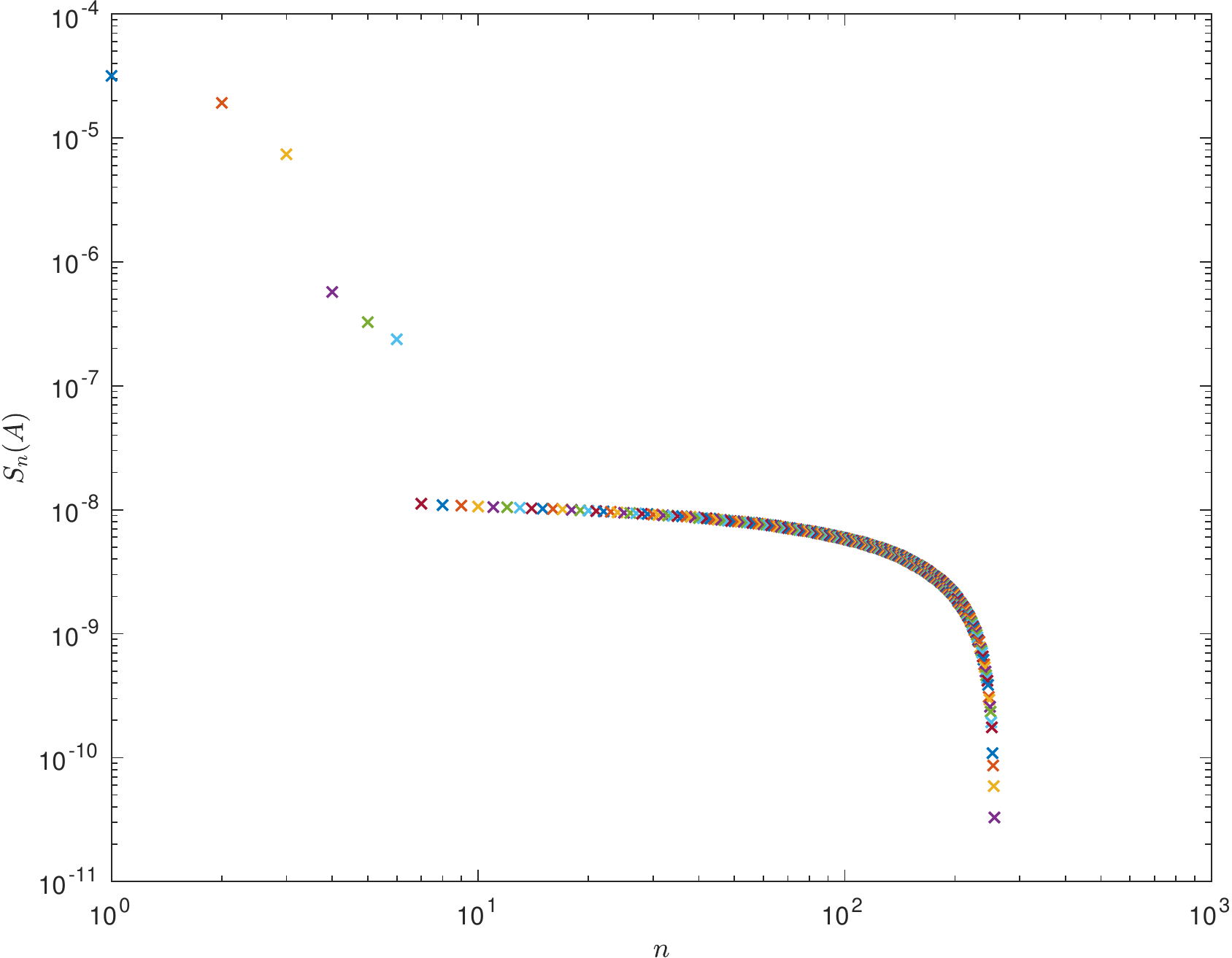} &
\includegraphics[width=0.5\textwidth]{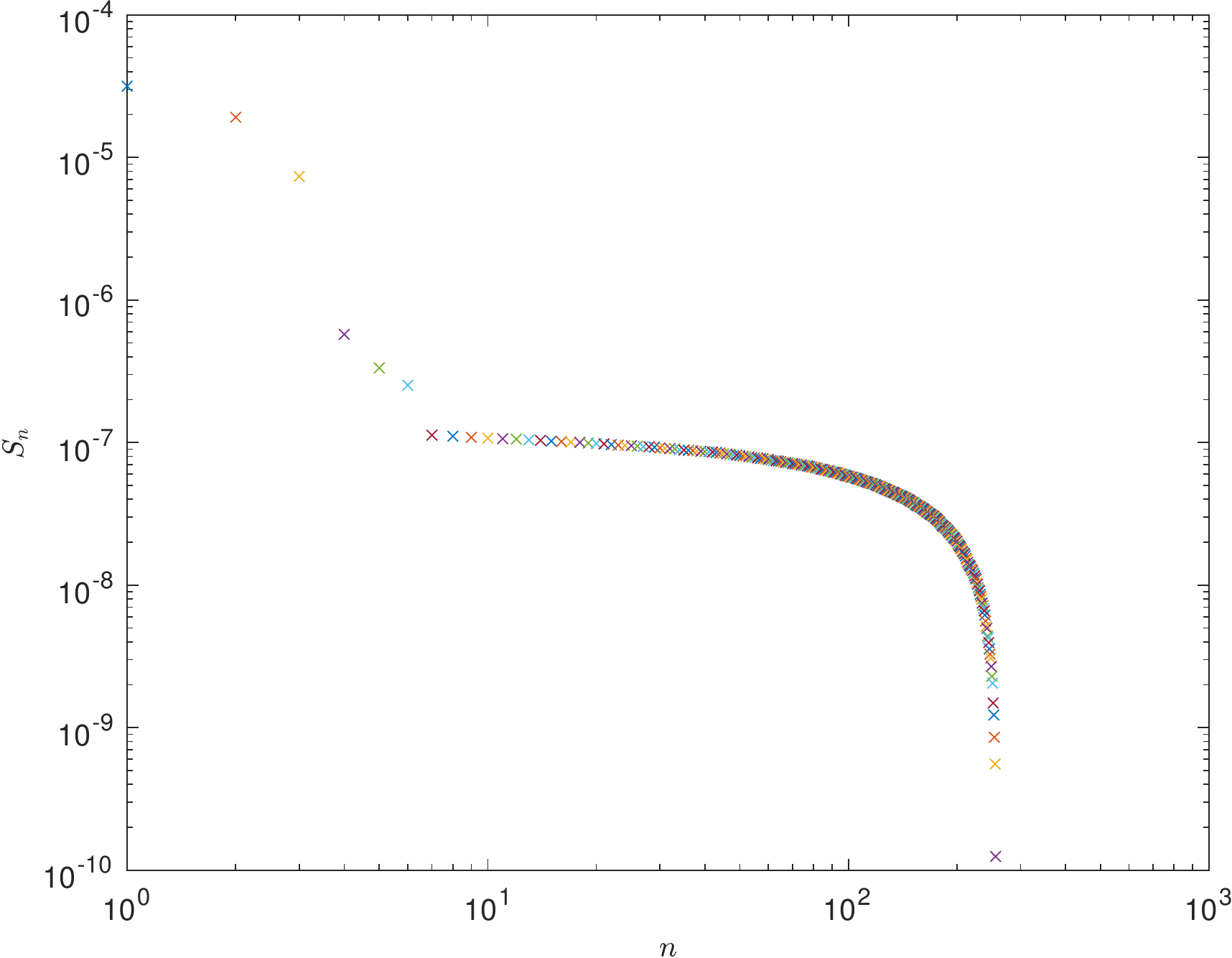} \\
 S_{\text{noise}}=0.01 S_4(A_0) & S_{\text{noise}}=0.1 S_4(A_0)
\end{array}$
\end{center}
\caption{Singular values $S_n(A)$: Evaluated for different levels of noise for identifying a coin and tetrahedron at $f=1 \times 10^5\text{Hz}$.} \label{fig:singularvalues}
\end{figure}
This is confirmed in Figure~\ref{fig:imfunf1_15e5} where we plot $I_{MU}$ on the plane $-0.5{\vec e}_3$. We observe that for $S_{\text{noise}}=0.01 S_1(A_0)$ we can only locate the coin, for $S_{\text{noise}}=0.01 S_4(A_0)$ we can locate both the coin and the tetrahedron and even by increasing the noise level to $10\%$ and setting  $S_{\text{noise}}=0.1 S_4(A_0)  $ both objects can still be identified.
\begin{figure}
\begin{center}
$\begin{array}{cc}
\includegraphics[width=0.5\textwidth]{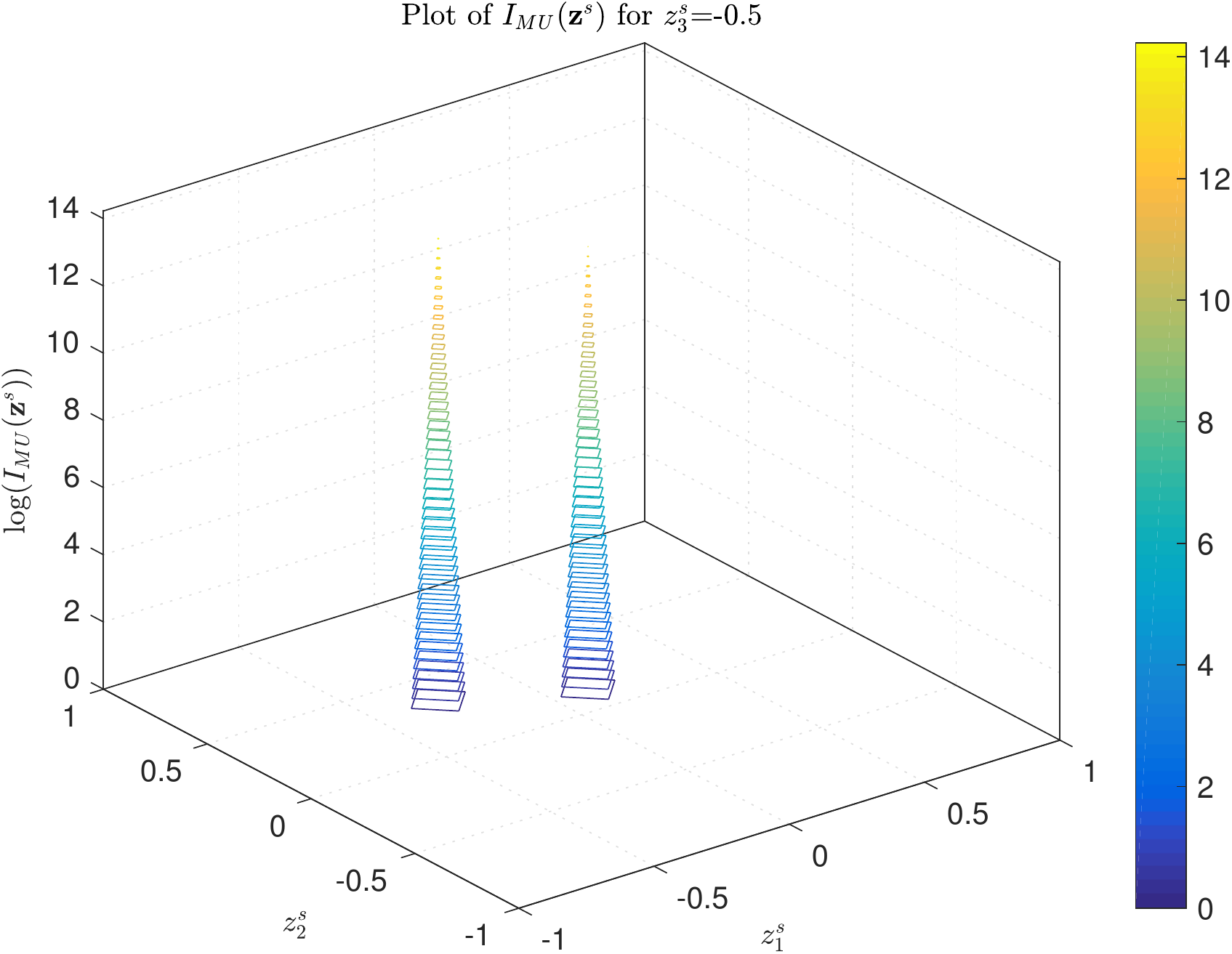} &
\includegraphics[width=0.5\textwidth]{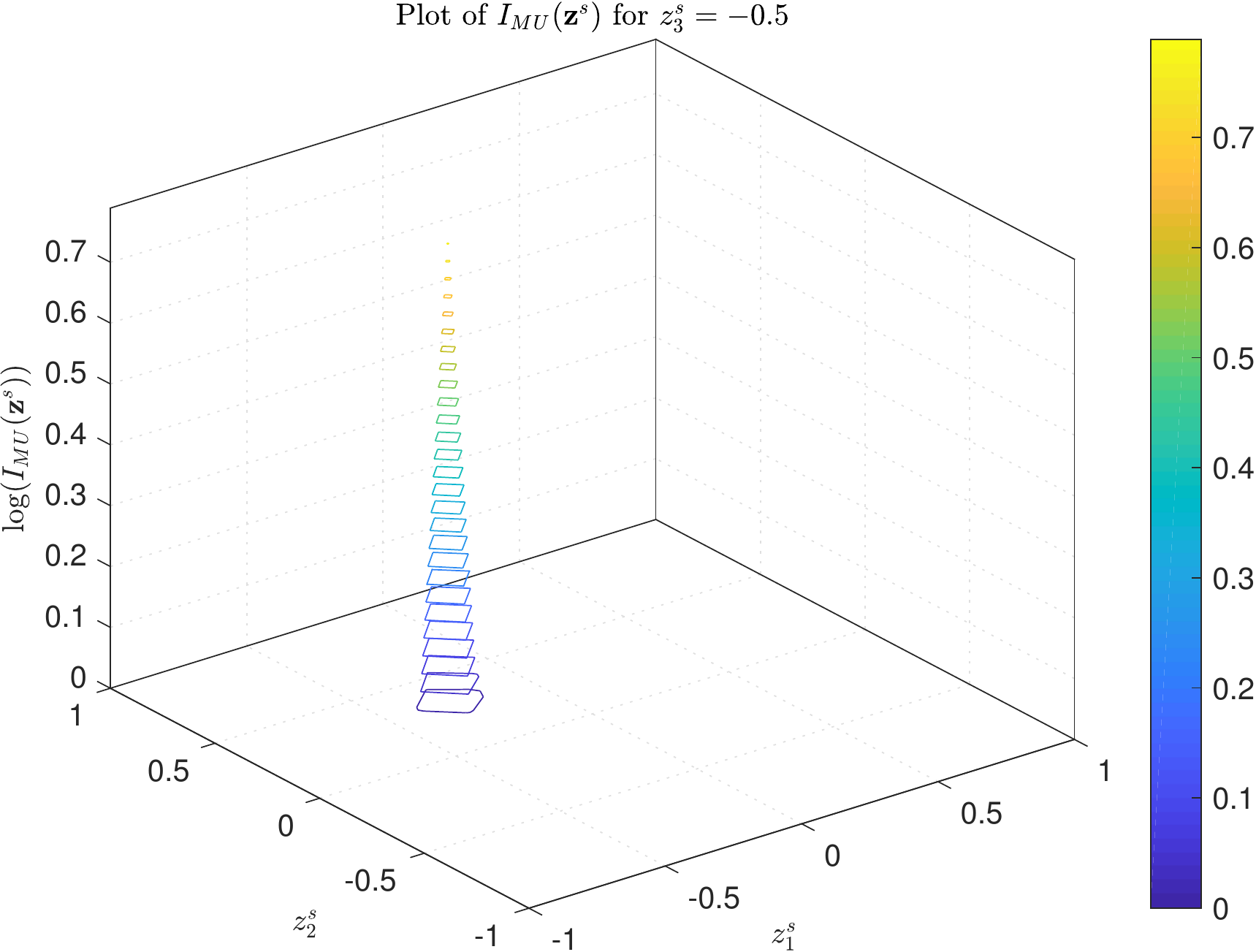} \\
S_{\text{noise}}=0 & S_{\text{noise}}=0.01 S_1(A_0)  \\
\includegraphics[width=0.5\textwidth]{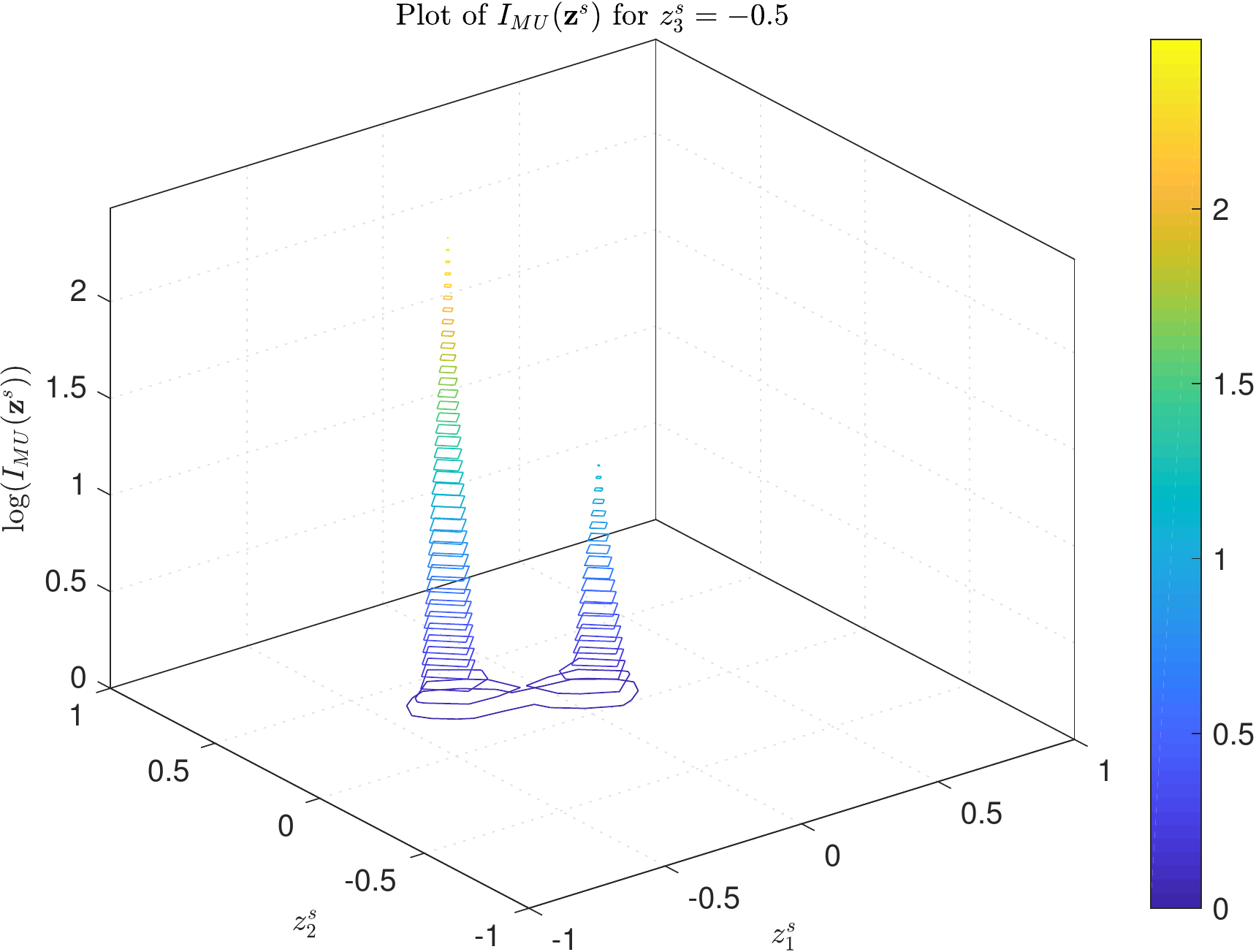} &
\includegraphics[width=0.5\textwidth]{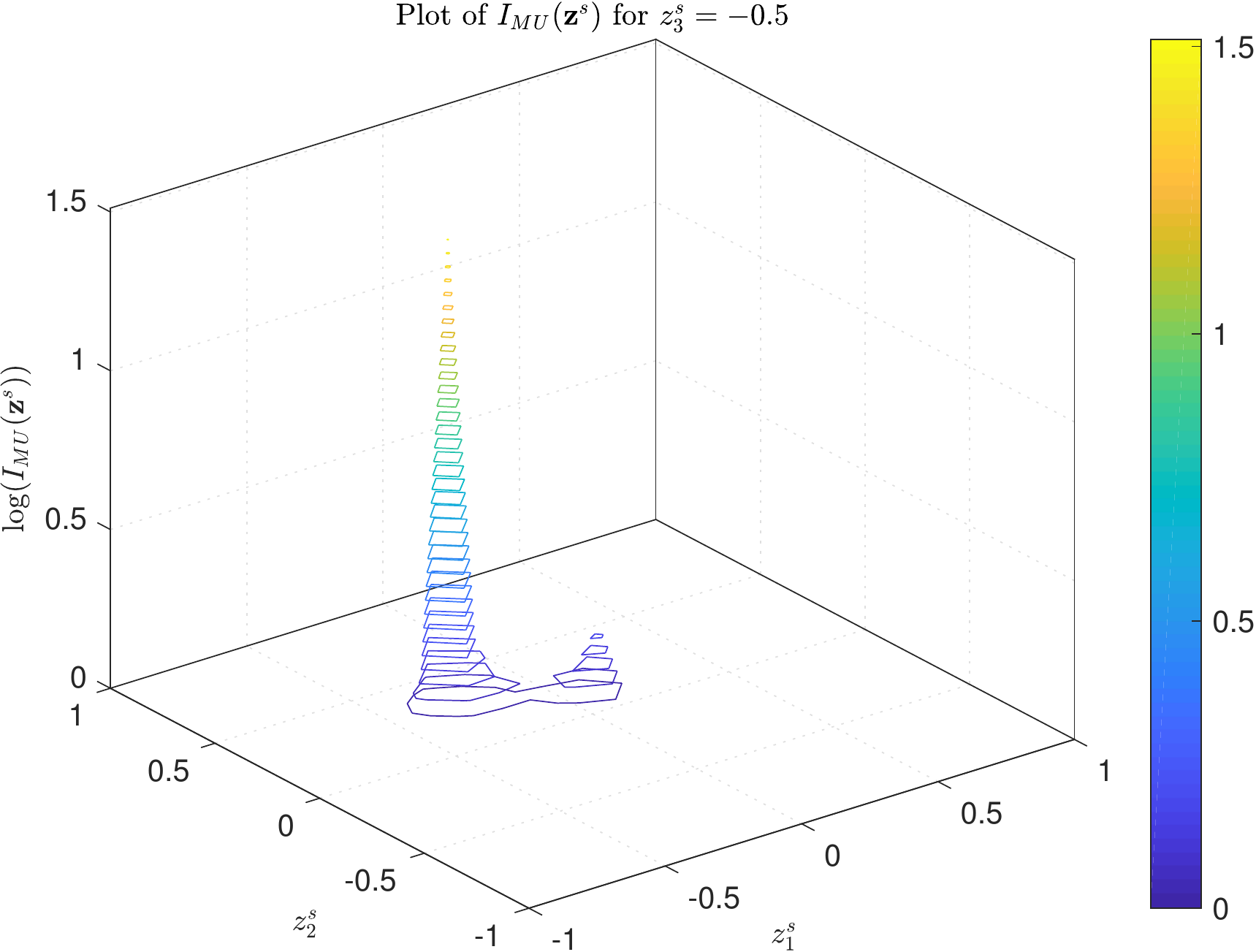} \\
 S_{\text{noise}}=0.01 S_4(A_0) & S_{\text{noise}}=0.1 S_4(A_0)
 \end{array}$
\end{center}
\caption{The imaging function $I_{MU}$: Evaluated on the plane $-0.5{\vec e}_3$ for different levels of noise for identifying a coin and tetrahedron at $f= 1\times 10^5\text{Hz}$.} \label{fig:imfunf1_15e5}
\end{figure}
On the other hand, choosing the frequency $f=132\text{Hz}$, such that  $S_n(A_0)$, $n=1,2,3$ are associated with the tetrahedron and $S_n(A_0)$, $n=4,5,6$ with the coin, Figure~\ref{fig:imfunf132} shows that the phenomena is reversed, and with a $10\%$ noise level and  $S_{\text{noise}}=0.1 S_4(A_0)  $, only the tetrahedron can be identified at this frequency.
\begin{figure}
\begin{center}
$\begin{array}{cc}
\includegraphics[width=0.5\textwidth]{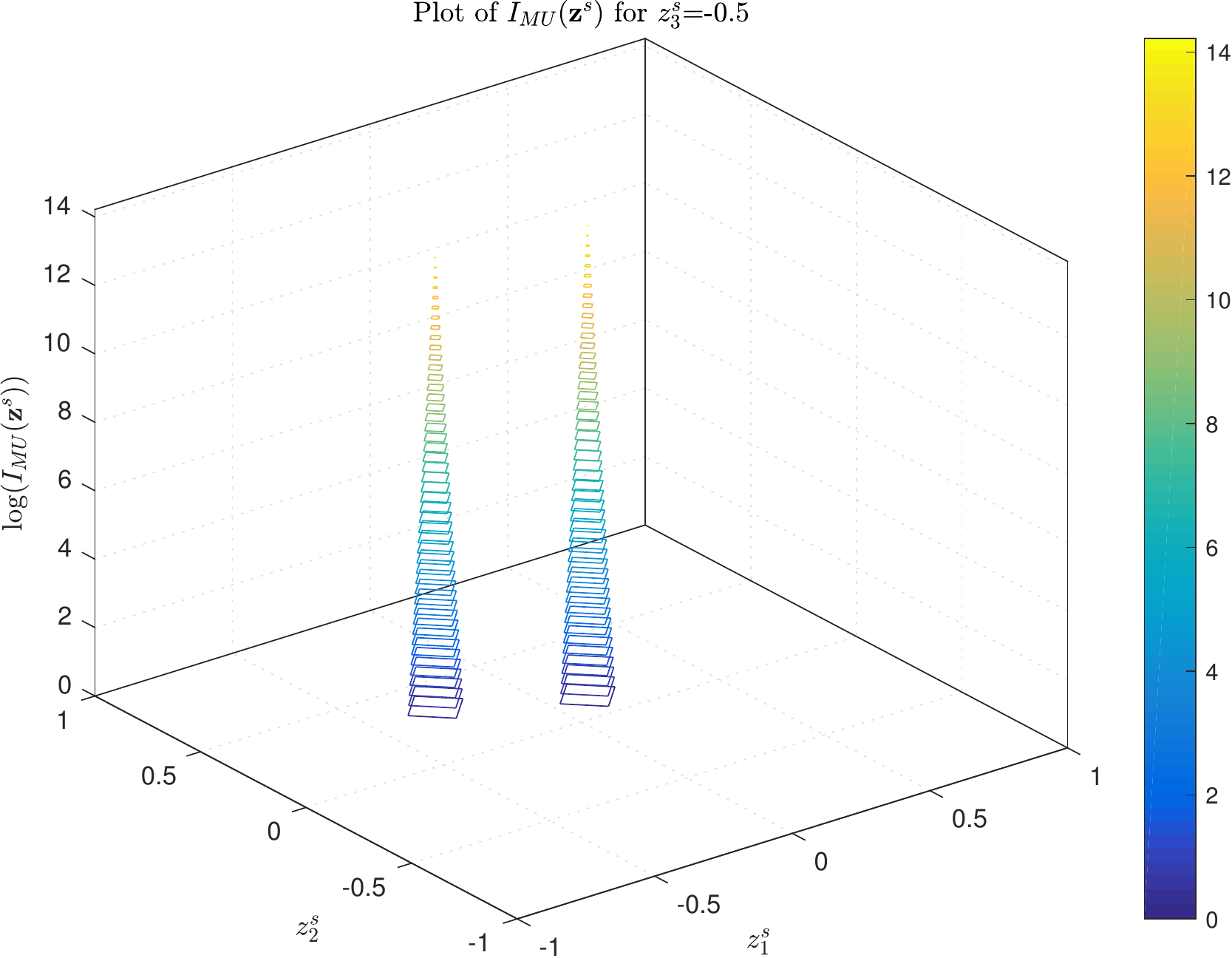} &
\includegraphics[width=0.5\textwidth]{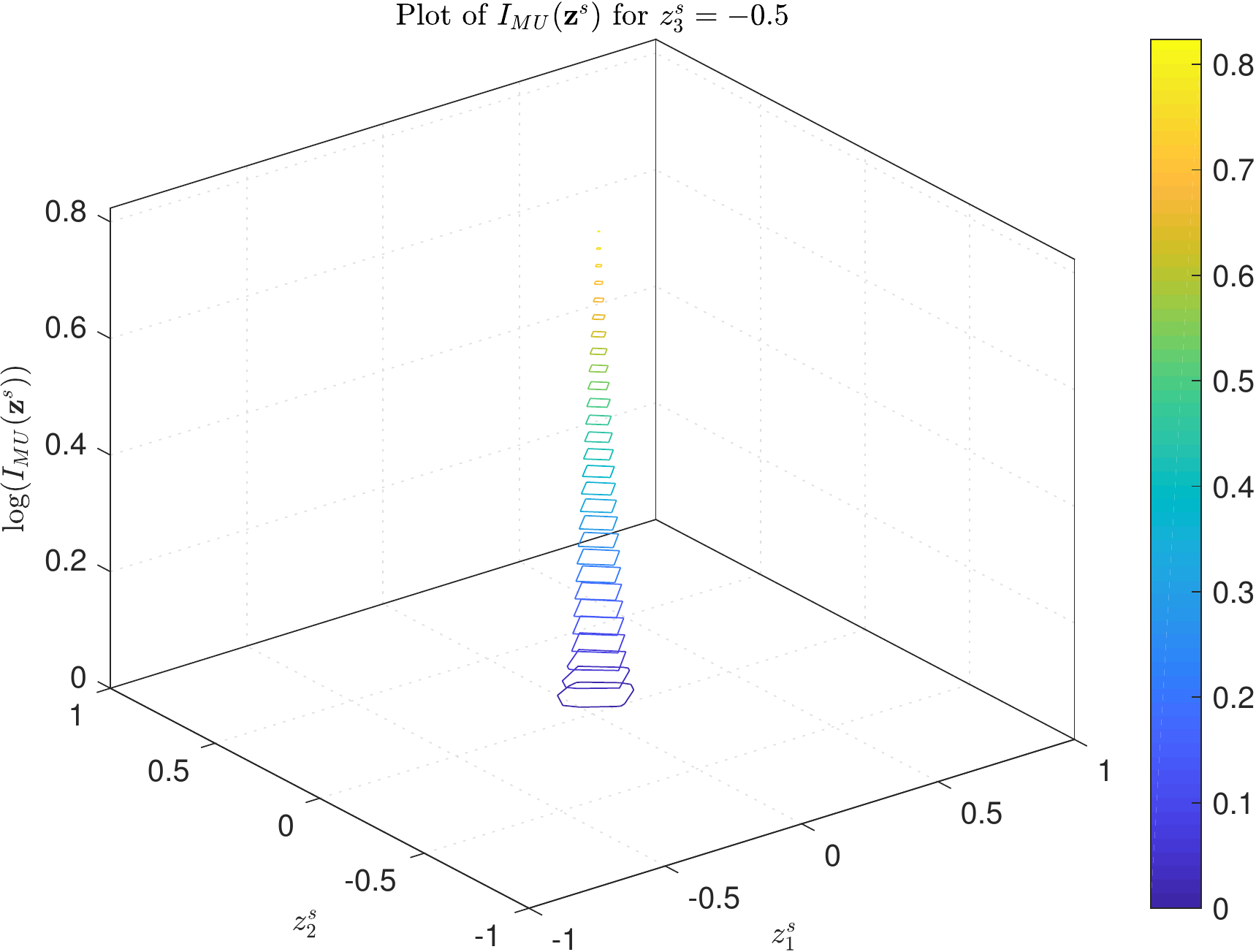} \\
 S_{\text{noise}}=0 &  S_{\text{noise}}=0.01 S_1(A_0) \\
\includegraphics[width=0.5\textwidth]{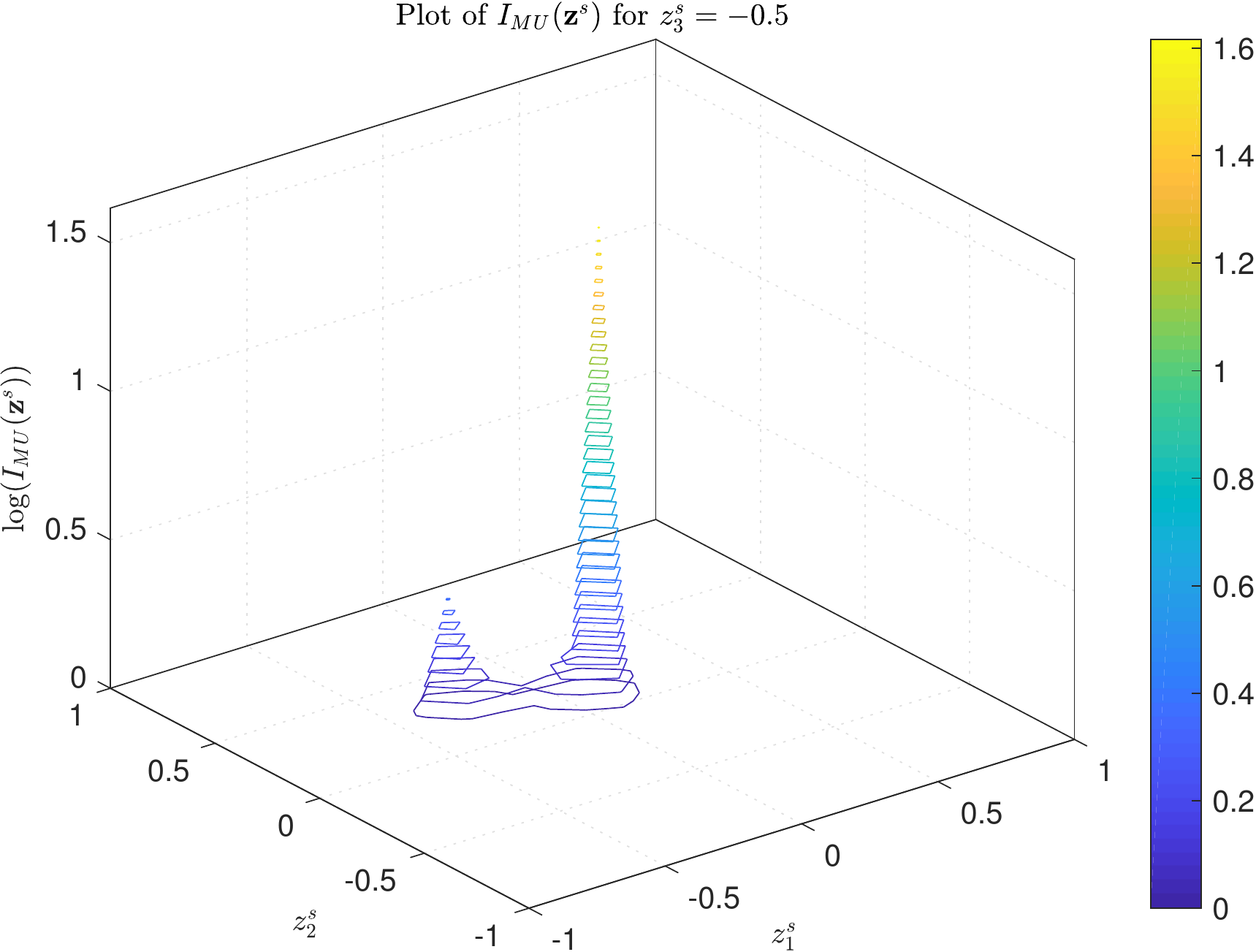} &
\includegraphics[width=0.5\textwidth]{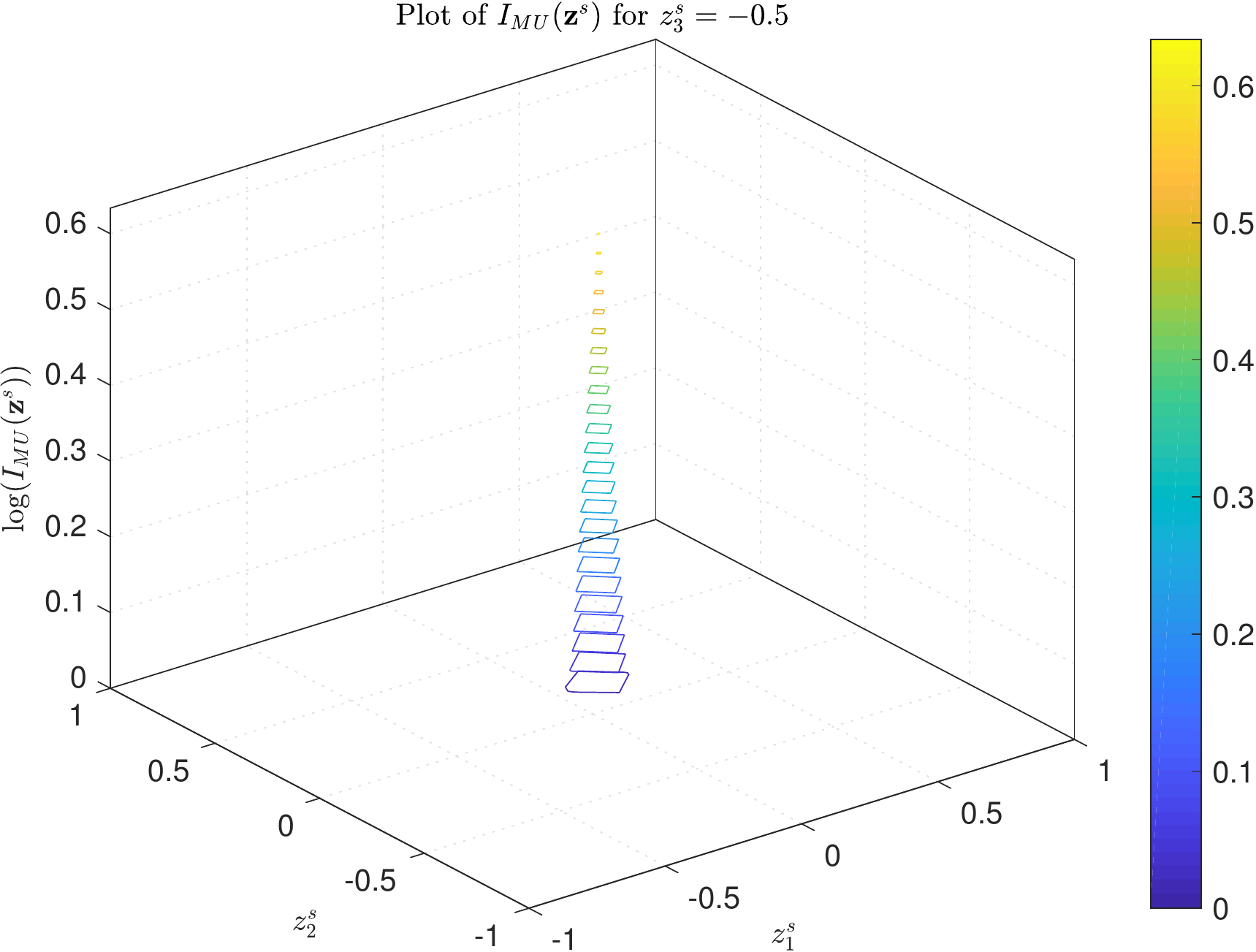} \\
 S_{\text{noise}}=0.01 S_4(A_0) &  S_{\text{noise}}=0.1 S_4(A_0) 
\end{array}$
\end{center}
\caption{The imaging function $I_{MU}$: Evaluated on the plane $-0.5{\vec e}_3$ for different levels of noise for identifying a coin and tetrahedron at $f=132 \text{Hz}$.} \label{fig:imfunf132}
\end{figure}

 \subsection{Object Identification}

 A dictionary-based classification technique for individual object identification has been proposed by Ammari {\em et al.}~\cite{ammarivolkov2013b} and this easily extends to the identification of multiple inhomogeneous objects. We propose a slight variation on that proposed by Ammari {\em et al.}, which uses the eigenvalues of the real and imaginary parts of the MPT as a classifier as oposed to its singular values at a range of frequencies. The motivation for this is the richness of the frequency spectra of the eigenvalues, as shown in Section~\ref{sect:freq}, and that it provides an increased number of values to classify each object. We also propose a strategy in which objects are put in to canonical form before forming the dictionary. The algorithm comprises of two stages as described below.

 \subsubsection{Off-line Stage}
In the off-line stage, given a set of $N_{\text{candidate}}$ candidate objects (which can include both homogenous and inhomogeneous objects), we put them in canonical form $(B_\alpha)^{(i)} = \alpha^{(i)} B^{(i)} +{\vec z}^{(i)}$, $i=1,\ldots, N_{\text{candidate}}$ by ensuring that the origin for ${\vec \xi}^{(i)}$ in $B^{(i)}$ coincides with the centre of mass of $B^{(i)}$ and the object's size $\alpha^{(i)}$ is chosen such that $|{\mathcal N}^0[\alpha^{(i)}{ B}^{(i)}]|=1   $~\footnote{For inhomogeneous objects we require $|{\mathcal N}^0[\alpha^{(i)}{\vec B}^{(i)}]|=1$ and
we replace $(B_\alpha)^{(i)}$ by  ${\vec B}_\alpha^{(i)} = \alpha^{(i)} {\vec B}^{(i)} +{\vec z}^{(i)}$, $B^{(i)}$ by ${\vec B}^{(i)}$ as well as ensuring the centre of mass coincides with the centre of mass of ${\vec B}^{(i)} = \bigcup_{n=1}^N B^{(i,n)}$.}  where ${\mathcal N}^0[\alpha^{(i)} {B}^{(i)}]=   {\mathcal T}[\alpha^{(i)} { B}^{(i)}] $ in the case of a homogenous object and corresponds to the P\'oyla-Szeg\"o tensor as well as being the characterisation for $\sigma_*=0$ for this object~\footnote{If $\mu_* =\mu_0$ we choose the object size by requiring the high conductivity limit to have unit determinant.}. In the case of an object with homogenous materials, the coefficients of ${\mathcal M}[\alpha^{(i)} B^{(i)}]$ are computed  by solving the transmission problem (\ref{eqn:transproblemthetar2}) using finite elements and then applying (\ref{eqn:mcheckmult}) and, in the case of an inhomogeneous object  (\ref{eqn:transproblemthetar3}) and (\ref{eqn:mcheckcup}) are used. In each case, the eigenvalues $\lambda_{\mathcal R} ( \alpha^{(i)} B^{(i)}, \omega_j)$ and $\lambda_{\mathcal I} ( \alpha^{(i)} B^{(i)} , \omega_j)$ are obtained for a range of frequencies $\omega_j$ and 
\begin{align}
D_i =&\{ \lambda_{\mathcal R} ( \alpha^{(i)} B^{(i)} , \omega_j) , \lambda_{\mathcal I} ( \alpha^{(i)} B^{(i)} , \omega_j), j=1,\ldots,N_\omega  \}   \nonumber \\ 
 & /  \max_{k=1,\ldots,N_\omega} ( |\lambda_{\mathcal R} ( \alpha^{(i)} B^{(i)} , \omega_k)| ,| \lambda_{\mathcal I} ( \alpha^{(i)} B^{(i)} , \omega_k)| )  \nonumber ,
\end{align}
forms the $i$th element of the dictionary 
\begin{equation}
{\mathcal D} = \{ D_1 ,D_2 ,\ldots , D_{N_{\text{candidate}} } \} \nonumber  .
\end{equation}

 \subsubsection{On-line Stage}
 In an extension to~\cite{ammarivolkov2013b}, the MPT coefficients for each of the targets $(T_\alpha)^{(i)}$, $i=1,\ldots,N_{\text{target}}$ can be recovered from the same data used to identify the number and locations of objects. Although, to do so, it is important to ensure that the dipole moments of the coils are chosen such that all the $6N_{\text{target}}$  coefficients can be recovered from the measured data~\cite{ledgerlionheart2018}. The coefficients are then the solution of the least squares problem
  \begin{equation}
( {\mathcal M} [(T_\alpha)^{(1)},\omega_j ], \cdots , {\mathcal M} [(T_\alpha)^{(N_{\text{objects}})},\omega_j ]) = \displaystyle \hbox{arg} \min_{\frak M} \| A(\omega_j) - L( {\frak M}) \|, \nonumber
 \end{equation} 
 which is repeated for $j=1,\ldots,N_\omega$.
 
 Then, for each target $(T_\alpha)^{(i)}$, we determine
 \begin{equation}
\hat {D}_i =\{ \lambda_{\mathcal R} ( (T_\alpha)^{(i)}, \omega_j) , \lambda_{\mathcal I} (( T_\alpha)^{(i)}, \omega_j) \} / \ \max_{k=1,\ldots,N_\omega} ( |\lambda_{\mathcal R} ( (T_\alpha)^{(i)}, \omega_k)| ,| \lambda_{\mathcal I} ( (T_\alpha)^{(i)}, \omega_k)| ) , \nonumber
\end{equation}
and find the closest match to $\hat{D}_i$ within the dictionary ${\mathcal D}$~\cite{ammarivolkov2013b}. Notice the target could also be inhomogeneous in which case $(T_\alpha)^{(i)}$ is replaced by ${\vec T}_\alpha^{(i)}$.

 \subsubsection{Numerical Example}
 As a challenging object identification example, we consider a dictionary consisting of parallelepipeds described in Section~\ref{sect:freq}, which consist of either two regions ${\vec P}_1:={\vec B}=B^{(1)} \bigcup B^{(2)}$ with ${\vec B}_\alpha = \alpha {\vec B}= \alpha (B^{(1)}\cup B^{(2)})$  or three regions ${\vec P}_2:={\vec B}=B^{(1)} \cup B^{(2)}\cup B^{(3)} $ with ${\vec B}_\alpha = \alpha {\vec B}= \alpha (B^{(1)}\cup B^{(2)}\cup B^{(3)} )$, and vary the material properties according to the descriptions previously described. We also consider the limiting case where the two (three) regions have the same parameters. The dictionary for these objects is generated according to the {\em off-line stage} with $\omega \in 2 \pi ( 2, 300, 4\times 10^3, 5\times 10^4, 2 \times 10^5)\text{rad/s}$, arbitrarily chosen over the frequency spectrum.
 
 For the {\em on-line stage} take ${\mathcal M} [{\vec T}_\alpha^{(i)},\omega_j ]$, $i=1,2$, $j=1,\ldots, N_\omega=5$ to be given by considering targets ${\vec T}_\alpha^{(i)}=\alpha R({\vec P}_i)$ where $R$ is an arbitrary rotation adding noise. In Figure~\ref{fig:classify} we illustrate the algorithms ability to differentiate between these similar objects. The red bars indicate the predicated classification, which is correct for the examples presented (it was also found to be correct for the cases of the other parallelepipeds). We can observe that greatest similarity in terms of the classification is between the two homogeneous parallelepipeds and between the two parallelepipeds with contrasting $\sigma$ and in each case the classification becomes more challenging as the noise level is increased. 
 
  \begin{figure}
 \begin{center} 
 $\begin{array}{ccc}
 \includegraphics[width=0.3\textwidth]{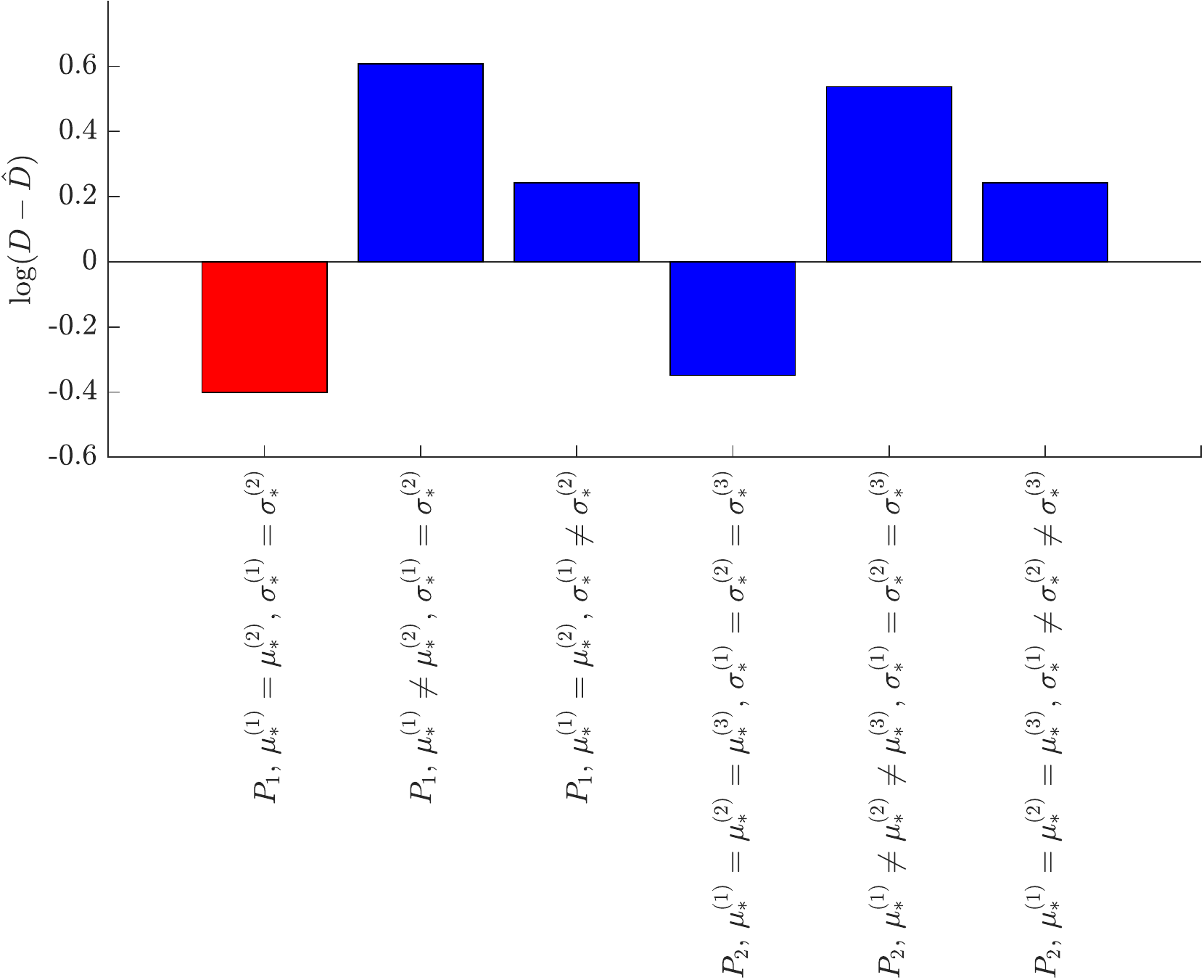} &
 \includegraphics[width=0.3\textwidth]{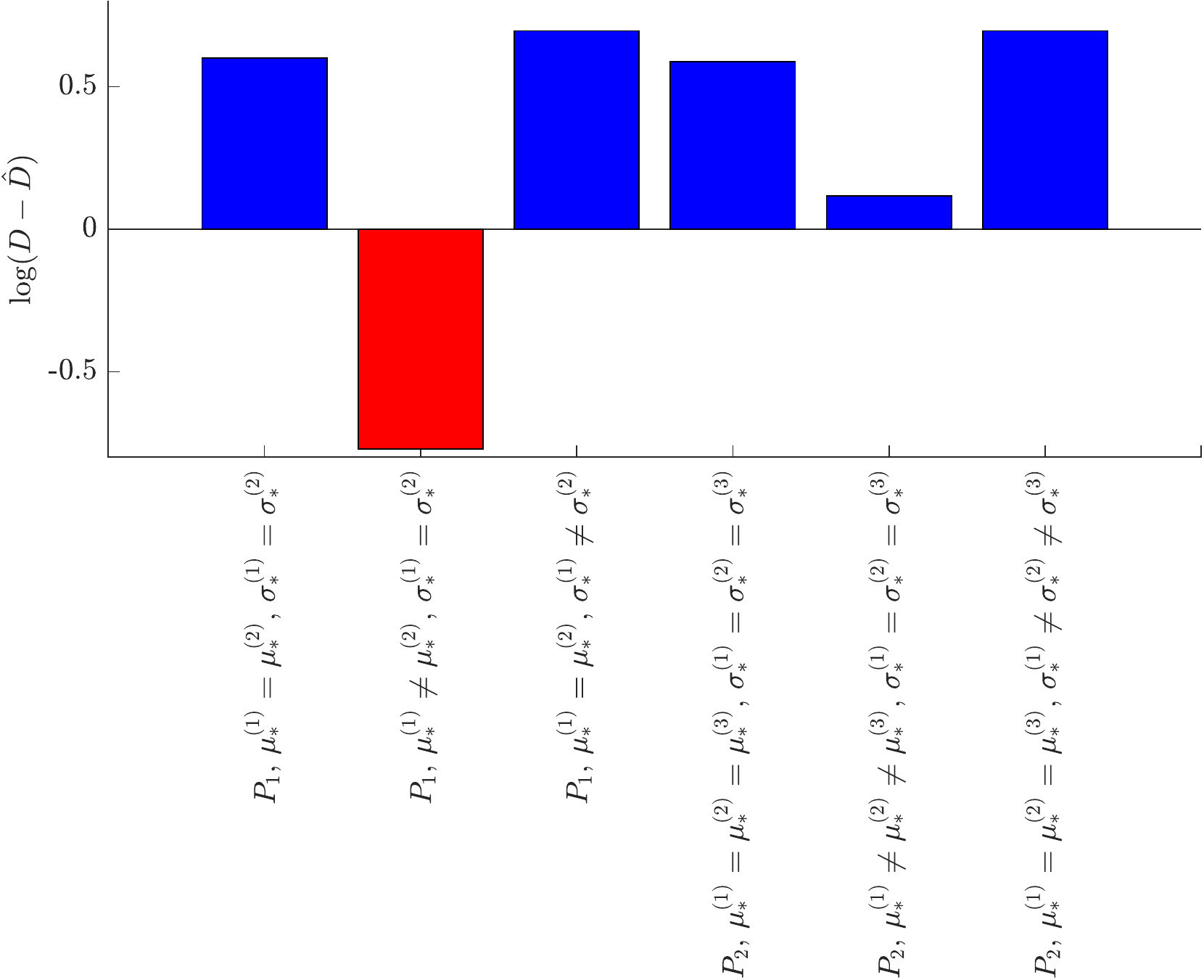} &
 \includegraphics[width=0.3\textwidth]{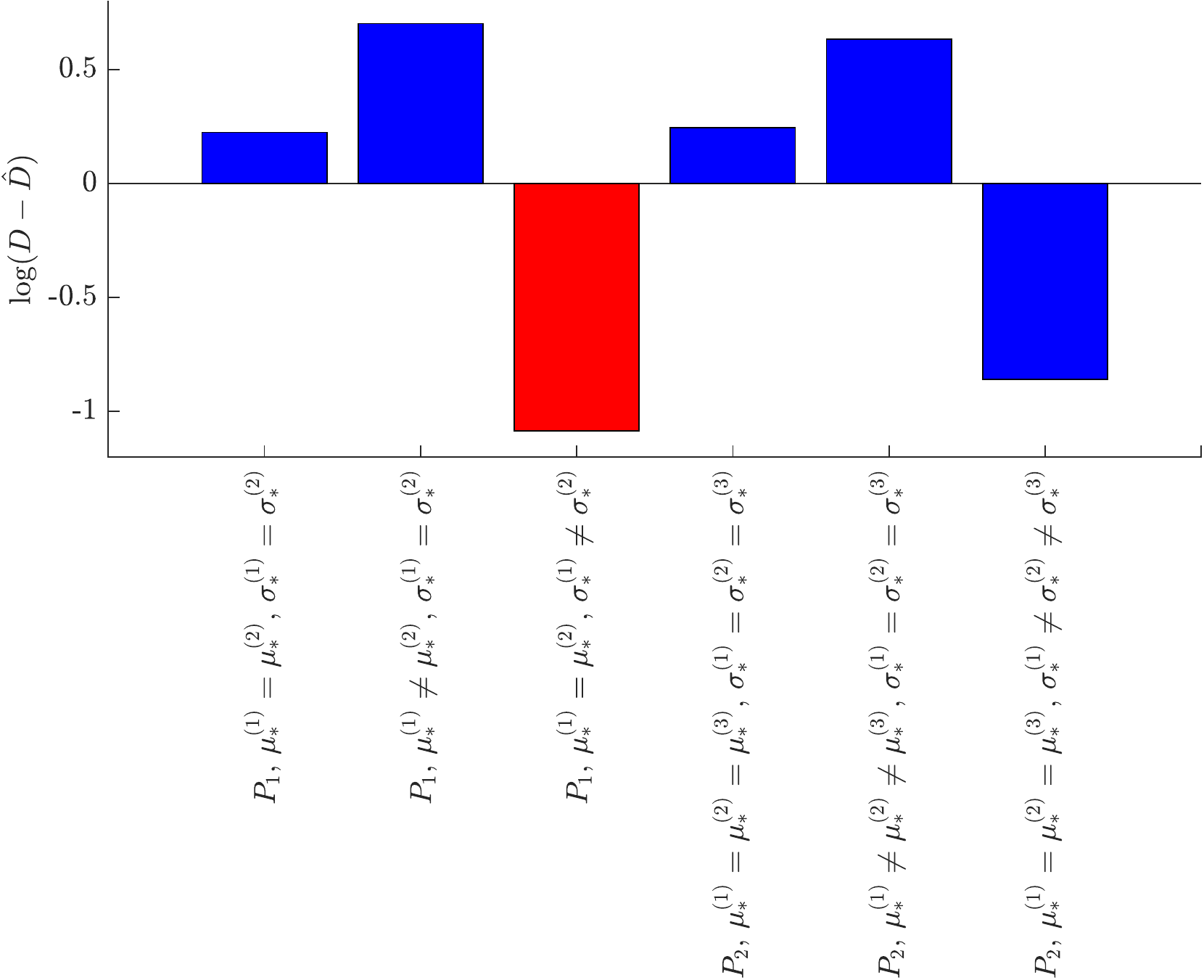} \\
 5\% &  5\% &  5\% \\ 
  \includegraphics[width=0.3\textwidth]{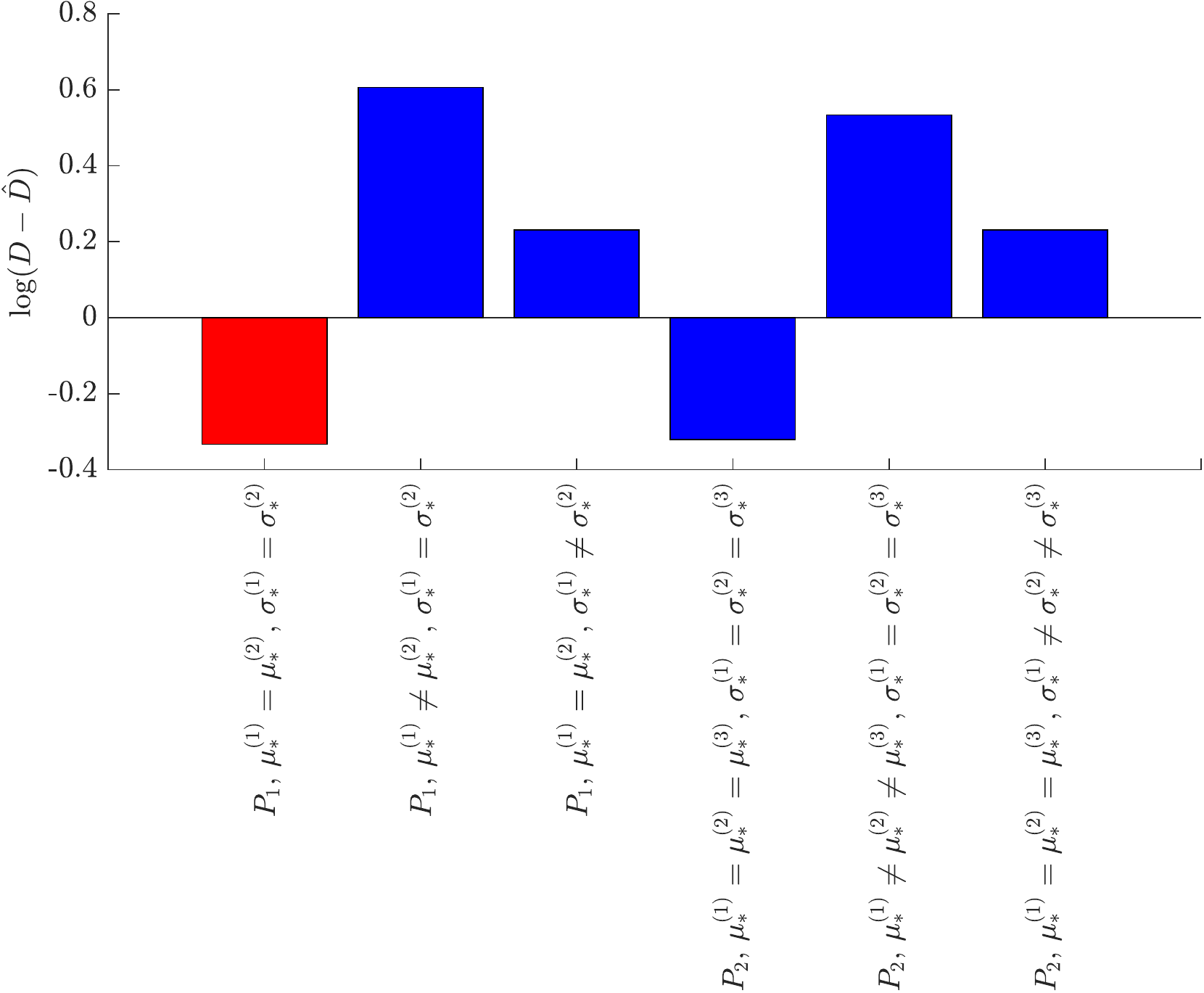} &
 \includegraphics[width=0.3\textwidth]{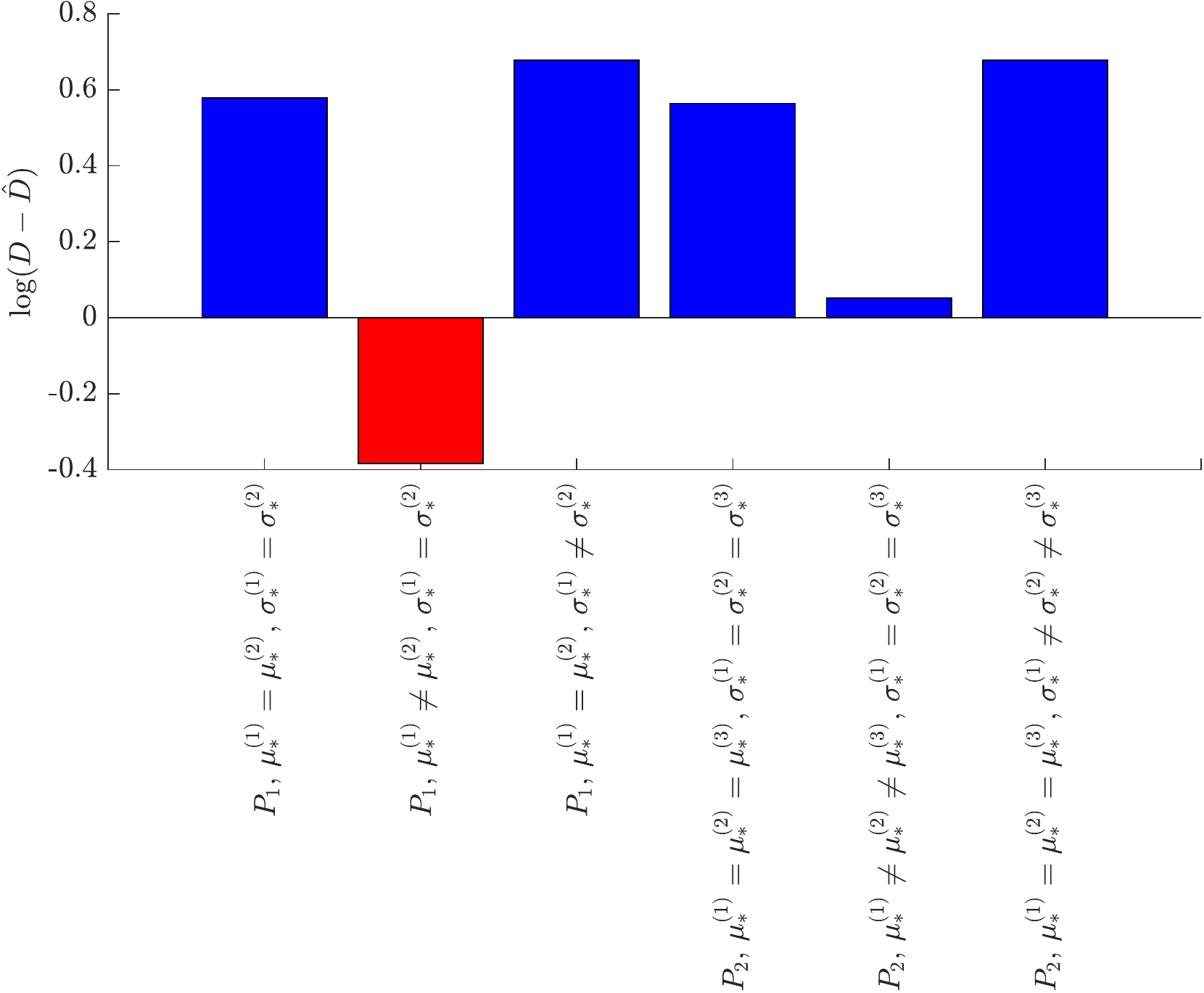} &
 \includegraphics[width=0.3\textwidth]{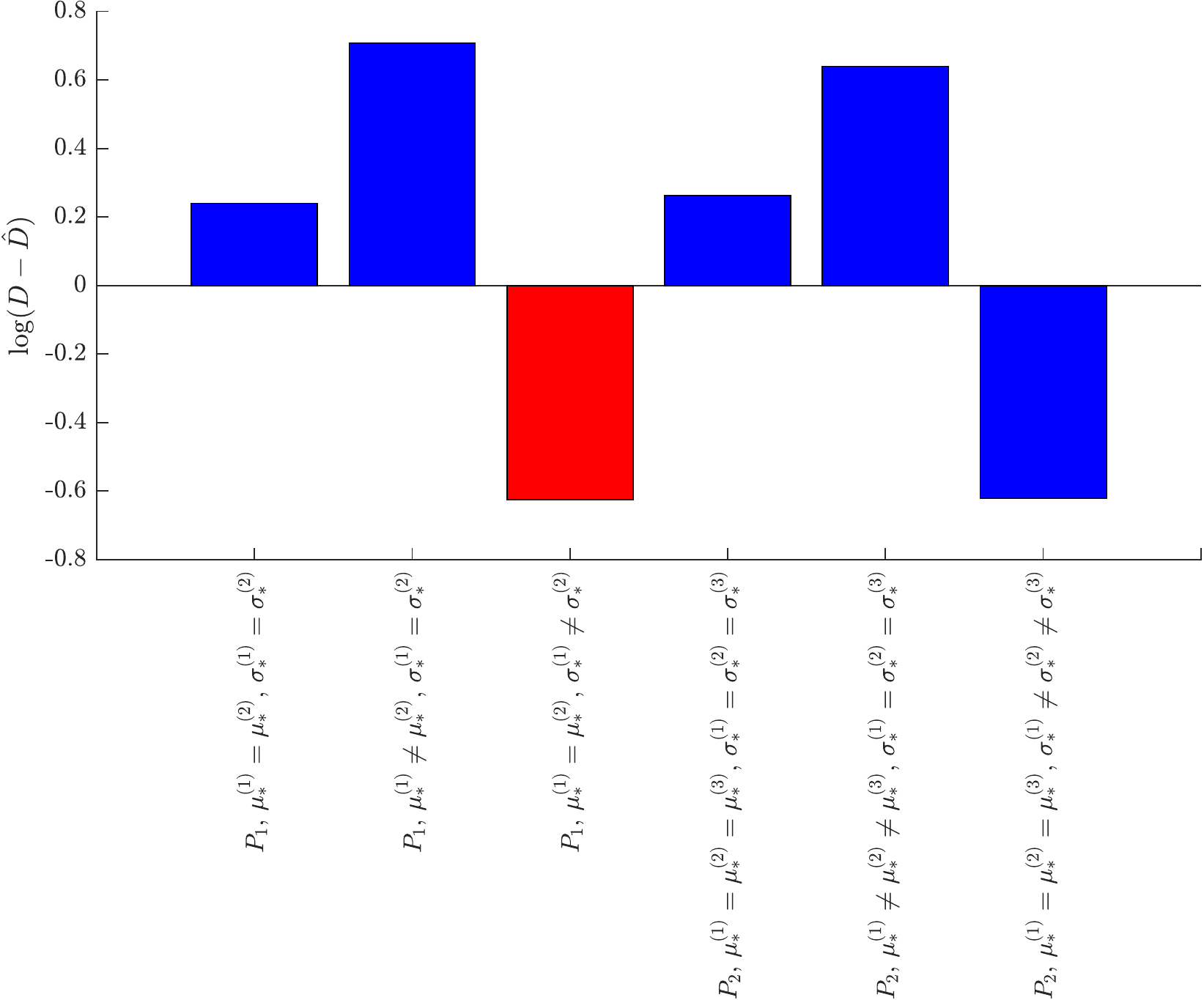} \\
 10\% &  10\% &  10\% 
 \end{array}$
 \end{center}
 \caption{Dictionary classification showing $\log \| {\mathcal D} - \hat{D}_i\|_2 $ : Top row show classification with $5\%$ noise, bottom row with $10\% $ noise, red indicates the predicted object, which is correct in all cases} \label{fig:classify}
 \end{figure}

By increasing the number of frequencies considered so that $N_\omega=7$ with $\omega \in 2 \pi ( 2, 300, 4\times 10^3, 5\times 10^4, 2 \times 10^5, 3 \times 10^6, 4 \times 10^7)\text{rad/s}$ we see in Figure~\ref{fig:classify_moref} that the certainty of the classification is improved for both noise levels.
  \begin{figure}
 \begin{center} 
 $\begin{array}{ccc}
 \includegraphics[width=0.3\textwidth]{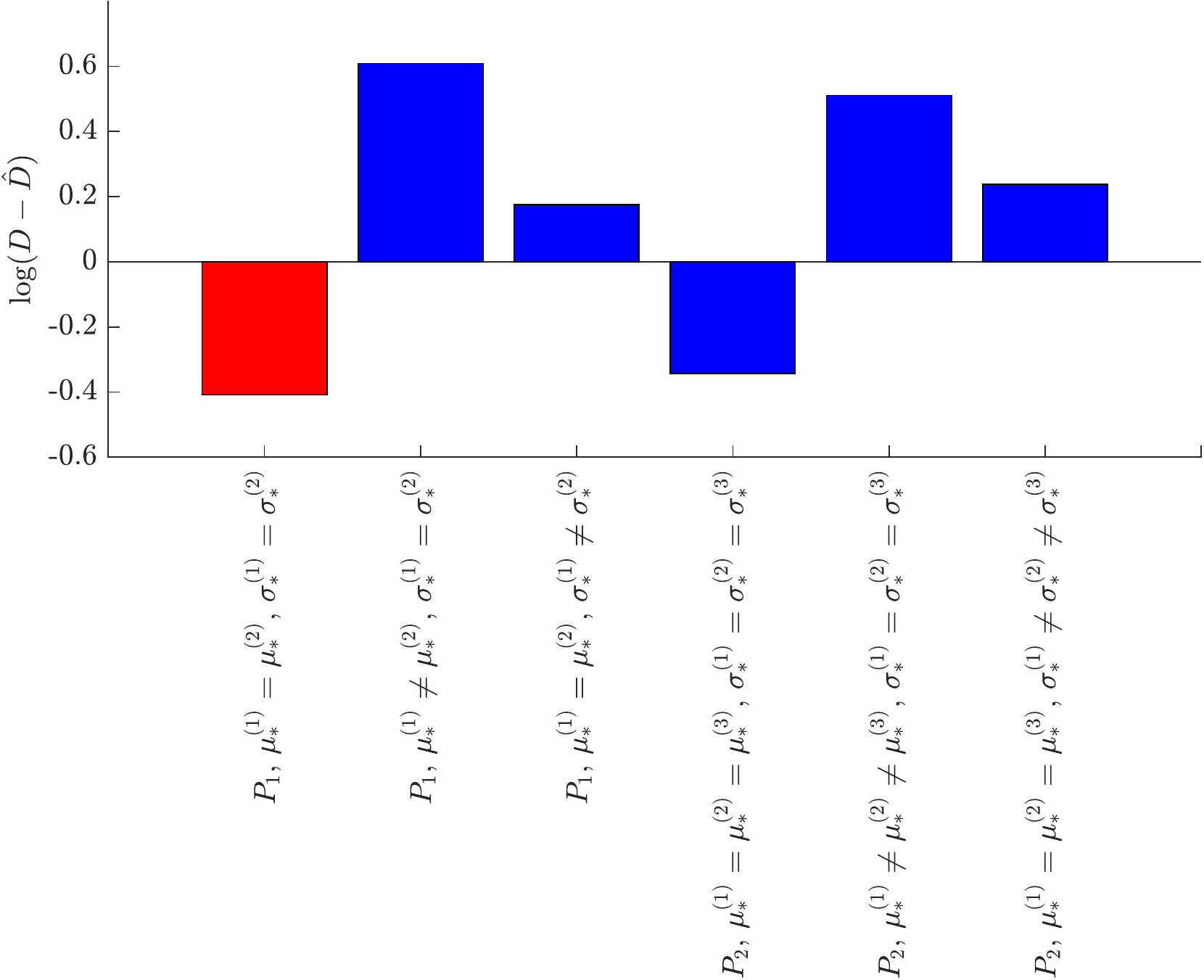} &
 \includegraphics[width=0.3\textwidth]{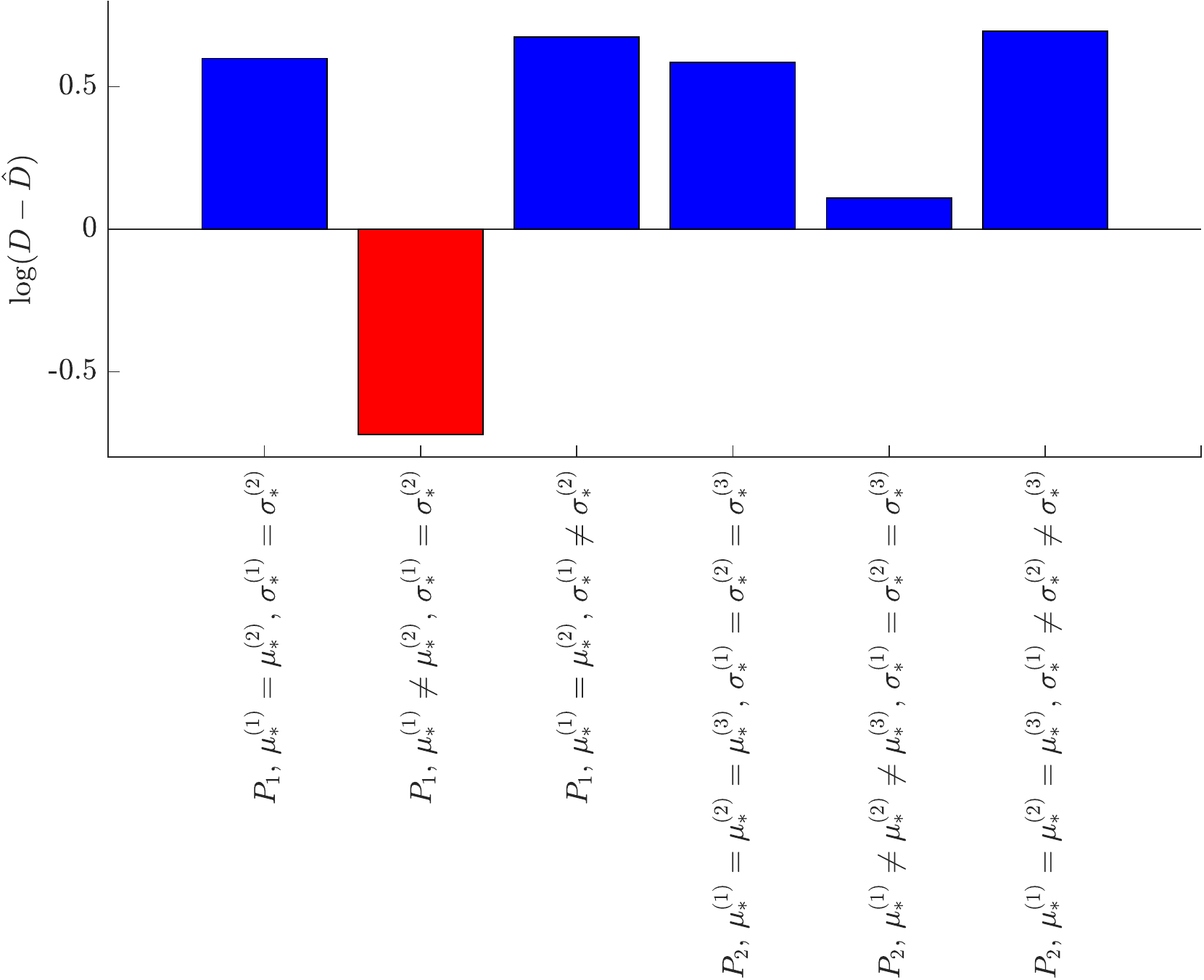} &
 \includegraphics[width=0.3\textwidth]{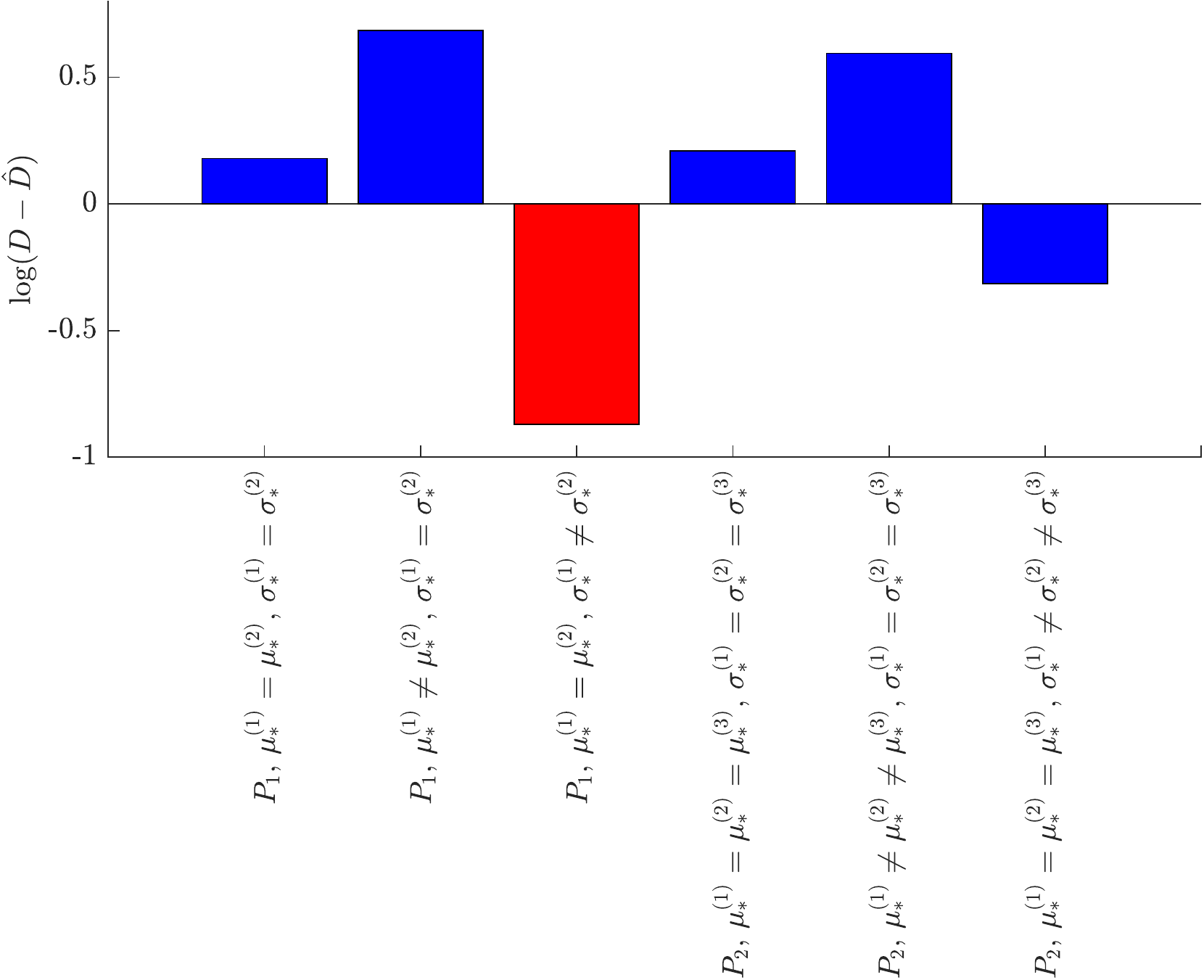} \\
 5\% &  5\% &  5\% \\ 
  \includegraphics[width=0.3\textwidth]{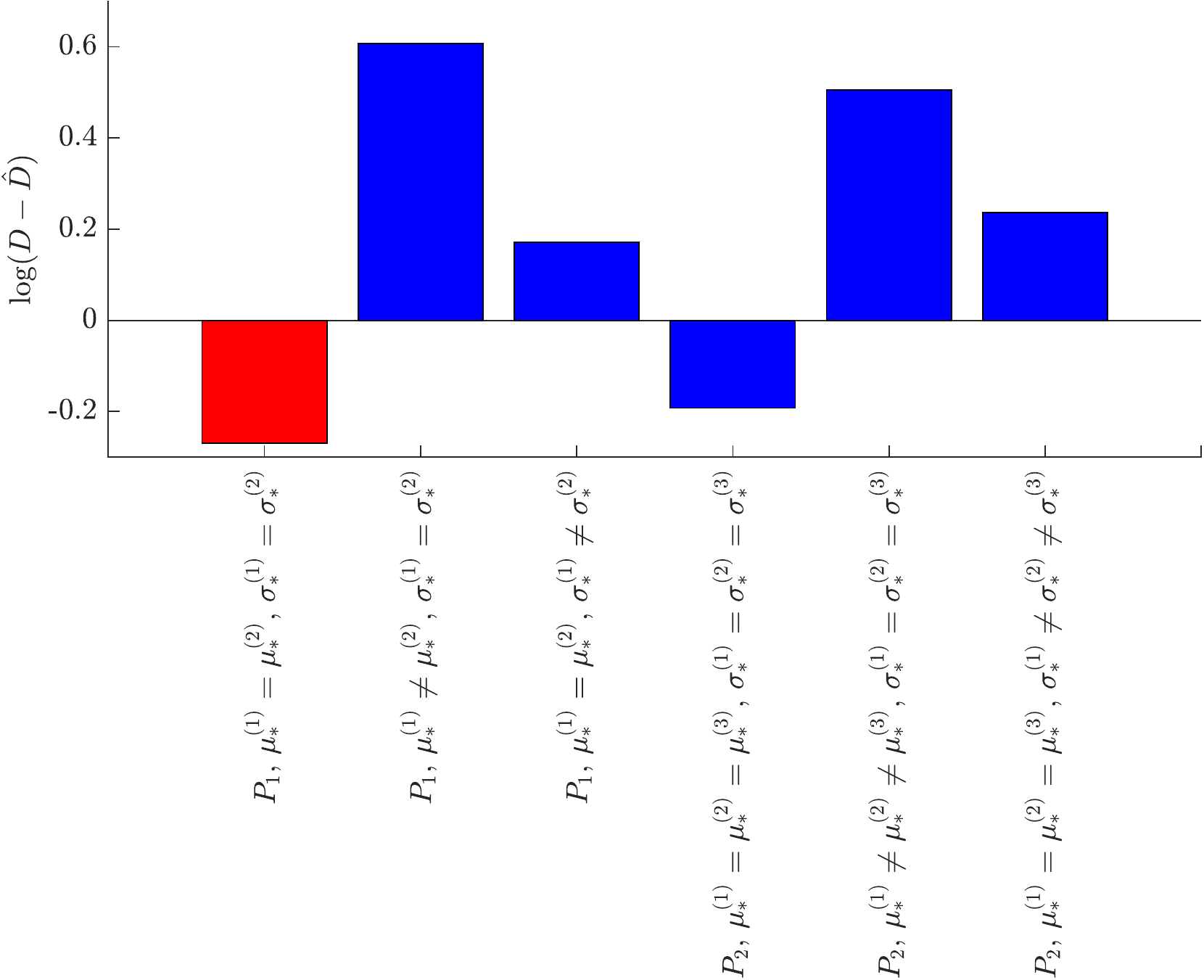} &
 \includegraphics[width=0.3\textwidth]{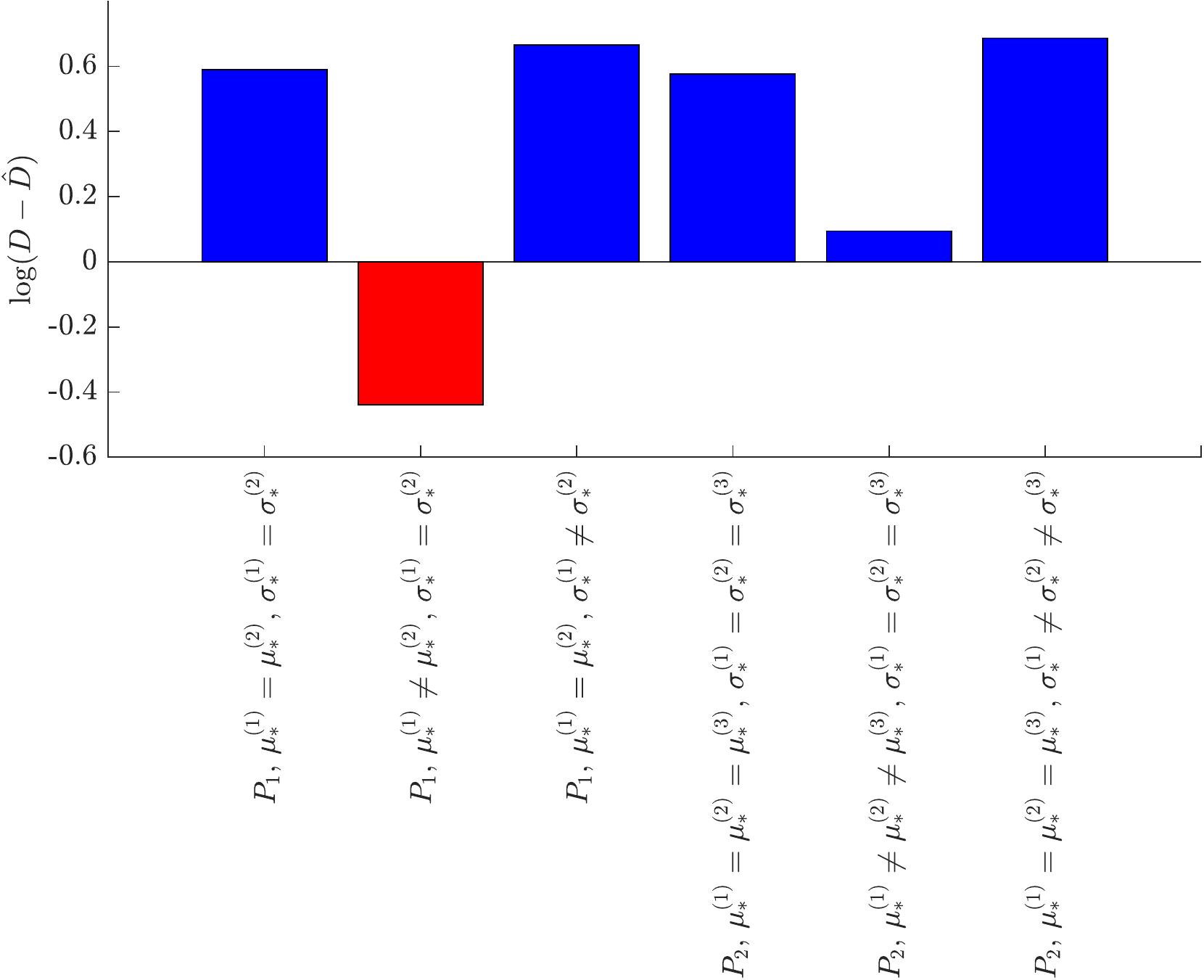} &
 \includegraphics[width=0.3\textwidth]{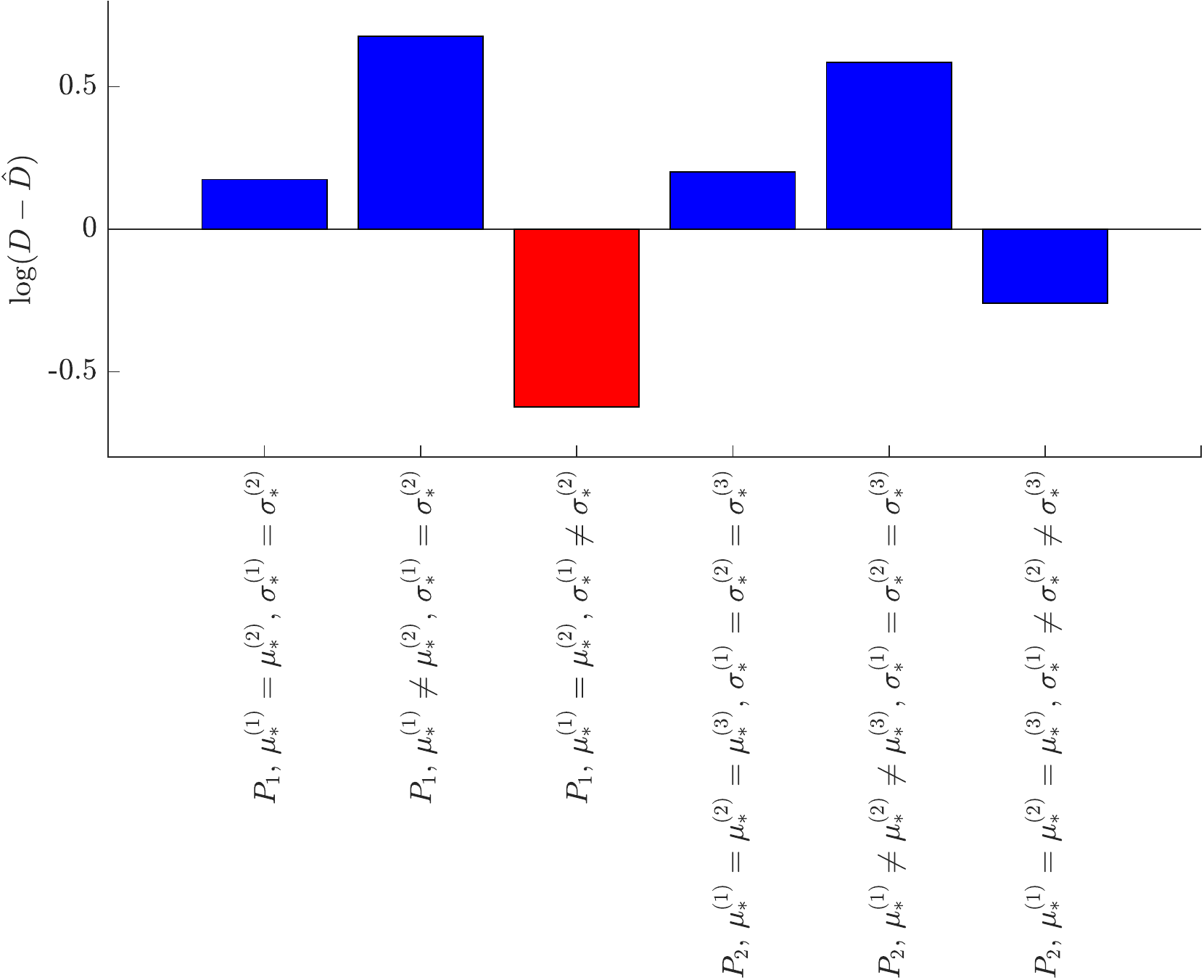} \\
 10\% &  10\% &  10\% 
 \end{array}$
 \end{center}
 \caption{Dictionary classification showing $\log \| {\mathcal D} - \hat{D}_i\|_2 $ : Top row show classification with $5\%$ noise, bottom row with $10\% $ noise, red indicates the predicted object, which is correct in all cases} \label{fig:classify_moref}
 \end{figure}

\bibliographystyle{plain}
\bibliography{ledgerlionheart}

\end{document}